\documentclass{article}

\usepackage[english]{babel}
\usepackage[utf8x]{inputenc}
\usepackage[T1]{fontenc}

\usepackage[top=3cm,bottom=2cm,left=3cm,right=3cm,marginparwidth=1.75cm]{geometry}
\usepackage{amsmath}
\usepackage{amssymb}
\usepackage{amsthm}
\usepackage{graphicx}
\usepackage{amsfonts}
\usepackage[colorinlistoftodos]{todonotes}
\usepackage[colorlinks=true, allcolors=blue]{hyperref}

\theoremstyle{plain}
\newtheorem{thm}{Theorem}[section]
\newtheorem{lem}[thm]{Lemma}
\newtheorem{prop}[thm]{Proposition}
\newtheorem*{cor}{Corollary}

\theoremstyle{definition}
\newtheorem{defn}{Definition}[section]

\newtheorem{exmp}{Example}[section]

\theoremstyle{remark}
\newtheorem*{rem}{Remark}
\newtheorem*{rems}{Remarks}

\title{Quantum Boson Algebra and Poisson Geometry of the Flag Variety}
\author{Li, Yu}
\date{\vspace{-5ex}}

\numberwithin{equation}{subsection}

\begin{document}
\maketitle

\begin{abstract}
In his work on crystal bases \cite{Kas}, Kashiwara introduced a certain degeneration of the quantized universal enveloping algebra of a semi-simple Lie algebra $\mathfrak g$, which he called a quantum boson algebra.  In this paper, we construct Kashiwara operators associated to all positive roots and use them to define a variant of Kashiwara's quantum boson algebra.  We show that a quasi-classical limit of the positive half of our variant is a Poisson algebra of the form $(P \simeq \mathbb C[\mathfrak n^{\ast}], \, \{~~,~~\}_P)$, where $\mathfrak n$ is the positive part of $\mathfrak g$ and $\{~~,~~\}_P$ is a Poisson bracket that has the same rank as, but is different from, the Kirillov-Kostant bracket $\{~~,~~\}_{KK}$ on $\mathfrak n^{\ast}$.  Furthermore, we prove that, in the special case of type $A$, any linear combination $a_1 \{~~,~~\}_P + a_2 \{~~,~~\}_{KK}$, $a_1, a_2 \in \mathbb C$, is again a Poisson bracket.  In the general case, we establish an isomorphism of $P$ and the Poisson algebra of regular functions on the open Bruhat cell in the flag variety.  In type $A$, we also construct a Casimir function on the open Bruhat cell, together with its quantization, which may be thought of as an analog of the linear function on $\mathfrak n^{\ast}$ defined by a root vector for the highest root.
\end{abstract}

\tableofcontents

\section{Introduction}

Let $\mathfrak g$ be a complex semi-simple Lie algebra and $\mathfrak g = \mathfrak n \oplus \mathfrak h \oplus \mathfrak n^-$ a triangular decomposition of $\mathfrak g$.  Using this data, one can define the quantized universal enveloping algebra $U$ (\cite{Dri}, \cite[4.3]{Jan}) deforming the usual universal enveloping algebra $\mathcal U$ of $\mathfrak g$.  In \cite{Kas}, Kashiwara defines, for each simple root $\alpha$, an operator $e'_{\alpha}$ acting on the negative half $U^-$ of $U$.  These operators allow Kashiwara to define what he calls the quantum boson algebra, which plays an important role in his theory of crystal bases (see \cite{Kas}).

In this paper, we construct Kashiwara operators $r'_{\lambda}: U^- \rightarrow U^-$ for each $\lambda \in \Phi^+$, where $\Phi^+$ is the set of positive roots.  These operators generate the positive half $C_q^+$ of a version of the so called quantum boson algebra $C_q$.  It turns out that the quasi-classical limit $P$ of a certain integral form of $C_q^+$ is a commutative algebra.  In fact, we will prove that (see Proposition \ref{PBW-qc} below)

\begin{thm} \label{first-thm}
There is a $\mathbb C$-algebra isomorphism $P \simeq \mathbb C[\overline r'_{\lambda}: \lambda \in \Phi^+]$, where $\overline r$ stands for the specialization of $r \in C_q^+$ at $q = 1$.
\end{thm}

It follows from Theorem \ref{first-thm} that one can identify $P$ with $\mathbb C[\mathfrak n^{\ast}]$.  So, by the Hayashi construction, $\mathfrak n^{\ast}$ comes equipped with a Poisson structure $\pi$ via the identification $\text{Spec} P \simeq \mathfrak n^{\ast}$.  Recall that two Poisson structures $\pi_1, \pi_2$ on a variety $X$ are called compatible if their sum is again a Poisson structure, equivalently, $\{a_1 \pi_1 + a_2 \pi_2: a_1, a_2 \in \mathbb C\}$ is a pencil of Poisson structures.  Regarding the relation between $\pi$ and the Kirillov-Kostant Poisson structure $\pi_{KK}$ on $\mathfrak n^{\ast}$, we prove that (see Theorem \ref{KK} below)

\begin{thm}
Let $\mathfrak g$ be of type $A$.  Then the Poisson structures $\pi$ and $\pi_{KK}$ on $\mathfrak n^{\ast}$ are compatible.
\end{thm}

Let $G$ be a connected algebraic group whose Lie algebra is $\mathfrak g$.  The classical $r$-matrix gives $G$ the structure of a Poisson Lie group \cite{BGY, Dri, EL, EvLu, EvLu2, GY, LM, Mi}.  The flag variety $G/B$ inherits a Poisson structure $\pi_{st}$ making the natural quotient map $G \rightarrow G/B$ Poisson.  $\pi_{st}$ is usually referred to as the standard Poisson structure.  Let $w_0$ be the longest element of the Weyl group of $(\mathfrak g, \mathfrak h)$, so $Bw_0B/B$ is the open Bruhat cell in $G/B$.  It turns out (see Theorem \ref{main-conj} below) that

\begin{thm} \label{second-thm}
There is a Poisson isomorphism $(\mathfrak n^{\ast}, \pi) \simeq (Bw_0B/B, \pi_{st})$.
\end{thm}

In \cite{EL}, Elek and Lu have proved that $\mathbb C [Bw_0B/B]$ has a natural structure of a cluster algebra, c.f. \cite{GY}.  One advantage of our approach is that the structure of symmetric Poisson CGL extension \cite{GY} on $\mathbb C [\mathfrak n^{\ast}]$ is very transparent.

In the next two results, we assume that $\mathfrak g \simeq \mathfrak {sl}_{n+1}$ and let $\alpha_1, \cdots, \alpha_n$ be the simple roots of $\mathfrak g$.  For each $\lambda \in \Phi^+$, we will construct a vector field $\overline F_{\lambda}$ on $\mathfrak n^{\ast}$ using our version of the quantum boson algebra $C_q$.  These vector fields can be used to deform the Poisson structure $\pi$.  The following result is a combination of Theorems \ref{deformed-bivect} and \ref{comp-gen} of the main text.

\begin{thm} \label{third-thm}
For any $1 \le i \le j \le n$, the bivector $\mathcal L_{\overline F_{\alpha_i + \cdots + \alpha_j}} \pi$ is Poisson, where $\mathcal L$ stands for the Lie derivative.  Moreover, the Poisson structures $\pi$ and $\mathcal L_{\overline F_{\alpha_i + \cdots + \alpha_j}} \pi$ are compatible.
\end{thm}

It follows from Theorem \ref{second-thm} and Theorem \ref{third-thm} that, in type $A$, the standard Poisson structure on the open Bruhat cell fits into a pencil of compatible Poisson structures.  It would be interesting to find an interpretation of the vector field $\overline F_{\alpha_i + \cdots + \alpha_j}$ in the context of the flag variety.

Theorem \ref{second-thm} also sheds some light on the Poisson center of the coordinate ring of $Bw_0B/B$.  To simplify notation, in type $A_n$, for $1 \le i \le j \le n$, we write $\overline r_{i,j}$ for $\overline r'_{\alpha_i + \cdots + \alpha_j}$ and $r_{i,j}$ for $r'_{\alpha_i + \cdots + \alpha_j}$.  For a partition
\begin{align}
\kappa = \{1 \le \kappa_1 < \kappa_2 < \cdots < \kappa_k = n\}
\end{align}
of $n$, we write $\overline r_{\kappa}$ for the monomial 
\begin{align}
\overline r_{1,\kappa_1} \overline r_{\kappa_1+1, \kappa_2} \cdots \overline r_{\kappa_{k-1}+1, \kappa_k}.
\end{align}
We prove the following result in Propositions \ref{Poi-cen} and \ref{quan-of-Casimir}.

\begin{thm}
(a) The element $\psi := (\sum \overline r_{\kappa})\overline r_{1,n} \in \mathbb C[\mathfrak n^{\ast}]$, where $\kappa$ runs over all partitions of $n$, is a Casimir function with respect to $\pi$.

(b) The element
\begin{align}
\Psi := \sum\limits_{\kappa \vdash n} q^{n - |\kappa|} r_{1,\kappa_1} r_{\kappa_1+1, \kappa_2} \cdots r_{\kappa_{k-1}+1, \kappa_k} r_{1,n}
\end{align}
is a central element of $C_q^+$.
\end{thm}

This paper is organized as follows.  In section \ref{basic-def} we generalize the construction of Kashiwara in the case of simple roots and associate a certain Kashiwara operator $r'_{\lambda}$ to every $\lambda \in \Phi^+$.  For this, we make use of the coalgebra structure on $U$ and Lusztig's braid group action on $U$ \cite{Lus}.  Most of the preparatory material is taken from \cite{Jan}.  The quantum boson algebra $C_q$ is defined in the same section.

In section \ref{basic-properties-boson} we construct a PBW basis for $C_q^+$ and prove certain Levendorskii-Soibelman and Leibiz type properties for Kashiwara operators.  Some of these properties will be key technical tools for the study of quasi-classical limits.

In section \ref{cl-sect} we introduce the quasi-classical limit $C_{cl}$ and study the Poisson bracket on the positive half $C_{cl}^+ \simeq P$ of $C_{cl}$ that comes from the Hayashi construction.

{\bf Acknowledgements.} The author would like to first thank V. Ginzburg for introducing the problem of studying Kashiwara operators and generously sharing the unpublished notes \cite{GG}.  The author would also like to thank him for his endless patience in listening to the author and his incomparable vision for finding interesting directions to proceed.  He also spent a tremendous amount of time teaching the author how to write a paper, line by line.  Without these the present paper would not have existed.  The author is very grateful to S. Evens for useful discussions; many statements and proofs in sections \ref{gen-rk-sect} - \ref{compatible-Pois-bra-sect} were suggested by him.  The author would also like to thank J.-H. Lu for an invitation to visit the University of Hong Kong and for many helpful comments, including an approach to proving Theorem \ref{main-conj}.  The hospitality of the University of Hong Kong is gratefully acknowledged.

\section{Basic Definitions} \label{basic-def}

\subsection{Review of Quantized Universal Enveloping Algebras} \label{review}

Let $\mathfrak g$ be a complex semi-simple Lie algebra and $\mathfrak g = \mathfrak n \oplus \mathfrak h \oplus \mathfrak n^-$ a triangular decomposition of $\mathfrak g$.  Thus $\mathfrak b := \mathfrak n \oplus \mathfrak h$ and $\mathfrak b^- := \mathfrak n^- \oplus \mathfrak h$ is a pair of opposite Borel subalgebras of $\mathfrak g$.  Associated to this data, we have a root system $(\Phi = \Phi^+ \cup \Phi^-, \Pi)$, where $\Phi$ is the set of all roots, $\Phi^+$ (resp. $\Phi^-$) is the set of positive (resp. negative) roots and $\Pi \subseteq \Phi^+$ is the set of simple roots.

Consider the field $\mathbb C(q)$ of rational functions in the variable $q$.  For $n \in \mathbb N$, we write $[n] := \frac {q^n - q^{-n}} {q - q^{-1}} \in \mathbb C(q)$.  We fix a symmetric invariant nondegenerate bilinear form $(~~,~~)$ on $\mathfrak h^{\ast}$, the dual vector space of $\mathfrak h$, such that $(\lambda,\lambda) = 2$ for each short root $\lambda$.  If $\lambda$ is a positive root, then we write $[n]_{\lambda} := \frac {q_{\lambda}^n - q_{\lambda}^{-n}} {q_{\lambda} - q_{\lambda}^{-1}}$, where $q_{\lambda} := q^{(\lambda,\lambda)/2}$.  Define the quantum factorials as $[n]! := \prod_{i=1}^n [i]$ and $[n]_{\lambda}^! := \prod_{i=1}^n [i]_{\lambda}$.  By convention, we put $[0]! := 1$ and $[0]_{\lambda}^! := 1$.  The quantum binomial coefficients are defined as follows: ${m \brack n} := \frac {[m]!} {[n]! [m-n]!}$ and ${m \brack n}_{\lambda} := \frac {[m]_{\lambda}^!} {[n]_{\lambda}^! [m-n]_{\lambda}^!}$, for $m \in \mathbb N$ and $n \in \mathbb Z_{\ge 0}$ with $n \le m$.  The subring $\mathbb C[q^{\pm 1}]_{[3]!}$, the localization of $\mathbb C[q^{\pm 1}]$ at the element $[3]!$, of $\mathbb C(q)$ will be denoted by $\mathcal A$.  It is worth pointing out that the quantum integers $[n]$ and $[n]_{\lambda}$, quantum factorials and quantum binomial coefficients defined in this paragraph are all elements of $\mathcal A$.

\begin{rem}
Our choice of the `integral form' $\mathcal A$ of $\mathbb C(q)$ is not standard.  The standard choice for $\mathcal A$ is $\mathbb C[q^{\pm 1}]$.  We make our non-standard choice only because we need to take care of the situation where there are more than one root lengths in the root system $(\Phi, \Pi)$.  See section \ref{LS-sect} for more details.
\end{rem}

Recall that the quantized universal enveloping algebra $U = U(\mathfrak g)$ is the $\mathbb C(q)$-algebra generated by $E_{\alpha}$, $F_{\alpha}$, $K_{\alpha}$ and $K_{\alpha}^{-1}$ for all $\alpha \in \Pi$, subject to the relations
\begin{gather}
K_{\alpha}K_{\alpha}^{-1} = K_{\alpha}^{-1}K_{\alpha} = 1, K_{\alpha}K_{\beta} = K_{\beta}K_{\alpha} \\
K_{\alpha}E_{\beta}K_{\alpha}^{-1} = q^{(\alpha, \beta)} E_{\beta} \\ 
K_{\alpha}F_{\beta}K_{\alpha}^{-1} = q^{-(\alpha, \beta)} F_{\beta}\\
E_{\alpha}F_{\beta} - F_{\beta}E_{\alpha} = \delta_{\alpha,\beta} \frac {K_{\alpha} - K_{\alpha}^{-1}} {q_{\alpha} - q_{\alpha}^{-1}} \\
\sum_{i=0}^{1-c_{\alpha\beta}} (-1)^i {1-c_{\alpha\beta} \brack i}_{\alpha} E_{\alpha}^{1-c_{\alpha\beta}-i}E_{\beta}E_{\alpha}^i = 0, ~ 
\sum_{i=0}^{1-c_{\alpha\beta}} (-1)^i {1-c_{\alpha\beta} \brack i}_{\alpha} F_{\alpha}^{1-c_{\alpha\beta}-i}F_{\beta}F_{\alpha}^i = 0,
\end{gather}
where $\delta_{\alpha,\beta}$ is the Kronecker delta and $c_{\alpha\beta} := 2 \frac {(\alpha,\beta)}{(\alpha,\alpha)}$.

Let $U^+$ (resp. $U^-$) be the $\mathbb C(q)$-subalgebra of $U$ generated by $E_{\alpha}$ (resp. $F_{\alpha}$), for all $\alpha \in \Pi$.  We call $U^+$ (resp. $U^-$) the positive (resp. negative) half of $U$.  Let $U^0$ be the $\mathbb C(q)$-subalgebra of $U$ generated by $K_{\alpha}^{\pm}$, for all $\alpha \in \Pi$.  One has a triangular decomposition: $U \simeq U^+ \otimes U^0 \otimes U^-$ of vector spaces over $\mathbb C(q)$, c.f. \cite[Theorem~4.21]{Jan}.

\begin{defn}
Define $U^+_{\mathcal A}$ to be the $\mathcal A$-subalgebra of $U^+$ generated by $E_{\alpha}^{(n)}$ for all $\alpha \in \Pi$ and $n \in \mathbb Z_{\ge 0}$, where $E_{\alpha}^{(n)} := \frac {E_{\alpha}^n} {[n]_{\alpha}^!}$.

Similarly, define $U^-_{\mathcal A}$ to be the $\mathcal A$-subalgebra of $U^-$ generated by $F_{\alpha}^{(n)}$ for all $\alpha \in \Pi$ and $n \in \mathbb Z_{\ge 0}$, where $F_{\alpha}^{(n)} := \frac {F_{\alpha}^n} {[n]_{\alpha}^!}$.
\end{defn}

Write $\mathbb Z \Phi$ for the root lattice, and $(\mathbb Z \Phi)^+$ (resp. $(\mathbb Z \Phi)^-$) for those $\mathbb Z$-linear combinations of simple roots all of whose coefficients are non-negative (resp. non-positive) integers.  For $\mu \in (\mathbb Z \Phi)^+$, define `root spaces' in $U$:
\begin{align}
U^+_{\mu} := \{x \in U^+: K_{\alpha} x K_{\alpha}^{-1} = q^{(\alpha,\mu)} x ~~ \text {for all} ~~ \alpha \in \Pi \}, \\
U^-_{-\mu} := \{y \in U^-: K_{\alpha} y K_{\alpha}^{-1} = q^{-(\alpha,\mu)} y ~~ \text {for all} ~~ \alpha \in \Pi \}.
\end{align}
We call $U^+_{\mu}$ (resp. $U^-_{-\mu}$) the $\mu$- (resp. ($-\mu$)-) root space of $U$.  It is clear that there are $\mathbb C(q)$-vector space direct sum decompositions
\begin{align} \label{direct-sum-decomp}
U^+ = \bigoplus \limits_{\mu \in (\mathbb Z \Phi)^+} U^+_{\mu}
~~
\text{and}
~~
U^- = \bigoplus \limits_{\mu \in (\mathbb Z \Phi)^+} U^-_{-\mu}.
\end{align}

Using the direct sum decompositions above, we make $U^-$ a $(\mathbb Z \Phi)^+$-graded algebra by putting $U^-_{-\mu}$ in degree $\mu$.

\begin{rem}
This grading of $U^-$ is not standard.  The negative half $U^-$ is positively graded.
\end{rem}

In \cite{Kas}, for each $\alpha \in \Pi$, Kashiwara defines  $\mathbb C(q)$-linear maps $e_{\alpha}', e_{\alpha}'': U^- \rightarrow U^-$, which satisfy the equation
\begin{align} \label{Kas-original-def}
E_{\alpha} y - y E_{\alpha} = \frac {K_{\alpha} e_{\alpha}''(y) - K_{\alpha}^{-1} e'_{\alpha}(y)} {q_{\alpha} - q_{\alpha}^{-1}}
\end{align}
for all $y \in U^-$.  In the literature, the maps $e_{\alpha}', e_{\alpha}''$ are usually referred to as Kashiwara operators associated to $\alpha \in \Pi$.  In what follows, we wish to present another perspective on the Kashiwara operators.

Recall that $U$ has a Hopf algebra structure where the coproduct $\Delta: U \rightarrow U \otimes U$ is defined by
\begin{align}
\Delta(E_{\alpha}) = E_{\alpha} \otimes 1 + K_{\alpha} \otimes E_{\alpha}, ~~ \Delta(F_{\alpha}) = F_{\alpha} \otimes K_{\alpha}^{-1} + 1 \otimes F_{\alpha}, ~~ \Delta(K_{\alpha}) = K_{\alpha} \otimes K_{\alpha}
\end{align}
for all $\alpha \in \Pi$.  If $\nu = \sum_{\beta \in \Pi} c_{\beta} \beta \in \mathbb Z \Phi$, we write $K_{\nu} := \prod_{\beta \in \Pi} K_{\beta}^{c_{\beta}}$.  For $\mu \in (\mathbb Z \Phi)^+$ and $y \in U^-_{-\mu}$, it is known, c.f. \cite[4.13]{Jan}, that $\Delta(y) \in \bigoplus \limits_{0 \le \nu \le \mu} U^-_{-\nu} \otimes U^-_{-(\mu-\nu)} K_{\nu}^{-1}$, where $\le$ is the standard partial order on the root lattice, i.e. $\lambda_1 \le \lambda_2$ if and only if $\lambda_2 - \lambda_1 \in (\mathbb Z \Phi)^+$.  For $\alpha \in \Pi$, since $F_{\alpha}$ generates $U^-_{-\alpha}$ over $\mathbb C(q)$, for each $\mu \in (\mathbb Z \Phi)^+$, there is a unique $\mathbb C(q)$-linear map $r'_{\alpha}: U^-_{-\mu} \rightarrow U^-_{-(\mu - \alpha)}$ such that
\begin{align}
\Delta(y) = 1 \otimes y + \sum_{\alpha \in \Pi} F_{\alpha} \otimes r'_{\alpha}(y) K_{\alpha}^{-1} + \text{other terms},
\end{align}
where `other terms' stands for summands in $U^-_{-\nu} \otimes U^-_{-(\mu-\nu)} K_{\nu}^{-1}$ for $\nu \in (\mathbb Z \Phi)^+$, $\nu \neq 0$ and $\nu \not \in \Pi$.  Using the direct sum decomposition (\ref{direct-sum-decomp}), we extend the $r'_{\alpha}$'s by $\mathbb C(q)$-linearity to $\mathbb C(q)$-linear endomorphisms of $U^-$.

\begin{defn} \label{def-sim}
For each $\alpha \in \Pi$, the $\mathbb C(q)$-linear endomorphism $r'_{\alpha}$ of $U^-$ is called the Kashiwara operator associated to $\alpha$.
\end{defn}

An important property of the Kashiwara operators is

\begin{lem} \cite[Lemma~6.17]{Jan}  \label{two-def-agree}
For each $\alpha \in \Pi$, there is a unique $\mathbb C(q)$-linear endomorphism $r_{\alpha}$ of $U^-$ satisfying the equation
\begin{align}
E_{\alpha}y - yE_{\alpha} = \frac {K_{\alpha} r_{\alpha}(y) - r'_{\alpha}(y) K_{\alpha}^{-1}} {q_{\alpha} - q_{\alpha}^{-1}}
\end{align}
for all $y \in U^-$.
\end{lem}

In view of formula (\ref{Kas-original-def}) and Lemma \ref{two-def-agree}, on each root space of $U^-$, $r'_{\alpha}$ differs from Kashiwara's original $e'_{\alpha}$ only by a power of $q$.  Hence, at the level of quasi-classical limits ($q=1$), $r'_{\alpha}$ and $e'_{\alpha}$ coincide.  The interested readers should consult \cite{Jan} for details.

We are going to construct Kashiwara operators associated to all positive, not necessarily simple, roots.  This will be done via Lusztig's braid group action \cite{Lus} on $U$.  We will follow the presentation in \cite[8.14]{Jan}.

Let $W$ be the Weyl group for $(\mathfrak g, \mathfrak h)$, and $\mathcal B$ the corresponding braid group.  $W$ acts on the root lattice $\mathbb Z \Phi$.  For $w \in W$ and $\mu \in \mathbb Z \Phi$, we write $w(\mu)$ for the action of $w$ on $\mu$.  For a simple root $\alpha$, we write $T_{\alpha}$ for the corresponding generator of $\mathcal B$.  Lusztig defines in \cite{Lus} an action of $\mathcal B$ on $U$ by $\mathbb C(q)$-algebra automorphisms by the following formulas:
\begin{align}
& T_{\alpha}(K_{\mu}) := K_{s_{\alpha}(\mu)} ~~ & \forall \alpha \in \Pi ~ \text{and} ~ \mu \in \mathbb Z \Phi; \\
& T_{\alpha}(E_{\alpha}) := -F_{\alpha}K_{\alpha}, \nonumber \\
& T_{\alpha}(F_{\alpha}) := -K_{\alpha}^{-1}E_{\alpha} ~~ & \forall \alpha \in \Pi; \\
& T_{\alpha}(E_{\beta}) := \sum_{j=0}^i (-1)^j q_{\alpha}^{-j} E_{\alpha}^{(i-j)} E_{\beta} E_{\alpha}^{(j)}, \nonumber \\
& T_{\alpha}(F_{\beta}) := \sum_{j=0}^i (-1)^j q_{\alpha}^{j} F_{\alpha}^{(j)} F_{\beta} F_{\alpha}^{(i-j)} ~~ & \forall \alpha, \beta \in \Pi ~ \text{and} ~ \alpha \neq \beta,
\end{align}
where $i := -\frac {2(\beta,\alpha)} {(\alpha,\alpha)}$.

Write $s_{\alpha}$ for the simple reflection associated to the simple root $\alpha$.  Let $w_0 \in W$ be the longest element of $W$ and fix a reduced expression
\begin{align} \label{red-exp}
w_0 = s_{\alpha_{i_1}} s_{\alpha_{i_2}} \cdots s_{\alpha_{i_N}}
\end{align}
for $w_0$.  Here $N$ is the length of $w_0$, equivalently, the number of positive roots in $\Phi$.  It is known, c.f. \cite{DP, Hum, Jan} and references therein, that such a reduced expression gives rise to an enumeration

\begin{align} \label{enumeration-of-roots}
\lambda_1 := \alpha_{i_1}, \lambda_2 := s_{\alpha_{i_1}}(\alpha_{i_2}), \cdots, \lambda_N := s_{\alpha_{i_1}} s_{\alpha_{i_2}} \cdots s_{\alpha_{i_{N-1}}}(\alpha_{i_N})
\end{align}
of all positive roots.  One also obtains a total linear order on the set of all positive roots defined by $\lambda_i \preceq \lambda_j$ if and only if $i \le j$.

Now, for each $1 \le k \le N$, define $E_{\lambda_k} (= E_k) := T_{\alpha_{i_1}} \cdots T_{\alpha_{i_{k-1}}} (E_{\alpha_{i_k}})$ and call it the root vector for the positive root $\lambda_k$.  For a vector $\vec d \in (\mathbb Z_{\ge 0})^N$, $\vec d = (d_1, \cdots, d_N)$, we define
\begin{align}
E^{\vec d} := E_N^{d_N} \cdots E_1^{d_1}.
\end{align}
Similarly, define $F_{\lambda_k} (= F_k) := T_{\alpha_{i_1}} \cdots T_{\alpha_{i_{k-1}}} (F_{\alpha_{i_k}})$ and $F^{\vec d} := F_N^{d_N} \cdots F_1^{d_1}$.  The PBW theorem \cite[Theorem~8.24]{Jan} for $U$ states

\begin{thm} \label{PBW-quan}
(a) The elements $E^{\vec d} K_{\mu} F^{\vec e}$, where $\vec d, \vec e \in (\mathbb Z_{\ge 0})^N$, $\mu \in \mathbb Z \Phi$, form a $\mathbb C(q)$-basis for the quantized universal enveloping algebra $U$.

(b) The elements $E^{\vec d}$ (resp. $F^{\vec e}$), where $\vec d \in (\mathbb Z_{\ge 0})^N$ (resp. $\vec e \in (\mathbb Z_{\ge 0})^N$), form a $\mathbb C(q)$-basis for the positive (resp. negetive) half $U^+$ (resp. $U^-$) of the quantized universal enveloping algebra $U$.
\end{thm}

There is also a version \cite{Lus2} of the PBW theorem for $U^{\pm}_{\mathcal A}$:

\begin{thm} \label{PBW-integral-version}
The elements $E^{(\vec d)}$ (resp. $F^{(\vec d)}$), where 
\begin{align}
E^{(\vec d)} := E_N^{(d_N)} \cdots E_1^{(d_1)} = (\frac 1 {[d_N]^!_{\lambda_N}} E_N^{d_N}) \cdots (\frac 1 {[d_1]^!_{\lambda_1}} E_1^{d_1})
\end{align}
(resp. $F^{(\vec d)} := F_N^{(d_N)} \cdots F_1^{(d_1)} = (\frac 1 {[d_N]^!_{\lambda_N}} F_N^{d_N}) \cdots (\frac 1 {[d_1]^!_{\lambda_1}} F_1^{d_1})$)
and $\vec d \in (\mathbb Z_{\ge 0})^N$,
form an $\mathcal A$-basis for the free $\mathcal A$-module $U^+_{\mathcal A}$ (resp. $U^-_{\mathcal A}$).
\end{thm}

\subsection{Construction of Kashiwara Operators Associcated to a (Non-simple) Positive Root} \label{Kas-op-sect}

Let $\mu \in (\mathbb Z \Phi)^+$ and $y \in U^-_{-\mu}$.  One knows that $\Delta(y) \in \bigoplus \limits_{0 \le \nu \le \mu} U^-_{-\nu} \otimes U^-_{-(\mu-\nu)} K_{\nu}^{-1}$.  Using that the $F^{\vec d}$'s with $\lambda_{\vec d} = \nu$, where $\lambda_{\vec d} := \sum_{k=1}^N d_k \lambda_k$, form a $\mathbb C(q)$-basis for $U^-_{-\nu}$, we conclude that there exist unique $\mathbb C(q)$-linear maps $r'_{\vec d}: U^-_{-\mu} \rightarrow U^-_{-(\mu - \lambda_{\vec d})}$ such that
\begin{align} \label{coproduct}
\Delta(y) = \sum_{\vec d} F^{\vec d} \otimes r'_{\vec d}(y) K_{\lambda_{\vec d}}^{-1}
\end{align}
for all $y \in U^-_{-\mu}$.  Using the direct sum decomposition (\ref{direct-sum-decomp}), we extend the maps $r'_{\vec d}$ by $\mathbb C(q)$-linearity to $\mathbb C(q)$-linear endomorphisms of $U^-$.

\begin{defn} \label{def}
(a) For each $\vec d \in (\mathbb Z_{\ge 0})^N$, the $\mathbb C(q)$-linear endomorphism $r'_{\vec d}$ of $U^-$ is called the Kashiwara operator associated to $\vec d$.

(b) If $\vec d = \vec e_k$ for some $1 \le k \le N$, i.e. $\vec d$ is the vector with a $1$ in the $k$th component and $0$'s elsewhere, we also write $r'_{\lambda_k} (=r'_k)$ for $r'_{\vec d}$ and refer to $r'_{\lambda_k} (=r'_k)$ as the Kashiwara operator associated to the (possibly non-simple) positive root $\lambda_k$.
\end{defn}

\begin{rem}
If $\lambda_k = \alpha$ for some $1 \le k \le N$ and $\alpha \in \Pi$, then $r'_{\lambda_k}$ coincides with $r'_{\alpha}$ from Definition \ref{def-sim}, because $F^{\vec e_k} = F_{\alpha}$.
\end{rem}


For each $\alpha \in \Pi$, left multiplication by $F_{\alpha}$ gives a $\mathbb C(q)$-linear endomorphism of $U^-$. The following definition is a slight modification of the one introduced by Kashiwara in \cite{Kas}.

\begin{defn} \label{def-boson-alg}
(a) The $\mathbb C(q)$-subalgebra of $\text{End}_{\mathbb C(q)} U^-$ generated by the operators of left multiplication by $F_{\alpha}$, for all $\alpha \in \Pi$, and the Kashiwara operators $r'_k$, for all $1 \le k \le N$, is called the quantum boson algebra.  This algebra will be denoted by $C_q = C_q(\mathfrak g)$.

(b) Define $C_q^+$ to be the $\mathbb C(q)$-subalgebra of $C_q$ generated by $r'_k$, for all $1 \le k \le N$.  Define $C_q^-$ to be the $\mathbb C(q)$-subalgebra of $C_q$ generated by the operators of left multiplication by $F_{\alpha}$, for all $\alpha \in \Pi$.  We will call $C_q^+$ (resp. $C_q^-$) the positive (resp. negative) half of $C_q$.
\end{defn}

We make $C_q$ a $\mathbb Z \Phi$-graded algebra in the following way.  For all $\alpha \in \Pi$, we put the operator of left multiplication by $F_{\alpha}$ in degree $-\alpha$; for all $1 \le k \le N$, we put $r'_k$ in degree $\lambda_k$.

The above definition of $C_q$ depends, a priori, on the choice of reduced expression (\ref{red-exp}) for $w_0$.  This is because the definition of the Kashiwara operators $r'_k$, $1 \le k \le N$, depends on the reduced expression.  We will show in Proposition \ref{Mat} below that $C_q^+$ is in fact independent of (\ref{red-exp}).

We now proceed to define an integral form of $C_q$.  Let $\mu \in (\mathbb Z \Phi)^+$ and $y \in U^-_{-\mu}$.  Recall that $\Delta(y) \in \bigoplus \limits_{0 \le \nu \le \mu} U^-_{-\nu} \otimes U^-_{-(\mu-\nu)} K_{\nu}^{-1}$ and that $\{F^{(\vec d)}: \lambda_{\vec d} = \nu\}$ is a $\mathbb C(q)$-basis for $U^-_{-\nu}$, by Theorem \ref{PBW-integral-version}.  Hence there exist unique $\mathbb C(q)$-linear maps $r'_{(\vec d)}: U^-_{-\mu} \rightarrow U^-_{-(\mu - \lambda_{\vec d})}$ so that
\begin{align} \label{int-Kas-op}
\Delta(y) = \sum_{\vec d} F^{(\vec d)} \otimes r'_{(\vec d)}(y) K_{\lambda_{\vec d}}^{-1}.
\end{align}

\begin{rem}
We emphasize that here we have used the $\mathbb C(q)$-basis $\{F^{(\vec d)}: \lambda_{\vec d} = \nu\}$ for $U^-_{-\nu}$, while at the beginning of this section we used the $\mathbb C(q)$-basis $\{F^{\vec d}: \lambda_{\vec d} = \nu\}$ for $U^-_{-\nu}$.  So, formula (\ref{int-Kas-op}) is not to be confused with formula (\ref{coproduct}).
\end{rem}

\begin{lem} \label{preserves-int-form}
The $\mathcal A$-module $U^-_{\mathcal A}$ is $r'_{(\vec d)}$-stable.  Thus, $r'_{(\vec d)}$ gives rise to an $\mathcal A$-linear endomorphism of $U^-_{\mathcal A}$.
\end{lem}

\begin{proof}
Let $y, y' \in U^-_{\mathcal A}$.  Write
\begin{align}
\Delta(y) = \sum_{\vec d} F^{(\vec d)} \otimes r'_{(\vec d)}(y) K_{\lambda_{\vec d}}^{-1}
~~ \text{and} ~~
\Delta(y') = \sum_{\vec e} F^{(\vec e)} \otimes r'_{(\vec e)}(y') K_{\lambda_{\vec e}}^{-1}.
\end{align}
Since $\Delta$ is an algebra morphism, we have
\begin{align}
\Delta(yy') = \Delta(y) \Delta(y') = \sum_{\vec d, \vec e} F^{(\vec d)} F^{(\vec e)} \otimes q^? r'_{(\vec d)}(y) r'_{(\vec e)}(y') K_{\lambda_{\vec d + \vec e}}^{-1},
\end{align}
where $q^?$ denotes a power of $q$ that is not going to matter for the rest of the proof.  Note that, by Theorem \ref{PBW-integral-version}, $F^{(\vec d)}, F^{(\vec e)} \in U^-_{\mathcal A}$, so $F^{(\vec d)} F^{(\vec e)} \in U^-_{\mathcal A}$, hence, again by Theorem \ref{PBW-integral-version}, $F^{(\vec d)} F^{(\vec e)}$ can be expressed as an $\mathcal A$-linear combination of $F^{(\vec f)}$ for all $\vec f \in (\mathbb Z_{\ge 0})^N$.  It follows that $r'_{(\vec f)}(yy') \in U^-_{\mathcal A}$, provided that $r'_{(\vec d)}(y), r'_{(\vec e)}(y') \in U^-_{\mathcal A}$ for all $\vec d, \vec e \in (\mathbb Z_{\ge 0})^N$.  Recall that $U^-_{\mathcal A}$ is generated by $F_{\alpha}^{(n)}$, for all $\alpha \in \Pi$ and $n \in \mathbb Z_{\ge 0}$.  So to prove the lemma it suffices to prove that $r'_{(\vec d)}(F_{\alpha}^{(n)}) \in U^-_{\mathcal A}$ for all $\alpha \in \Pi$, $n \in \mathbb Z_{\ge 0}$, and $\vec d \in (\mathbb Z_{\ge 0})^N$.

By an easy induction argument, one shows that
\begin{align}
\Delta(F_{\alpha}^{(n)}) = \sum\limits_{i=0}^n F_{\alpha}^{(i)} \otimes q_{\alpha}^{i(n-i)} F_{\alpha}^{(n-i)} K_{i\alpha}^{-1}.
\end{align}
The statement follows.
\end{proof}

\begin{cor}
For $1 \le k \le N$, $r'_k$ maps $U^-_{\mathcal A}$ to itself.  Thus $r'_k$ gives rise to an $\mathcal A$-linear endomorphism of $U^-_{\mathcal A}$.
\end{cor}

In view of the corollary above, we define the integral form of $C_q$ as

\begin{defn} \label{def-int-boson-alg}
The $\mathcal A$-subalgebra of $\text{End}_{\mathcal A} (U^-_{\mathcal A})$ generated by the operators of left multiplication by $F_{\alpha}^{(n)}$, for $\alpha \in \Pi$, $n\in \mathbb Z_{\ge 0}$, and the Kashiwara operators $r'_k$, for $1 \le k \le N$, is called the integral form of the quantum boson algebra $C_q$.  This algebra will be denoted by $C_{\mathcal A}$.
\end{defn}

\section{Basic Properties of the Quantum Boson Algebra $C_q(\mathfrak g)$} \label{basic-properties-boson}

\subsection{PBW Property and Independence of Reduced Expression}

In this section we will construct a PBW type basis for the algebra $C_q^+$ and prove that $C_q^+$ is independent of (\ref{red-exp}).  To this end, we need to make use of the Drinfeld-Killing bilinear form \cite{Dri, Jan}, which we now recall.  Let $U^{\ge 0}$ (resp. $U^{\le 0}$) denote the $\mathbb C(q)$-subalgebra of $U$ generated by $E_{\alpha}$ (resp. $F_{\alpha}$), $K_{\alpha}$ and $K_{\alpha}^{-1}$, for all $\alpha \in \Pi$.

\begin{prop} \label{Dri-Kil} \cite{Dri} \cite[Proposition~6.12]{Jan}
There exists a unique $\mathbb C(q)$-bilinear pairing $<~~,~~>: U^{\le 0} \times U^{\ge 0} \rightarrow \mathbb C(q)$ such that, for all $x, x' \in U^+$, all $y, y' \in U^-$, all $\mu, \nu \in \mathbb Z \Phi$ and all $\alpha, \beta \in \Pi$, the following equations hold
\begin{gather}
<K_{\mu}, K_{\nu}> = q^{-(\mu,\nu)}, ~ <K_{\mu}, E_{\alpha}> = 0, ~ <F_{\alpha}, K_{\mu}> = 0, \\
<F_{\alpha}, E_{\beta}> = -\delta_{\alpha\beta} (q_{\alpha} - q_{\alpha}^{-1})^{-1}, \\
<y, xx'> = <\Delta(y), x' \otimes x>, ~ <yy', x> = <y \otimes y', \Delta(x)>.
\end{gather}
\end{prop}
Here, for $x, x' \in U^{\ge 0}$ and $y, y' \in U^{\le 0}$, we write $<x \otimes x', y \otimes y'> := <x,y><x',y'>$.

\begin{prop} \cite[6.13, Proposition~8.29, 8.30]{Jan} \label{Dri-Kil-val}
(a)  For $x \in U^+$, $y \in U^-$ and $\lambda, \mu \in \mathbb Z \Phi$, we have
\begin{align}
<y K_{\lambda}, x K_{\mu}> = q^{-(\lambda, \mu)} <y, x>.
\end{align}

(b)  For $\vec d, \vec e \in (\mathbb Z_{\ge 0})^N$, $<F^{\vec d}, E^{\vec e}>$ is nonzero only when $\vec d = \vec e$, in which case
\begin{align}<F^{\vec d}, E^{\vec e}> = \prod_{i=1}^N <F_{\lambda_i}^{d_i}, E_{\lambda_i}^{d_i}>.
\end{align}

(c)  Let $\lambda$ be a positive root and $d$ a non-negative integer, then
\begin{align}
<F_{\lambda}^{d}, E_{\lambda}^{d}> = (-1)^d q_{\lambda}^{d(d-1)/2} \frac {[d]_{\lambda}^!} {(q_{\lambda} - q_{\lambda}^{-1})^d}.
\end{align}
\end{prop}

The following lemma follows readily from Proposition \ref{Dri-Kil} and coassociativity of $\Delta$.

\begin{lem} \label{mix}
For every $\vec d \in (\mathbb Z_{\ge 0})^N$, one has
\begin{align}
r'_{\vec d} = q_{\lambda_1}^{-\frac 12 d_1(d_1-1)} \cdots q_{\lambda_N}^{-\frac 12 d_N(d_N-1)}([d_1]_{\lambda_1}^! \cdots [d_N]_{\lambda_N}^!)^{-1} (r'_{N})^{d_N} \circ \cdots \circ (r'_{1})^{d_1}.
\end{align}
\end{lem}

\begin{lem} \label{fund-comp}
Let $\vec d, \vec e \in (\mathbb Z_{\ge 0})^N$ be such that $\lambda_{\vec d} = \lambda_{\vec e}$.  Then $r'_{\vec d} (F^{\vec e}) = 
\begin{cases}
1 & \mbox{if } \vec d = \vec e \\
0 & \mbox{if } \vec d \neq \vec e.
\end{cases}$
\end{lem}

\begin{proof}
By Proposition \ref{Dri-Kil}, noticing that $r'_{\vec d} (F^{\vec e}) \in \mathbb C(q)$ whenever $\lambda_{\vec d} = \lambda_{\vec e}$, we have
\begin{align}
<F^{\vec e}, E^{\vec d}> & = <\Delta(F^{\vec e}), E^{\vec d} \otimes 1> \nonumber \\
& = <\cdots + F^{\vec d} \otimes r'_{\vec d} (F^{\vec e}) K_{\lambda_{\vec d}}^{-1} + \cdots, E^{\vec d} \otimes 1> \nonumber \\
& = <F^{\vec d}, E^{\vec d}> r'_{\vec d} (F^{\vec e}),
\end{align}
where the displayed summand in the second line of the equality above is the only one whose first tensor factor is a nonzero multiple of $F^{\vec d}$.  Now the conclusion follows easily.
\end{proof}

Lemma \ref{fund-comp} will be a key ingredient in the construction of PBW type bases for $C_q^+$.


Notice that part (b) of Theorem \ref{PBW-quan} already gives a PBW basis for the negative half $C_q^-$ of $C_q$.  For a PBW basis for the positive half $C_q^+$ of $C_q$, we have

\begin{prop} \label{PBW-boson}
The elements $(r'_{\lambda_N})^{d_N} \cdots (r'_{\lambda_1})^{d_1}$, for all $d_1, \cdots, d_N \in \mathbb Z_{\ge 0}$, form a $\mathbb C(q)$-basis for the algebra $C_q^+$.
\end{prop}

\begin{proof}
The assertion that the elements $(r'_{\lambda_N})^{d_N} \cdots (r'_{\lambda_1})^{d_1}$, for all $d_1, \cdots, d_N \in \mathbb Z_{\ge 0}$, span $C_q^+$ over $\mathbb C(q)$ follows from Lemma \ref{mix} and Lemma \ref{com-rel} below.

For linear independence we recall that $U^-$ is a $\mathbb Z \Phi$-graded algebra: for each $\mu \in (\mathbb Z \Phi)^+$, the root space $U^-_{-\mu}$ has degree $\mu$.  The $\mathbb Z \Phi$-gradings on $C_q^+$ and $U^-$ are compatible in the following sense.  If $\lambda, \mu \in (\mathbb Z \Phi)^+$, $r \in C_q^+$ is homogeneous of degree $\lambda$, $y \in U^-$ is homogeneous of degree $\mu$, then, $r(y)$, the action of $r$ on $y$, is homogeneous of degree $\mu - \lambda$.

Now let $R = \sum a_{\vec d} (r'_{\lambda_N})^{d_N} \cdots (r'_{\lambda_1})^{d_1}$, where $a_{\vec d} \in \mathbb C(q)$ for each $\vec d \in (\mathbb Z_{\ge 0})^N$, be such that $R = 0$ in $C_q^+$.  Let $R = \sum_{\mu} R_{\mu}$, where $R_{\mu}$ is homogeneous of degree $\mu$, be a decomposition $R$ according to the grading by $\mathbb Z \Phi$.  For any homogeneous element $y \in U^-$, we have
\begin{align}
0 = R(y) = \sum_{\mu} R_{\mu}(y).
\end{align}
By compatibility of the gradings on $C_q^+$ and $U^-$ in the previous paragraph, the degree of $R_{\mu}(y)$ for different $\mu$'s are different.  It follows that $R_{\mu} (y) = 0$ for all $\mu \in \mathbb Z \Phi$.  Upon replacing $R$ with $R_{\mu}$, we may assume that $R$ is homogeneous of degree $\mu$.

Now we can write $R = \sum_{\lambda_{\vec d} = \mu} a_{\vec d} (r'_{\lambda_N})^{d_N} \cdots (r'_{\lambda_1})^{d_1}$.  For any $\vec e \in (\mathbb Z_{\ge 0})^N$ with $\lambda_{\vec e} = \mu$, consider the element $R(F^{\vec e})$ in $U^-$.  On the one hand, since $R$ is equal to zero in $C_q^+$, it is zero as an endomorphism of $U^-$.  Hence, $R(F^{\vec e})$ must also be $0$.  On the other hand, by Lemma \ref{mix} and Lemma \ref{fund-comp}, $R(F^{\vec e})$ is a nonzero multiple of $a_{\vec e}$.  Consequently, $a_{\vec e}$ must be $0$.  As $\vec e$ runs over all vectors in $(\mathbb Z_{\ge 0})^N$ with $\lambda_{\vec e} = \mu$, we conclude that all coefficients in the relation $R$ are $0$.
\end{proof}

We now prove that $C_q^+$ is independent of the reduced expression (\ref{red-exp}), as promised.

\begin{prop} \label{Mat}
$C_q^+$ does not depend on the choice (\ref{red-exp}) of reduced expression for $w_0$.
\end{prop}

\begin{proof}
Since every reduced expression for $w_0$ can be transformed to any other one by repeatedly applying the braid relations a finite number of times, it suffices to show that if we apply one single braid relation to the reduced expression (\ref{red-exp}), the algebra $C_q^+$ does not change.

We carefully work out the following simple case.  The proofs for all other cases are similar and will be omitted.

Suppose that in the reduced expression (\ref{red-exp}) there is a segment that reads
\begin{align} \label{segment}
w_0 = \cdots s_{\alpha} s_{\beta} s_{\alpha} \cdots,
\end{align}
where $\alpha, \beta \in \Pi$ and the rank two root system generated by $\alpha$ and $\beta$ is of type $A_2$.  Applying the braid relation $s_{\alpha} s_{\beta} s_{\alpha} = s_{\beta} s_{\alpha} s_{\beta}$ to (\ref{segment}), we get a new reduced expression for $w_0$:
\begin{align} \label{new-seg}
w_0 = \cdots s_{\beta} s_{\alpha} s_{\beta} \cdots.
\end{align}
Here, the only differences between the reduced expressions (\ref{segment}) and (\ref{new-seg}) are in the displayed portions.

Recall that (see formula (\ref{enumeration-of-roots})) (\ref{segment}) and (\ref{new-seg}) induce enumerations
\begin{align}
\lambda_1, \cdots, \lambda_N,
\end{align}
and
\begin{align}
\lambda'_1, \cdots, \lambda'_N
\end{align}
of positive roots, respectively.  Moreover, the reduced expressions (\ref{segment}) and (\ref{new-seg}) give rise to root vectors $F_{\lambda_1}, \cdots, F_{\lambda_N}$ and $F_{\lambda'_1}, \cdots, F_{\lambda'_N}$, respectively, as in section \ref{review}.

Let the integer $i$, $1 \le i \le N-2$, be the index of the first displayed $\alpha$ in (\ref{segment}), so the indices for $\beta$ and the second $\alpha$ in (\ref{segment}) are $i+1$ and $i+2$, respectively.  Thus, $\lambda_j = \lambda'_j$ and $F_{\lambda_j} = F_{\lambda'_j}$ whenever $j \neq i, i+1, i+2$.  For the indices $i, i+1, i+2$, an easy computation shows that
\begin{align}
F_{\lambda_i} = F_{\lambda'_{i+2}}, ~ F_{\lambda_{i+2}} = F_{\lambda'_i}, ~
F_{\lambda_{i+1}} = -q^{-1} F_{\lambda'_{i+1}} + (q^{-1} - q) F_{\lambda'_{i+2}}F_{\lambda'_i}.
\end{align}

For every integer $j$ satisfying $1 \le j \le N$, let $r'_{\lambda_j}$ be defined as in Definition \ref{def}.  Replacing the $\mathbb C(q)$-basis $\{F^{\vec d}: \vec d \in (\mathbb Z_{\ge 0})^N\}$ for $U^-$ that we have used in Definition \ref{def} with the $\mathbb C(q)$-basis $\{(F_{\lambda'_N})^{d_N} \cdots (F_{\lambda'_1})^{d_1}: \vec d = (d_1, \cdots, d_N) \in (\mathbb Z_{\ge 0})^N\}$, one defines Kashiwara operators $r'_{\lambda'_1}, \cdots, r'_{\lambda'_N}$ corresponding to the reduced expression (\ref{new-seg}).  More precisely, let $\mu \in (\mathbb Z \Phi)^+$.  For each $\vec d \in (\mathbb Z_{\ge 0})^N$, one has a unique $\mathbb C(q)$-linear map $s'_{\vec d}: U^-_{-\mu} \rightarrow U^-_{-(\mu - d_1 \lambda'_1 - \cdots - d_N \lambda'_N)}$ such that 
\begin{align}
\Delta(y) = \sum (F_{\lambda'_N})^{d_N} \cdots (F_{\lambda'_1})^{d_1} \otimes s'_{\vec d} (y) K^{-1}_{d_1 \lambda'_1 + \cdots + d_N \lambda'_N}
\end{align}
for all $y \in U^-_{-\mu}$.  Extending by $\mathbb C(q)$-linearity one obtains $\mathbb C(q)$-linear endomorphisms $s'_{\vec d}$ of $U^-$.  If $\vec d = \vec e_j$ for some $j$ satisfying $1 \le j \le N$, we also write $r'_{\lambda'_j}$ for $s'_{\vec e_j}$.  The analysis of root vectors above shows that $r'_{\lambda_j} = r'_{\lambda'_j}$ whenever $j \neq i, i+1, i+2$.  To compare $r'_{\lambda_i}, r'_{\lambda_{i+1}}, r'_{\lambda_{i+2}}$ with $r'_{\lambda'_i}, r'_{\lambda'_{i+1}}, r'_{\lambda'_{i+2}}$ we let $y \in U^-$ and compute
\begin{align}
\Delta(y) = & \sum F^{\vec d} \otimes r'_{\vec d} (y) K_{\lambda_{\vec d}}^{-1} \nonumber \\
= & \sum F_{\lambda'_N}^{d_N} \cdots F_{\lambda'_{i+3}}^{d_{i+3}} F_{\lambda'_i}^{d_{i+2}} (-q^{-1} F_{\lambda'_{i+1}} + (q^{-1} - q) F_{\lambda'_{i+2}}F_{\lambda'_i})^{d_{i+1}} F_{\lambda'_{i+2}}^{d_i} F_{\lambda'_{i-1}}^{d_{i-1}} \cdots F_{\lambda'_1}^{d_1} \otimes  r'_{\vec d} (y) K_{\lambda_{\vec d}}^{-1}.
\end{align}
From this we see that to compute $r'_{\lambda'_i}, r'_{\lambda'_{i+1}}, r'_{\lambda'_{i+2}}$, we only need to look at those summands in the very last expression whose first tensor factor is of the form $F_{\lambda'_i}$, $F_{\lambda'_i} F_{\lambda'_{i+2}}$, $-q^{-1} F_{\lambda'_{i+1}} + (q^{-1} - q) F_{\lambda'_{i+2}}F_{\lambda'_i}$ or $F_{\lambda'_{i+2}}$.  Call these four summands the relevant summands.

It is easy to see that the first relevant summand is
\begin{align}
F_{\lambda'_i} \otimes r'_{\lambda_{i+2}} (y) K_{\lambda_{i+2}}^{-1}.
\end{align}
It follows that $r'_{\lambda'_i} = r'_{\lambda_{i+2}}$.  Similarly, by looking at the last relevant summand, we conclude that $r'_{\lambda'_{i+2}} = r'_{\lambda_i}$.

Note that $F_{\lambda'_i} F_{\lambda'_{i+2}} = -q^{-1} F_{\lambda'_{i+1}} + q^{-1} F_{\lambda'_{i+2}} F_{\lambda'_i}$.  So the sum of the second and third relevant summands is equal to
\begin{multline}
(-q^{-1} F_{\lambda'_{i+1}} + q^{-1} F_{\lambda'_{i+2}} F_{\lambda'_i}) \otimes r'_{\vec e_i + \vec e_{i+2}} (y) K_{\lambda_i + \lambda_{i+2}}^{-1} \\
+ (-q^{-1} F_{\lambda'_{i+1}} + (q^{-1} - q) F_{\lambda'_{i+2}}F_{\lambda'_i}) \otimes r'_{\lambda_{i+1}} (y) K_{\lambda_{i+1}}^{-1} \\
= F_{\lambda'_{i+1}} \otimes (-q^{-1}) (r'_{\lambda_{i+2}} \circ r'_{\lambda_i} + r'_{\lambda_{i+1}}) (y) K_{\lambda'_{i+1}}^{-1} + \cdots,
\end{multline}
where we have used Lemma \ref{mix}, and $\cdots$ stands for summands whose first tensor factor is not a constant multiple of $F_{\lambda'_{i+1}}$.  From this computation it follows that $r'_{\lambda'_{i+1}} = -q^{-1} (r'_{\lambda_{i+2}} \circ r'_{\lambda_i} + r'_{\lambda_{i+1}})$.

Therefore, the $\mathbb C(q)$-subalgebra of $\text{End}_{\mathbb C(q)} U^-$ generated by $r'_{\lambda'_1}, \cdots, r'_{\lambda'_N}$ is a subalgebra of the $\mathbb C(q)$-subalgebra of $\text{End}_{\mathbb C(q)} U^-$ generated by $r'_{\lambda_1}, \cdots, r'_{\lambda_N}$.  By symmetry, the $\mathbb C(q)$-subalgebra of $\text{End}_{\mathbb C(q)} U^-$ generated by $r'_{\lambda_1}, \cdots, r'_{\lambda_N}$ is, in turn, a subalgebra of the $\mathbb C(q)$-subalgebra of $\text{End}_{\mathbb C(q)} U^-$ generated by $r'_{\lambda'_1}, \cdots, r'_{\lambda'_N}$.
\end{proof}

More generally, fix an arbitrary element $w \in W$.  It is known that every reduced expression for $w$ is a subword of some reduced expression for $w_0$, c.f. \cite{DP} \cite[Theorem~1.8]{Hum}.  Hence, we may assume that $w$ has a reduced expression of the form $w = s_{\alpha_{i_{j}}} s_{\alpha_{i_{j+1}}} \cdots s_{\alpha_{i_{k}}}$, for some $1 \le j \le k \le N$.  We define $C_q^+[w]$ to be the $\mathbb C(q)$-vector subspace of $C_q^+$ spanned by the monomials $(r'_{\lambda_{k}})^{d_{k}} \cdots (r'_{\lambda_{j}})^{d_{j}}$, where $d_{j}, \cdots, d_{k}$ range over all non-negative integers.  Similarly to the proof of Proposition \ref{Mat}, one shows that the vector space $C_q^+[w]$ is independent of the reduced expression $w = s_{\alpha_{i_{j}}} s_{\alpha_{i_{j+1}}} \cdots s_{\alpha_{i_{k}}}$, i.e. we have

\begin{cor}
The vector space $C_q^+[w]$ only depends on the element $w$ in the Weyl group $W$.
\end{cor}

\begin{rem}
We shall see in the next section that $C_q^+[w]$ is in fact a $\mathbb C(q)$-subalgebra of $\text{End}_{\mathbb C(q)} U^-$ (or of $C_q^+$).  In the literature of quantum groups, a $\mathbb C(q)$-vector subspace $U^+[w]$ of $U^+$ has attracted considerable interest, c.f. \cite{DP}.  It is known that $U^+[w]$ is in fact a $\mathbb C(q)$-subalgebra of $U^+$.  In this sense, our theory of quantum boson algebra parallels the classical and well-known theory of quantized universal enveloping algebra.
\end{rem}

\begin{defn}
(a) Let $C^+_{\mathcal A}$ be the $\mathcal A$-subalgebra of $\text{End}_{\mathcal A} U^-_{\mathcal A}$ generated by $r'_k$, for all $1 \le k \le N$.

(b) Let $C^-_{\mathcal A}$ be the $\mathcal A$-subalgebra of $\text{End}_{\mathcal A} U^-_{\mathcal A}$ generated by the operators of left multiplication by $F_{\alpha}^{(n)}$, for all $\alpha \in \Pi$ and $n \in \mathbb Z_{\ge 0}$.
\end{defn}

We will show (see Proposition \ref{PBW-boson-int} below) that various results in this section admit integral counterparts.  In particular, by inspecting the proof of Proposition \ref{Mat}, one obtains

\begin{prop} \label{Mat-int}
The $\mathcal A$-module $C^+_{\mathcal A}$ does not dependent on the choice (\ref{red-exp}) of reduced expression for $w_0$.
\end{prop}

As above, if $w = s_{\alpha_{i_{j}}} s_{\alpha_{i_{j+1}}} \cdots s_{\alpha_{i_{k}}}$, where $1 \le j \le k \le N$, is a subword of $w_0 = s_{\alpha_{i_1}} \cdots s_{\alpha_{i_N}}$, we define $C^+_{\mathcal A}[w]$ to be the $\mathcal A$-submodule of $C^+_{\mathcal A}$ spanned by the monomials $(r'_{\lambda_{k}})^{d_{k}} \cdots (r'_{\lambda_{j}})^{d_{j}}$, where $d_{j}, \cdots, d_{k}$ range over all non-negative integers.  We also have

\begin{cor}
The $\mathcal A$-module $C^+_{\mathcal A}[w]$ depends only on $w$.
\end{cor}

\subsection{Levendorskii-Soibelman Type Properties for Kashiwara Operators} \label{LS-sect}

In order to study quasi-classical limits, we need some information about the commutator of a pair of Kashiwara operators.  Recall the following result of Levendorskii and Soibelman.

\begin{thm} \cite{DP, LS} \label{LS}
For $1 \le i < j \le N$, one has
\begin{align}
& E_{\lambda_i} E_{\lambda_j} - q^{(\lambda_i, \lambda_j)} E_{\lambda_j} E_{\lambda_i} = \sum\limits_{\vec d \in (\mathbb Z_{\ge 0})^N} c_{\vec d} E^{\vec d} ~ \text{and} \\
& F_{\lambda_i} F_{\lambda_j} - q^{(\lambda_i, \lambda_j)} F_{\lambda_j} F_{\lambda_i} = \sum_{\vec d \in (\mathbb Z_{\ge 0})^N} \tilde c_{\vec d} F^{\vec d},
\end{align}
where $c_{\vec d}, \tilde c_{\vec d} \in \mathcal A$ and $c_{\vec d}, \tilde c_{\vec d} = 0$ whenever $d_k \neq 0$ for some $k \in [1,i] \cup [j,N]$.
\end{thm}

\begin{rem}
This version of the `Levendorskii-Soibelman straightening law' is due to De Concini and Procesi \cite{DP}.  It is claimed in {\it loc. cit.} that the coefficients $c_{\vec d}, \tilde c_{\vec d}$ in the theorem above are in $\mathbb C[q^{\pm 1}]$.  However, we believe that this is not quite the case.  One needs the localization $\mathcal A = \mathbb C[q^{\pm 1}]_{[3]!}$ of $\mathbb C[q^{\pm 1}]$ for the statement to hold, as can be seen in the examples where $\mathfrak g$ is of type $B_2$ or $G_2$.  This is one of the reasons for our choice of $\mathcal A$ as an integral form for $\mathbb C(q)$.
\end{rem}

For our purpose of studying quasi-classical limits, we need the following stronger version of Theorem \ref{LS}.  The proof will be postponed to the appendix.  Our argument is inspired by the proof of Theorem \ref{LS} in \cite{DP}.

\begin{thm} \label{LS-str}
Retain the notations in Theorem \ref{LS}.  If $c_{\vec d} \neq 0$ (resp. $\tilde c_{\vec d} \neq 0$), then $c_{\vec d}$ (resp. $\tilde c_{\vec d}$) is divisible by $(1-q)^{(\sum_{i=1}^N d_i) - 1}$ in the unique factorization domain $\mathcal A$.
\end{thm}

For the rest of this paper, we will only use the statement in Theorem \ref{LS-str} involving the $E$'s.  So, whenever two positive roots $\lambda_i$ and $\lambda_j$ ($1 \le i < j \le N$) are given, we write $c_{\vec d}$ for the coefficient for $E^{\vec d}$ in $E_{\lambda_i} E_{\lambda_j} - q^{(\lambda_i, \lambda_j)} E_{\lambda_j} E_{\lambda_i} = \sum\limits_{\vec d \in (\mathbb Z_{\ge 0})^N} c_{\vec d} E^{\vec d}$ as in Theorems \ref{LS} and \ref{LS-str}.  To simplify notation, for $1 \le i < j \le N$, we write $r'_{i,j}$ for $r'_{\vec e_i + \vec e_j}$.

The following result will play a crucial role in the study of Poisson geometry of the quasi-classical limit of $C^+_{\mathcal A}$.  The proof is based on the fact that the coproduct $\Delta$ on $U$ is coassociative \cite[Proposition~4.11]{Jan}. 

\begin{lem} \label{com-rel}
For $1 \le i < j \le N$, the following formula holds for the commutator $[r'_{i},r'_{j}]$ in $\text{End}_{\mathbb C(q)} U^-$ of the Kashiwara operators $r'_{i}$ and $r'_{j}$:
\begin{align} \label{com-rel-Kas}
[r'_{i}, r'_{j}] = (q^{(\lambda_i, \lambda_j)}-1) r'_{i,j} + <F_{\lambda_i}, E_{\lambda_i}>^{-1} <F_{\lambda_j}, E_{\lambda_j}>^{-1} \sum\limits_{\lambda_{\vec d} = \lambda_i + \lambda_j} c_{\vec d} <F^{\vec d}, E^{\vec d}> r'_{\vec d}.
\end{align}
\end{lem}

\begin{proof}
Let $y \in U^-$.  Our strategy is to use coassociativity of $\Delta$, namely the equality $(1 \otimes \Delta) \circ \Delta (y) = (\Delta \otimes 1) \circ \Delta (y)$.  We isolate terms of the form
\begin{align}
F_{\lambda_j} \otimes F_{\lambda_i} K_{\lambda_j}^{-1} \otimes X K_{\lambda_i + \lambda_j}^{-1}
\end{align}
on the two sides, where $X$ is an element of $U^-$ that we would like to compute.

For the left-hand side we have $\Delta (y) = \cdots + F_{\lambda_j} \otimes r'_j (y) K_{\lambda_j}^{-1} + \cdots$, where the displayed summand is the only one whose first tensor factor is a constant multiple of $F_{\lambda_j}$.  So we have
\begin{align}
(1 \otimes \Delta) \circ \Delta (y) = & \cdots + F_{\lambda_j} \otimes \Delta (r'_j (y) K_{\lambda_j}^{-1}) + \cdots \nonumber \\
= & \cdots + F_{\lambda_j} \otimes F_{\lambda_i} K_{\lambda_j}^{-1} \otimes r'_i \circ r'_j (y) K_{\lambda_i + \lambda_j}^{-1} + \cdots.
\end{align}
It follows that the relevant term $X$ is equal to $r'_i \circ r'_j (y)$.

For the right-hand side we compute as follows.  First note that if
\begin{align}
\Delta (F_{\lambda_j} F_{\lambda_i}) = \cdots + F_{\lambda_j} \otimes a F_{\lambda_i} K_{\lambda_j}^{-1} + \cdots
\end{align}
for some $a \in \mathbb C(q)$, where the displayed summand is the only one which is a constant multiple of $F_{\lambda_j} \otimes F_{\lambda_i} K_{\lambda_j}^{-1}$, then we have
\begin{align}
a <F_{\lambda_i}, E_{\lambda_i}> <F_{\lambda_j}, E_{\lambda_j}> = & <\Delta (F_{\lambda_j} F_{\lambda_i}), E_{\lambda_j} \otimes E_{\lambda_i}> \nonumber \\
= & <F_{\lambda_j} F_{\lambda_i}, E_{\lambda_i} E_{\lambda_j}> \nonumber \\
= & <F_{\lambda_j} F_{\lambda_i}, q^{(\lambda_i, \lambda_j)} E_{\lambda_j} E_{\lambda_i} + \sum c_{\vec d} E^{\vec d}> \nonumber \\
= & q^{(\lambda_i, \lambda_j)} <F_{\lambda_i}, E_{\lambda_i}> <F_{\lambda_j}, E_{\lambda_j}>,
\end{align}
where the second equality follows from Proposition \ref{Dri-Kil} and the last equality follows from Proposition \ref{Dri-Kil-val}.  This implies that $a = q^{(\lambda_i, \lambda_j)}$.  Similarly, let $\vec d \in (\mathbb Z_{\ge 0})^N$ be such that $d_k = 0$ whenever $k \le i$ or $k \ge j$.  If
\begin{align}
\Delta (F^{\vec d}) = \cdots + F_{\lambda_j} \otimes b F_{\lambda_i} K_{\lambda_j}^{-1} + \cdots
\end{align}
for some $b \in \mathbb C(q)$, where the displayed summand is the only one which is a constant multiple of $F_{\lambda_j} \otimes F_{\lambda_i} K_{\lambda_j}^{-1}$, then we have
\begin{align}
b <F_{\lambda_i}, E_{\lambda_i}> <F_{\lambda_j}, E_{\lambda_j}> = & <\Delta (F^{\vec d}), E_{\lambda_j} \otimes E_{\lambda_i}> \nonumber \\
= & <F^{\vec d}, E_{\lambda_i} E_{\lambda_j}> \nonumber \\
= & <F_{\vec d}, q^{(\lambda_i, \lambda_j)} E_{\lambda_j} E_{\lambda_i} + \sum c_{\vec e} E^{\vec e}> \nonumber \\
= & c_{\vec d} <F^{\vec d}, E^{\vec d}>.
\end{align}
This implies that $b = <F_{\lambda_i}, E_{\lambda_i}>^{-1} <F_{\lambda_j}, E_{\lambda_j}>^{-1} c_{\vec d} <F^{\vec d}, E^{\vec d}>$.

Let $z \in U^-$.  By Proposition \ref{Dri-Kil-val}, a nonzero multiple of $F_{\lambda_j} \otimes F_{\lambda_i} K_{\lambda_j}^{-1}$ occurs as a summand in $\Delta (z)$ if and only if $<\Delta (z), E_{\lambda_j} \otimes E_{\lambda_i}>$ is nonzero.  From the equation
\begin{align}
<\Delta (z), E_{\lambda_j} \otimes E_{\lambda_i}> = <z, E_{\lambda_i} E_{\lambda_j}> = <z, q^{(\lambda_i, \lambda_j)} E_{\lambda_j} E_{\lambda_i} + \sum c_{\vec d} E^{\vec d}>,
\end{align}
we see that to determine the summand in $\Delta (z)$ which is a multiple of $F_{\lambda_j} \otimes F_{\lambda_i} K_{\lambda_j}^{-1}$, we only need to concentrate on the summands in $z$ which are constant multiples of $F_{\lambda_j} F_{\lambda_i}$ or $F^{\vec d}$ for some $\vec d \in (\mathbb Z_{\ge 0})^N$ with $d_k = 0$ whenever $k \le i$ or $k \ge j$.  Hence, we get
\begin{align}
& (\Delta \otimes 1) \circ \Delta (y) \nonumber \\
= & (\Delta \otimes 1) (\cdots + F_{\lambda_j} F_{\lambda_i} \otimes r'_{i,j} (y) K_{\lambda_i + \lambda_j}^{-1} + \sum F^{\vec d} \otimes r'_{\vec d} (y) K_{\lambda_{\vec d}}^{-1} + \cdots) \nonumber \\
= & \cdots + \Delta (F_{\lambda_j} F_{\lambda_i}) \otimes r'_{i,j} (y) K_{\lambda_i + \lambda_j}^{-1} + \sum \Delta (F^{\vec d}) \otimes r'_{\vec d} (y) K_{\lambda_{\vec d}}^{-1} + \cdots \nonumber \\
= & 
\cdots + F_{\lambda_j} \otimes F_{\lambda_i} K_{\lambda_j}^{-1} \otimes (q^{(\lambda_i, \lambda_j)} r'_{i,j} (y) \nonumber \\
& + \sum <F_{\lambda_i}, E_{\lambda_i}>^{-1} <F_{\lambda_j}, E_{\lambda_j}>^{-1} c_{\vec d} <F^{\vec d}, E^{\vec d}> r'_{\vec d} (y)) K_{\lambda_i + \lambda_j}^{-1} + \cdots,
\end{align}
where in the last step we have used our computation in the previous paragraph.  Therefore, the relevant term $X$ on the right hand side is given by:
\begin{align}
X = q^{(\lambda_i, \lambda_j)} r'_{i,j} (y) + \sum <F_{\lambda_i}, E_{\lambda_i}>^{-1} <F_{\lambda_j}, E_{\lambda_j}>^{-1} c_{\vec d} <F^{\vec d}, E^{\vec d}> r'_{\vec d} (y).
\end{align}

Comparing the two sides, we get
\begin{align}
r'_i \circ r'_j = q^{(\lambda_i, \lambda_j)} r'_{i,j} + \sum <F_{\lambda_i}, E_{\lambda_i}>^{-1} <F_{\lambda_j}, E_{\lambda_j}>^{-1} c_{\vec d} <F^{\vec d}, E^{\vec d}> r'_{\vec d}.
\end{align}
By Lemma \ref{mix}, we have $r'_{i,j} = r'_j \circ r'_i$.  So the last equality is equivalent to
\begin{align}
r'_i \circ r'_j - r'_j \circ r'_i = (q^{(\lambda_i, \lambda_j)} - 1) r'_{i,j} + \sum <F_{\lambda_i}, E_{\lambda_i}>^{-1} <F_{\lambda_j}, E_{\lambda_j}>^{-1} c_{\vec d} <F^{\vec d}, E^{\vec d}> r'_{\vec d},
\end{align}
proving the desired formula.
\end{proof}

\begin{cor}
For each $w \in W$, the $\mathbb C(q)$-vector subspace $C_q^+[w]$ of $C_q^+$ is a $\mathbb C(q)$-subalgebra of $C_q^+$.
\end{cor}

\begin{proof}
This follows easily from the definition of $C_q^+[w]$, Lemma \ref{mix}, Theorem \ref{LS} and Lemma \ref{com-rel}.
\end{proof}

For $r, r' \in C_q^+$, write $[r,r']$ for the commutator $r \circ r' - r' \circ r$ of $r$ and $r'$ in $C_q^+$.  For $1 \le i < j \le N$, by Proposition \ref{PBW-boson}, we can write
\begin{align}
[r'_i,r'_j] = \sum h_{\vec d} (r'_N)^{d_N} \cdots (r'_1)^{d_1}, ~ h_{\vec d} \in \mathbb C(q).
\end{align}

\begin{prop} \label{ast}
Except for the case where $\vec d = \vec e_i + \vec e_j$, $h_{\vec d} = 0$ whenever $d_k \neq 0$ for some $k \in [1,i] \cup [j,N]$.  Moreover, all $h_{\vec d}$'s are elements of $\mathcal A$ and are divisible by $1 - q$.
\end{prop}

\begin{proof}
We rewrite the Levendorskii-Soibelman straightening law using the $\mathbb C(q)$-basis $\{E^{(\vec d)}: \vec d \in (\mathbb Z_{\ge 0})^N \}$ of $U^+$ as follows:
\begin{align}
E_{\lambda_i} E_{\lambda_j} - q^{(\lambda_i, \lambda_j)} E_{\lambda_j} E_{\lambda_i} = \sum c_{(\vec d)} E^{(\vec d)}.
\end{align}
According to Theorem \ref{LS}, each $c_{(\vec d)}$ is in $\mathcal A$ and is, moreover, divisible by $[d_1]_{\lambda_i}^! \cdots [d_N]_{\lambda_N}^!$.  In fact, it is easy to see that $c_{(\vec d)} = [d_1]_{\lambda_i}^! \cdots [d_N]_{\lambda_N}^! c_{\vec d}$.  It follows from the proof of Lemma \ref{com-rel} that
\begin{align} \label{intermediate-formula}
[r'_{i}, r'_{j}] = (q^{(\lambda_i, \lambda_j)}-1) r'_j \circ r'_i + <F_{\lambda_i}, E_{\lambda_i}>^{-1} <F_{\lambda_j}, E_{\lambda_j}>^{-1} \sum\limits_{\lambda_{\vec d} = \lambda_i + \lambda_j} c_{(\vec d)} <F^{(\vec d)}, E^{(\vec d)}> r'_{(\vec d)}.
\end{align}
By Proposition \ref{Dri-Kil-val}, we know that
\begin{align}
& <F^{(\vec d)}, E^{(\vec d)}> = \nonumber \\
& (-1)^{d_1 + \cdots + d_N} q_{\lambda_1}^{d_1(d_1-1)/2} \cdots q_{\lambda_N}^{d_N(d_N-1)/2} (q_{\lambda_1} - q_{\lambda_1}^{-1})^{-d_1} \cdots (q_{\lambda_N} - q_{\lambda_N}^{-1})^{-d_N} ([d_1]_{\lambda_1}^! \cdots [d_N]_{\lambda_N}^!)^{-1}.
\end{align}
Hence, for the coefficient $<F_{\lambda_i}, E_{\lambda_i}>^{-1} <F_{\lambda_j}, E_{\lambda_j}>^{-1} c_{(\vec d)} <F^{(\vec d)}, E^{(\vec d)}>$ of $r'_{(\vec d)}$ in (\ref{intermediate-formula}), we have
\begin{align}
& <F_{\lambda_i}, E_{\lambda_i}>^{-1} <F_{\lambda_j}, E_{\lambda_j}>^{-1} c_{(\vec d)} <F^{(\vec d)}, E^{(\vec d)}> = \nonumber \\
& (-1)^{d_1 + \cdots + d_N} q_{\lambda_1}^{d_1(d_1-1)/2} \cdots q_{\lambda_N}^{d_N(d_N-1)/2} (q_{\lambda_i} - q_{\lambda_i}^{-1}) (q_{\lambda_j} - q_{\lambda_j}^{-1}) (q_{\lambda_1} - q_{\lambda_1}^{-1})^{-d_1} \cdots (q_{\lambda_N} - q_{\lambda_N}^{-1})^{-d_N} c_{\vec d}.
\end{align}
By Theorem \ref{LS-str}, $c_{\vec d}$ is an element of $\mathcal A$ divisible by $(1-q)^{d_1 + \cdots + d_N - 1}$.  It follows from our last formula that the coefficient of $r'_{(\vec d)}$ in (\ref{intermediate-formula}) is an element of $\mathcal A$; moreover this element is divisible by $1 - q$.  Lemma \ref{mix} clearly implies that
\begin{align}
r'_{(\vec d)} = q_{\lambda_1}^{-\frac 12 d_1(d_1-1)} \cdots q_{\lambda_N}^{-\frac 12 d_N(d_N-1)} (r'_{N})^{d_N} \circ \cdots \circ (r'_{1})^{d_1}.
\end{align}
The result follows by plugging this into formula (\ref{intermediate-formula}).
\end{proof}

It follows from Proposition \ref{ast} that operators of the form $(r'_N)^{d_N} \circ \cdots \circ (r'_1)^{d_1}$, $d_1, \cdots, d_N \in \mathbb Z_{\ge 0}$, span $C^+_{\mathcal A}$ as an $\mathcal A$-submodule of $C_q^+$.  Since these operators are linearly independent over $\mathbb C(q)$ by Proposition \ref{PBW-boson}, they are also linearly independent over $\mathcal A$.  Thus we have proved the following PBW theorem for $C^+_{\mathcal A}$.

\begin{prop} \label{PBW-boson-int}
The elements $(r'_N)^{d_N} \circ \cdots \circ (r'_1)^{d_1}$, for all $d_1, \cdots, d_N \in \mathbb Z_{\ge 0}$, form an $\mathcal A$-basis for the integral form $C^+_{\mathcal A}$ of $C_q^+$.
\end{prop}

Another consequence of Proposition \ref{ast} is

\begin{cor}
For each $w \in W$, the $\mathcal A$-submodule $C^+_{\mathcal A}[w]$ of $C^+_{\mathcal A}$ is an $\mathcal A$-subalgebra of $C^+_{\mathcal A}$.
\end{cor}

\begin{proof}
This follows easily from Lemma \ref{com-rel}, Proposition \ref{ast} and the definition of $C^+_{\mathcal A}[w]$.
\end{proof}

\begin{cor}
The $\mathbb C$-algebra $C^+_{\mathcal A}/(1-q) C^+_{\mathcal A}$ is commutative.
\end{cor}

\begin{proof}
Note that $C^+_{\mathcal A}/(1-q) C^+_{\mathcal A}$ is generated as a $\mathbb C$-algebra by $\overline r'_i$, for all $1 \le i \le N$, where $\overline r'_i$ stands for the image of $r'_i$ in $C^+_{\mathcal A}/(1-q) C^+_{\mathcal A}$.  By Proposition \ref{ast}, each coefficient in the commutator $[r'_{i},r'_{j}]$ is divisible by $1-q$.  It follows that in $C^+_{\mathcal A}/(1-q) C^+_{\mathcal A}$ we have $[\overline r'_i, \overline r'_j] = 0$.
\end{proof}

\subsection{Leibniz Rule for Kashiwara Operators} \label{Leib-sect}

The primary goal of this section is to study commutation relations between Kashiwara operators and the operators of left multiplication by a root vector in the negative half of the quantized universal enveloping algebra.  Material in this section will be used in section \ref{alg-asp-sect} to see that, for all $1 \le i \le N$, commutating with the root vector $F_i$ gives rise to a derivation on the quasi-classical limit of the integral form of the positive half of the quantum boson algebra.

We need some notations.  For $\vec d, \vec e \in (\mathbb Z_{\ge 0})^N$, the product $F^{(\vec d)} F^{(\vec e)}$ is an element of $U^-_{\mathcal A}$, hence it can be written as an $\mathcal A$-linear combination
\begin{align}
F^{(\vec d)} F^{(\vec e)} = \sum\limits_{\vec f} n_{\vec f}^{\vec d, \vec e} F^{(\vec f)}, ~ n_{\vec f}^{\vec d, \vec e} \in \mathcal A.
\end{align}
If $\vec f = \vec e_i$ for some $1 \le i \le N$, we put $n_{\lambda_i}^{\vec d, \vec e} := n_{\vec f}^{\vec d, \vec e}$, for simplicity.



Recall that $U^-$ is a $(\mathbb Z \Phi)^+$-graded algebra.  The $\mathcal A$-subalgebra $U^-_{\mathcal A}$ of $U^-$ inherits a $(\mathbb Z \Phi)^+$-grading whose degree $\mu$ component is $U^-_{\mathcal A} \cap U^-_{-\mu}$.

\begin{lem} \label{Leib} {\bf (Leibniz rule for Kashiwara operators)}
Let $\lambda$ be a positive root and $y, y' \in U^-_{\mathcal A}$ homogeneous elements.  Then, we have
\begin{align} \label{Leib-formula}
r'_{\lambda} (y y') = \sum q^{(\lambda_{\vec d}, \text{wt} (y') - \lambda_{\vec e})} n_{\lambda}^{\vec d, \vec e} r'_{(\vec d)} (y) r'_{(\vec e)} (y'),
\end{align}
where $\text{wt} (y') \in (\mathbb Z \Phi)^+$ stands for the degree of the homogeneous element $y' \in U^-_{\mathcal A}$.
\end{lem}

\begin{proof}
One equates the summand of the form $F_{\lambda} \otimes X K_{\lambda}^{-1}$ on the two sides of the equation $\Delta(y y') = \Delta(y) \Delta(y')$.  The statement follows easily from this.
\end{proof}

Let $\vec d \in (\mathbb Z_{\ge 0})^N$ and $1 \le j \le N$, by Theorem \ref{PBW-integral-version} and Lemma \ref{preserves-int-form}, we have
\begin{align} \label{unknown-coeff}
r'_{(\vec d)} (F_j) = \sum \limits_{\vec f} c^{\lambda_j}_{\vec d, \vec f} F^{(\vec f)},
\end{align}
where $c^{\lambda_j}_{\vec d, \vec f} \in \mathcal A$.

\begin{lem} \label{Leib-lem}
For all $1 \le i,j \le N$, $\mu \in (\mathbb Z \Phi)^+$ and $y \in U^-_{\mathcal A} \cap U^-_{-\mu}$, we have
\begin{flalign} \label{vf}
r'_i (F_j y) - F_j r'_i (y) = \sum \limits_{\vec d \neq 0} q^{(\lambda_{\vec d}, \mu - \lambda_{\vec e})} \, n_{\lambda_i}^{\vec d, \vec e} \, c^{\lambda_j}_{\vec d, \vec f} \, F^{(\vec f)} r'_{(\vec e)} (y).
\end{flalign}
\end{lem}

\begin{proof}
In Lemma \ref{Leib} take $\lambda = \lambda_i$ and $y = F_j$.  Note that the summand corresponding to $\vec d = 0$ on the right-hand side of formula (\ref{Leib-formula}) is nothing but $F_j r'_i (y')$.  Moving this summand to the left-hand side, we get
\begin{flalign}
r'_i (F_j y') - F_j r'_i (y') = \sum \limits_{\vec d \neq 0, \lambda_{\vec d} + \lambda_{\vec e} = \lambda_i} q^{(\lambda_{\vec d}, \text{wt} y' - \lambda_{\vec e})} n_{\lambda_i}^{\vec d, \vec e} r'_{(\vec d)} (F_j) r'_{(\vec e)} (y').
\end{flalign}
The result follows by replacing $y'$ with $y$ and plugging in formula (\ref{unknown-coeff}).
\end{proof}

Introduce the notation $|\vec d| := \sum_{i=1}^N d_i$, for $\vec d \in (\mathbb Z_{\ge 0})^N$.  The following lemma is helpful to simplify formula (\ref{vf}).

\begin{lem} \label{a-coeff}
For all $1 \le j \le N$ and $\vec d, \vec f \in (\mathbb Z_{\ge 0})^N$, the element $c_{\vec d, \vec f}^{\lambda_j}$ is divisible by $(1-q)^{|\vec d| + |\vec f| - 1}$ in $\mathcal A$.
\end{lem}

The proof of the lemma is tedious and will be omitted.

Taking quasi-classical limit of formula (\ref{vf}), we deduce that


\begin{cor}
For $1 \le i,j \le N$, the image of the commutator $[r'_i, F_j]$ in $\text{End}_{\mathcal A} (U^-_{\mathcal A})/(1-q)\text{End}_{\mathcal A} (U^-_{\mathcal A})$ belongs to $C^+_{\mathcal A}/ (1-q)C^+_{\mathcal A}$.
\end{cor}

\begin{proof}
By Lemma \ref{a-coeff}, $(1 - q)^{|\vec d| + |\vec f| - 1}$ divides $c_{\vec d, \vec f}^{\lambda_j}$.  Suppose $\vec f \neq 0$.  Since $\vec d \neq 0$ for all summands on the right-hand side of formula (\ref{vf}), the coefficient $q^{(\lambda_{\vec d}, \mu - \lambda_{\vec e})} n_{\lambda_i}^{\vec d, \vec e} c^{\lambda_j}_{\vec d, \vec f}$ of $F^{(\vec f)} r'_{(\vec e)} (y)$ must be divisible by $(1 - q)^{|\vec d| + |\vec f| - 1}$.  Since $|\vec d| + |\vec f| - 1 \ge 1 + 1 - 1 = 1$, we see that $q^{(\lambda_{\vec d}, \mu - \lambda_{\vec e})} n_{\lambda_i}^{\vec d, \vec e} c^{\lambda_j}_{\vec d, \vec f}$ is divisible by $(1 - q)$.  Therefore, after quotienting out by the (2-sided) ideal generated by $1-q$ in $\text{End}_{\mathcal A} U^-_{\mathcal A}$, the only summands in the image of the commutator $[r'_i, F_j]$ in $\text{End}_{\mathcal A} (U^-_{\mathcal A})/(1-q)\text{End}_{\mathcal A} (U^-_{\mathcal A})$ that can possibly survive are those with $\vec f = 0$.
\end{proof}

In words, for $1 \le j \le N$, taking commutator with the operator of left multiplication by $F_j$ gives rise to a well-defined operator on $C^+_{\mathcal A}/ (1-q)C^+_{\mathcal A}$.  These operators will be studied in detail in section \ref{compatible-Pois-bra-sect} below.

\section{Quasi-classical Limit of the Quantum Boson Algebra} \label{cl-sect}

In this section we study algebraic and Poisson-geometric properties of the quasi-classical limit $C_{cl}$ of the quantum boson algebra $C_q$, using facts we collected in previous sections.

Recall our definition of the quantum boson algebra $C_q$, its positive half $C_q^+$ (resp. negative half $C_q^-$) and their integral forms $C_{\mathcal A}$ and $C^+_{\mathcal A}$ (resp. $C^-_{\mathcal A}$).

\begin{defn}
Let $C_{cl} := C_{\mathcal A}/(1-q)C_{\mathcal A}$, $C^+_{cl} := C^+_{\mathcal A}/(1-q)C^+_{\mathcal A}$ and $C^-_{cl} := C^-_{\mathcal A}/(1-q)C^-_{\mathcal A}$.  $C_{cl}$ (resp. $C^+_{cl}$ and $C^-_{cl}$) is called the quasi-classical limit of $C_q$ (resp. $C^+_q$ and $C^-_q$).  The quasi-classical limit $C^+_{cl}$ of $C_q^+$ will also be denoted by $P$.
\end{defn}

\begin{rem}
Here, by definition we have $C^-_{cl} = \mathcal U^-$, a non-commutative algebra.
\end{rem}

\subsection{Algebraic Aspects} \label{alg-asp-sect}

In this section we will construct a PBW basis for $C_{cl}$ and give a presentation of $C_{cl}$ by generators and relations.  We will also prove that $C_{cl}$ is a simple $\mathbb C$-algebra whose center is $\mathbb C$.

Let $\overline F_i$ be the image of $F_i$, $1 \le i \le N$, under the natural quotient map $C_{\mathcal A} \rightarrow C_{cl}$.  Let $\overline F^{(\vec d)}$, $\overline r'_i$ and $\overline r'_{(\vec d)}$, for $\vec d \in (\mathbb Z_{\ge 0})^N$, be defined similarly.

We have the following analog of Proposition \ref{PBW-boson}.

\begin{prop} \label{PBW-qc}
(a) The elements $\overline F^{(\vec d)} (\overline r'_1)^{e_1} \cdots (\overline r'_N)^{e_N}$, for all $\vec d, \vec e \in (\mathbb Z_{\ge 0})^N$, form a $\mathbb C$-basis for $C_{cl}$.  Similarly, the elements $(\overline r'_1)^{e_1} \cdots (\overline r'_N)^{e_N} \overline F^{(\vec d)}$, for all $\vec d, \vec e \in (\mathbb Z_{\ge 0})^N$, form a $\mathbb C$-basis for $C_{cl}$.

(b) The elements $(\overline r'_1)^{d_1} \cdots (\overline r'_N)^{d_N}$, for all $\vec d \in (\mathbb Z_{\ge 0})^N$, form a $\mathbb C$-basis for $P$.
\end{prop}

\begin{proof}
(a) We prove that the elements $\overline F^{(\vec d)} (\overline r'_1)^{e_1} \cdots (\overline r'_N)^{e_N}$, $\vec d, \vec e \in (\mathbb Z_{\ge 0})^N$, form a $\mathbb C$-basis for $C_{cl}$.  The proof for the other statement is similar.

It is clear that $\overline F^{(\vec d)} (\overline r'_1)^{e_1} \cdots (\overline r'_N)^{e_N}$, $\vec d, \vec e \in (\mathbb Z_{\ge 0})^N$, span $C_{cl}$ over $\mathbb C$.  We must prove that they are linearly independent over $\mathbb C$.

Recall that $C_q$ a $\mathbb Z \Phi$-graded algebra.  This grading induces a $\mathbb Z \Phi$-grading on $C_{cl}$.  Namely, for each $1 \le i \le N$, $\overline F_i$ has degree $-\lambda_i$ and $\overline r'_i$ has degree $\lambda_i$.  Recall that we have the standard partial order $\le$ on the root lattice $\mathbb Z \Phi$.

Assume that a non-trivial linear combination $R = \sum a_{\vec d, \vec e} \overline F^{(\vec d)} (\overline r'_1)^{e_1} \cdots (\overline r'_N)^{e_N}$ equals zero in $C_{cl}$.  Without loss of generality, we may assume that $R$ is homogeneous.  Write $R'$ for the sum of the summands $a_{\vec d, \vec e} \overline F^{(\vec d)} (\overline r'_1)^{e_1} \cdots (\overline r'_N)^{e_N}$ in $R$ such that $a_{\vec d, \vec e} \neq 0$ and $\text{deg} ((\overline r'_1)^{e_1} \cdots (\overline r'_N)^{e_N})$ is minimal with respect to the partial order $\le$ on the root lattice.  Define
\begin{align}
\mu := \min \{\text{deg} ((\overline r'_1)^{e_1} \cdots (\overline r'_N)^{e_N}): & ~ \exists \vec d ~ \text{such that} ~ a_{\vec d, \vec e} \overline F^{(\vec d)} (\overline r'_1)^{e_1} \cdots (\overline r'_N)^{e_N} \nonumber \\
& \text{is a summand of} ~ R ~ \text{and} ~ a_{\vec d, \vec e} \neq 0\}.
\end{align}
Write $R''$ for $R - R'$.

For any $\vec f \in (\mathbb Z_{\ge 0})^N$ such that $\overline F^{(\vec f)} \in \mathcal U^-$ satisfies $\text{deg} \overline F^{(\vec f)} = \mu$, we have
\begin{align}
R''(\overline F^{(\vec f)}) = 0
\end{align}
by degree considerations.  It follows that $0 = R(\overline F^{(\vec f)}) = R'(\overline F^{(\vec f)})$.  By Lemma \ref{mix} and Lemma \ref{fund-comp}, it is clear that for all $\vec e \in (\mathbb Z_{\ge 0})^N$ with $\text{deg} ((\overline r'_1)^{e_1} \cdots (\overline r'_N)^{e_N}) = \mu$, we have
\begin{align}
(\overline r'_1)^{e_1} \cdots (\overline r'_N)^{e_N} (\overline F^{(\vec f)}) = \delta_{\vec e, \vec f},
\end{align}
where $\delta_{\vec e, \vec f}$ is the Kronecker delta.  Therefore, it follows from the equality $R'(\overline F^{(\vec f)}) = 0$ that
\begin{align}
0 = \sum\limits_{\vec d} a_{\vec d, \vec f} \overline F^{(\vec d)}.
\end{align}
Since the $\overline F^{(\vec d)}$'s, for $\vec d \in (\mathbb Z_{\ge 0})^N$, are linearly independent over $\mathbb C$, $a_{\vec d, \vec f}$ must be $0$ for all $\vec d \in (\mathbb Z_{\ge 0})^N$.  Varying $\vec f$, we see that $a_{\vec d, \vec f} = 0$ for all $\vec d, \vec f \in (\mathbb Z_{\ge 0})^N$ such that $\overline F^{(\vec f)} \in \mathcal U^-$ satisfies $\text{deg} \overline F^{(\vec f)} = \mu$.  It follows that all coefficients in $R'$ are zero.  This contradicts our choice of $R'$.

(b) That the elements $(\overline r'_1)^{d_1} \cdots (\overline r'_N)^{d_N}$, for all $\vec d \in (\mathbb Z_{\ge 0})^N$, span $P$ over $\mathbb C$ follows easily from Proposition \ref{PBW-boson-int}.

Let $R = \sum a_{\vec d} (\overline r'_1)^{d_1} \cdots (\overline r'_N)^{d_N}$, $a_{\vec d} \in \mathbb C$, be a $\mathbb C$-linear relation among the $(\overline r'_1)^{d_1} \cdots (\overline r'_N)^{d_N}$'s.  For any $\vec e \in (\mathbb Z_{\ge 0})^N$, we have by Lemma \ref{mix} and Lemma \ref{fund-comp} that
\begin{align}
0 = R(\overline F^{(\vec e)}) = a_{\vec e}.
\end{align}
Hence, all coefficients in $R$ are zero.  Therefore, the elements $(\overline r'_1)^{d_1} \cdots (\overline r'_N)^{d_N}$, for $\vec d \in (\mathbb Z_{\ge 0})^N$, are linearly independent over $\mathbb C$.
\end{proof}

We have an analog of Propositions \ref{Mat} and \ref{Mat-int}.

\begin{prop} \label{Mat-qc}
The $\mathbb C$-algebras $C_{cl}$ and $P$ are independent of the reduced expression (\ref{red-exp}).
\end{prop}

\begin{proof}
This follows easily from Proposition \ref{Mat-int}.
\end{proof}

As a $\mathbb C$-algebra, $C_{cl}$ (resp. $P$) is generated by $\overline F_1, \cdots, \overline F_N, \overline r'_1, \cdots, \overline r'_N$ (resp. $\overline r'_1, \cdots, \overline r'_N$).  The $\overline F_i$'s satisfy their usual relations in $\mathcal U^-$: $\overline F_i \overline F_j - \overline F_j \overline F_i = N_{i,j} [\overline F_i, \overline F_j]$, where the square bracket on the right hand side is the Lie bracket on $\mathfrak n^-$ and the $N_{i,j}$'s are the Chevalley coefficients.  The $\overline r'_i$'s, according to our discussion at the end of section \ref{LS-sect}, commute.  (See the last corollary of section \ref{LS-sect}.)  Commutation relation between $\overline r'_i$ and $\overline F_j$ can be derived easily from formula (\ref{vf}) and discussions at the end of section \ref{Leib-sect}.  Putting these together, we have obtained a presentation of the $\mathbb C$-algebras $C_{cl}$ and $P$.

\begin{prop} \label{a-presentation}
(a) As a $\mathbb C$-algebra, $C_{cl}$ is generated by $\overline F_1, \cdots, \overline F_N, \overline r'_1, \cdots, \overline r'_N$ subject to the relations
\begin{align}
\label{def-rel-1} & \overline F_i \overline F_j - \overline F_j \overline F_i - N_{i,j} [\overline F_i, \overline F_j] = 0, & \\
\label{def-rel-2} & \overline r'_i \overline r'_j - \overline r'_j \overline r'_i = 0, & \\
\label{def-rel-3} & \overline r'_i \overline F_j - \overline F_j \overline r'_i - \sum \limits_{\vec d \neq 0} (n_{\lambda_i}^{\vec d, \vec e} c^{\lambda_j}_{\vec d, 0})|_{q=1} \overline r'_{(\vec e)} = 0,
\end{align}
for all $1 \le i,j \le N$, where $(n_{\lambda_i}^{\vec d, \vec e} c^{\lambda_j}_{\vec d, 0})|_{q=1}$ stands for the image of $n_{\lambda_i}^{\vec d, \vec e} c^{\lambda_j}_{\vec d, 0}$ in $\mathcal A/(1-q)\mathcal A \simeq \mathbb C$.

(b) $P$ is the polynomial algebra over $\mathbb C$ with $\overline r'_1, \cdots, \overline r'_N$ as generators.
\end{prop}

Recall the discussion at the end of section \ref{Leib-sect}, in particular the last corollary of section \ref{Leib-sect}.  From there we know that, for $1 \le i,j \le N$, the commutator $[\overline F_i, \overline r'_j]$ is a polynomial in $\overline r'_1, \cdots, \overline r'_N$.  Namely,
\begin{align} \label{def-act}
[\overline F_i, \overline r'_j] = - \sum \limits_{\vec d \neq 0} (n_{\lambda_j}^{\vec d, \vec e} c^{\lambda_i}_{\vec d, 0})|_{q=1} \overline r'_{(\vec e)} = - \sum \limits_{\vec d \neq 0} (n_{\lambda_j}^{\vec d, \vec e} c^{\lambda_i}_{\vec d, 0})|_{q=1} (\overline r'_1)^{e_1} \cdots (\overline r'_N)^{e_N}.
\end{align}
For each $1 \le i \le N$, define an action of $\overline F_i$ on $P$ by derivations so that, for $1 \le j \le N$, the action of $\overline F_i$ on $\overline r'_j$ is given by (\ref{def-act}).  The following lemma clearly follows from Theorem \ref{LS-str}.

\begin{lem} \label{rep}
The relation (\ref{def-rel-1}) in Proposition \ref{a-presentation} acts on $P$ by $0$.
\end{lem}

\begin{rem}
It follows from Lemma \ref{rep} that $P$ has a structure of a $\mathcal U^-$-module.  Recall that $\mathcal U^-$ is a Hopf algebra, where comultiplication on $\mathcal U^-$ is given by $\Delta (y) = y \otimes 1 + 1 \otimes y$ for all $y \in \mathfrak n^-$.  What we have done so far in this section implies that $P$ is a module algebra over the Hopf algebra $\mathcal U^-$ and $C_{cl}$ is isomorphic to $P \rtimes \mathcal U^-$, the smash product of $\mathcal U^-$ with the $\mathcal U^-$-module algebra $P$.
\end{rem}

Since we know quite a lot about representations of nilpotent Lie algebras, c.f. the work of Dixmier \cite{Dix1, Dix2} and Kirillov \cite{Kir}, it is an interesting problem to look for interpretations of the $\mathfrak n^-$ representation $P$ which arise naturally in other contexts.

Observe that $P$ is very far from being irreducible as an $\mathfrak n^-$ representation.  In fact, recall that $P$ is a $(\mathbb Z \Phi)^+$-graded algebra where, for $1 \le j \le N$, $\overline r'_j$ has degree $\lambda_j$.  For each $1 \le i \le N$, the action of the root vector $\overline F_i$ in $\mathfrak n^-$ decreases the degree of every homogeneous element of $P$ by $\lambda_i$.  Hence, for any non-zero $f \in P$, the smallest sub-representation of $P$ containing $f$ is a non-trivial finite dimensional sub-representation.

Define an action of $P$ on itself by left multiplication, so the two halves $C_{cl}^+ = P$ and $C_{cl}^-$ of $C_{cl}$ both act on $P$.  One evidently has the following

\begin{prop} \label{mod}
The relation (\ref{def-rel-3}) in Proposition \ref{a-presentation} acts by $0$ on $P$.  In particular, $P$ has a structure of a module over the algebra $C_{cl}$.
\end{prop}

Next we prove

\begin{thm} \label{sim-mod}
$P$ is a simple $C_{cl}$-module.
\end{thm}

\begin{proof}
Let $M$ be a nonzero $C_{cl}$-submodule of $P$.  We choose a nonzero element $v \in M$ in the following way.  For each $0 \neq m \in M$, let $D_m$ be the set of degrees of homogeneous components of $m$.  $D_m$, as a subset of $\mathbb Z \Phi$, is partially ordered by the standard partial order $\le $ on $\mathbb Z \Phi$.  Write $H_m$ for the subset of $D_m$ consisting of all those elements that are maximal in $D_m$ with respect to $\le$.  Consider the set $\bigcup \limits_{0 \neq m \in M} H_m$.  This is a subset of $\mathbb Z \Phi$ and, hence, is partially ordered by $\le$.  By the well-ordering principle, this set has a minimal element $\mu$.  $\mu$ belongs to $H_v$ for some $v \in M$.  This is our choice of $v$.  Assume that $v$ is not a scalar, i.e. $v$ is a non-constant polynomial in $\overline r'_1, \cdots, \overline r'_N$.

Since $\mu \in H_v$, among all homogeneous components of $v$, there is one whose degree is maximal with respect to the partial order $\le$, and this maximal degree is $\mu$.  Up to a nonzero constant, the homogeneous component of $v$ whose degree equals $\mu$ has a summand of the form $(\overline r'_1)^{d_1} \cdots (\overline r'_N)^{d_N}$ for some $\vec d \in (\mathbb Z_{\ge 0})^N$ with $\lambda_{\vec d} = \mu$.

Let $j$ be the last nonzero component of $\vec d$.  We compute $\overline F_j \cdot ((\overline r'_1)^{d_1} \cdots (\overline r'_N)^{d_N})$.  Note that this amounts to computing the commutator of $F_j$ with $(r'_1)^{d_1} \cdots (r'_N)^{d_N}$ in $C_q$, by our definition of the action of $\overline F_j$ on $P$.  For this we compute in a similar way as in the proof of Lemma \ref{Leib}.  The result we get is
\begin{align}
\overline F_j \cdot ((\overline r'_1)^{d_1} \cdots (\overline r'_N)^{d_N}) = - \sum (n_{\vec d}^{\lambda_j, \vec e})|_{q=1} \overline r'_{(\vec e)},
\end{align}
where $n^{\lambda_j, \vec e}_{\vec d} := n^{\vec e_j, \vec e}_{\vec d}$ and $(n^{\lambda_j, \vec e}_{\vec d})|_{q=1}$ is the image in $\mathcal A/(1-q)\mathcal A \simeq \mathbb C$ of $n^{\lambda_j, \vec e}_{\vec d}$.  We will not present the details of the computation.  Instead, let us demonstrate a special case.  The computation in general is similar.  By Proposition \ref{a-presentation}, for $1 \le i,j \le N$, $\overline F_j \cdot \overline r'_i = - \sum \limits_{\vec e \neq 0} (n_{\lambda_i}^{\vec e, \vec f} c^{\lambda_j}_{\vec e, 0})|_{q=1} \overline r'_{(\vec f)}$.  Recall that $r'_{(\vec e)} (F_j) = \sum c_{\vec e, \vec f}^{\lambda_j} F^{(\vec f)}$ for all $\vec e \in (\mathbb Z_{\ge 0})^N$ and $1 \le j \le N$.  So we have
\begin{align}
<F_j, E^{(\vec e)}> = <\Delta (F_j), E^{(\vec e)} \otimes 1> = <F^{(\vec e)}, E^{(\vec e)}> <r'_{(\vec e)} (F_j), 1> = <F^{(\vec e)}, E^{(\vec e)}> c^{\lambda_j}_{\vec e, 0}.
\end{align}
It follows that when $\vec e = \vec e_j$, $c^{\lambda_j}_{\vec e, 0} = 1$ and when $\vec e \neq \vec e_j$, $c^{\lambda_j}_{\vec e, 0} = 0$.  Consequently, we get $\overline F_j \cdot \overline r'_i = - \sum (n_{\lambda_i}^{\lambda_j, \vec f} )|_{q=1} \overline r'_{(\vec f)}$, as desired.

Define $\vec h := (d_1, \cdots, d_{j-1}, d_j - 1, 0, \cdots, 0)$.  It is clear that $n_{\vec d} ^{\lambda_j, \vec h} = d_j$ and $n_{\vec d'} ^{\lambda_j, \vec h} = 0$ for all other $\vec d'$ with $\text{deg} ((\overline r'_1)^{d'_1} \cdots (\overline r'_N)^{d'_N}) = \mu$.  It follows that
\begin{align}
\overline F_j \cdot ((\overline r'_1)^{d_1} \cdots (\overline r'_N)^{d_N}) = - d_j (\overline r'_1)^{h_1} \cdots (\overline r'_N)^{h_N} - \sum \limits_{\vec e \neq \vec h} (n_{\vec d}^{\lambda_j, \vec e})|_{q=1} \overline r'_{(\vec e)}
\end{align}
and the coefficient of $(\overline r'_1)^{h_1} \cdots (\overline r'_N)^{h_N}$ in $\overline F_j \cdot v$ equals $- \sum \limits_{\vec d'} (n^{\lambda_j, \vec h}_{\vec d'})|_{q=1} = - d_j$.  In particular, $\overline F_j \cdot v \neq 0$.  By construction, $\mu - \lambda_j$ is a maximal element, with respect to the usual partial order $\le$, in the set of degrees of homogeneous components of $\overline F_j \cdot v$.  So $\mu - \lambda_j$ is an element of the set $\bigcup \limits_{0 \neq m \in M} H_m$.  But $\mu - \lambda_j < \mu$.  This contradicts minimality of $\mu$.  It follows that $v$ must be a nonzero constant.

Note that the $C_{cl}$-submodule $M$ of $P$ is, in particular, an ideal in $P$.  Since $M$ contains a nonzero constant $v$, it must be all of $P$.  Therefore, the $C_{cl}$-module $P$ is simple.
\end{proof}

The structures of $C_{cl}$ that we have explained so far have a strong linear-algebraic or representation-theoretic flavor.  We next explore some structures of $C_{cl}$ that are purely algebraic in nature.

\begin{thm} \label{sim-alg}
$C_{cl}$ is a simple $\mathbb C$-algebra.
\end{thm}

\begin{proof}
Let $I \subseteq C_{cl}$ be a 2-sided ideal.  We must prove that $I$ is either $0$ or $C_{cl}$.

Note first that $I \cap P$ is an ideal in $P$.  Moreover, $I \cap P$ is stable under the action of $\mathcal U^-$.  In fact, for $f \in I \cap P$ and $1 \le i \le N$, we have $\overline F_i \cdot f = \overline F_i f - f \overline F_i$.  Since $I$ is a 2-sided ideal and $f \in I$, both $\overline F_i f$ and $f \overline F_i$ are in $I$.  So, we must have $\overline F_i \cdot f \in I$.  We already know that $\overline F_i \cdot f \in P$.  Hence, $\overline F_i \cdot f \in I \cap P$.

It follows from our argument in the previous paragraph that $I \cap P$ is a $C_{cl}$-submodule of $P$.  By Theorem \ref{sim-mod}, $I \cap P$ is either $P$ or $0$.

If $I \cap P = P$, then $1 \in P \subseteq I$.  It follows that $I = C_{cl}$, since $I$ is a 2-sided ideal in $C_{cl}$.  We are done with this case.

For the rest of the proof we assume that $I \cap P = 0$.

Suppose that $I \not\subseteq P$.  Let $f \in I$ be such that $f \not\in P$.  By the PBW theorem for $C_{cl}$ (Proposition \ref{PBW-qc}), we can expand $f$ as
\begin{align}
f = \sum \overline F^{(\vec d')} f_{\vec d'},
\end{align}
where, for each $\vec d' \in (\mathbb Z_{\ge 0})^N$, $f_{\vec d'}$ is an element in $P$.  Define $S := \{\vec d': \vec d' \neq 0, ~ f_{\vec d'} \neq 0 \}$.  $S$ is partially ordered by declaring that $\vec d' \le \vec d''$ if and only if $\lambda_{\vec d'} \le \lambda_{\vec d''}$.  Since $S$ is a finite set, it must have a maximal element with respect to $\le$.  Let $\vec d$ be one such.  Write $j$ for the first nonzero component of $\vec d$ and define $\vec h := (0, \cdots, 0, d_j - 1, d_{j+1}, \cdots, d_N)$.

Since $I$ is a 2-sided ideal in $C_{cl}$, we have $[\overline r'_j, f] = \overline r'_j f - f \overline r'_j \in I$.  By Proposition \ref{PBW-qc}, we can write
\begin{align}
[\overline r'_j, f] = \sum \overline F^{(\vec h')} g_{\vec h'},
\end{align}
where, for each $\vec h' \in (\mathbb Z_{\ge 0})^N$, $g_{\vec h'}$ is an element in $P$.  We analyze the coefficient $g_{\vec h}$ of $\overline F^{(\vec h)}$ in $[\overline r'_j, f]$.  Suppose that the summand $\overline F^{(\vec d')} f_{\vec d'}$, where $\vec d' \in (\mathbb Z_{\ge 0})^N$, in $f$ contributes to $\overline F^{(\vec h)}$ after taking commutator with $\overline r'_j$.  Then we must have $\lambda_{\vec d'} - \lambda_j \ge \lambda_{\vec h}$.  So $\lambda_{\vec d'} \ge \lambda_{\vec h} + \lambda_j = \lambda_{\vec d}$.  Since $\vec d$ is chosen to be maximal in $S$, the last inequality implies that $\lambda_{\vec d'} = \lambda_{\vec d}$.  Let $\vec d' \in (\mathbb Z_{\ge 0})^N$ satisfy $\lambda_{\vec d'} = \lambda_{\vec d}$.  By Lemma \ref{Leib}, we have
\begin{align}
[\overline r'_j, \overline F^{(\vec d')}] = \sum \limits_{\vec k \neq 0} (n^{\vec k, \vec l}_{\lambda_j} r'_{(\vec k)} (F^{(\vec d')}))|_{q=1} \overline r'_{(\vec l)}.
\end{align}
Since $\lambda_{\vec d'} = \lambda_{\vec d} = \lambda_{\vec h} + \lambda_j$, to compute the coefficient of $\overline F^{(\vec h)}$, we only need to concentrate on those summands in the last displayed expression with $\lambda_{\vec k} = \lambda_j$.  If $\vec k \in (\mathbb Z_{\ge 0})^N$ satisfies $\lambda_{\vec k} = \lambda_j$, then for $n^{\vec k, \vec l}_{\lambda_j}$ to be nonzero,  we must have $\vec l = 0$.  Note that $n^{\vec k, 0}_{\lambda_j}$ equals $1$ if $\vec k = \vec e_j$ and equals $0$ in all other cases.  So to compute the coefficient of $\overline F^{(\vec h)}$ in $[\overline r'_j, f]$, we need only compute $r'_j (F^{(\vec d')})$.  Note that
\begin{align}
<F_j, E_j> <r'_j (F^{(\vec d')}), E^{(\vec h)}> = & <\Delta (F^{(\vec d')}), E_j \otimes E^{(\vec h)}> \nonumber \\
= & <F^{(\vec d')}, E^{(\vec h)} E_j> & \nonumber \\
= & \frac 1 {[d_j]_{\lambda_j}} <F^{(\vec d')}, E^{(\vec d)}>.
\end{align}
The last expression is nonzero if and only if $\vec d = \vec d'$.  It follows that for the purpose of computing the coefficient of $\overline F^{(\vec h)}$ in $[\overline r'_j, f]$, we only need to concentrate on the summand $\overline F^{(\vec d)} f_{\vec d}$ of $f$.  Moreover, the coefficient of $\overline F^{(\vec h)}$ in $[\overline r'_j, f]$ is
\begin{align}
(\frac 1 {[d_j]_{\lambda_j}} <F_j, E_j>^{-1} <F^{(\vec h)}, E^{(\vec h)}>^{-1} <F^{(\vec d)}, E^{(\vec d)}>)|_{q=1} f_{\vec d},
\end{align}
which is nonzero.

From the argument in the last paragraph it follows that, for all $f \in I$ such that $f \not\in P$, there exists a natural number $j$ between $1$ and $N$ such that $[\overline r'_j , f] \neq 0$ and $[\overline r'_j , f] \in I$.  Moreover, the procedure of taking commutator with $\overline r'_j$ increases the degree in the $\overline F$'s.

Now, take an element $f \in I$.  If $f$ happens to be in $P$, then nothing needs to be done.  Otherwise, find $j_1$ such that $1 \le j_1 \le N$ and $0 \neq [\overline r'_{j_1}, f] \in I$.  If $[\overline r'_{j_1}, f]$ happens to be in $P$, then we stop here.  Otherwise, find $j_2$ such that $1 \le j_2 \le N$ and $0 \neq [\overline r'_{j_2}, [\overline r'_{j_1}, f]] \in I$.  We iterate this procedure.  Since taking commutator with the $\overline r'$'s increases the degree in the $\overline F$'s, this procedure terminates in finitely many steps, producing a nonzero element in $I \cap P$.  This contradicts our assumption that $I \cap P = 0$.  Hence, $I$ must be contained in $P$.  It then follows that $I = I \cap P = 0$.
\end{proof}

The following theorem describes the center of the $\mathbb C$-algebra $C_{cl}$.  Its proof uses techniques very similar to those that have been used in the proof of Theorem \ref{sim-alg}.  So we omit the proof.

\begin{thm} \label{center}
The natural inclusion map $\mathbb C \hookrightarrow C_{cl}$ is an isomorphism from $\mathbb C$ to the center $Z(C_{cl})$ of $C_{cl}$.
\end{thm}

\subsection{Poisson-Geometric Aspects} \label{Pois-asp-sect}

In this section we study various Poisson geometric properties of the quasi-classical limit $P$.

\subsubsection{Hayashi Construction and the Poisson Bracket on $P$} \label{Hayashi-sect}

To define the Poisson bracket on $P$, let us first recall the Hayashi construction.

Let $R$ be a commutative $\mathbb C$-algebra of Krull dimension $1$.  Assume that there exists an element $t \in R$ such that $\mathbb C \simeq R/tR$ via the natural inclusion of $\mathbb C$ into $R$ followed by the natural quotient map from $R$ to $R/tR$.  Left multiplication by $t$ defines an endomorphism of $R$.  It is easy to see that this endomorphism induces maps $R/tR \rightarrow tR/t^2R$ and $tR/t^2R \rightarrow t^2R/t^3R$.  Assume, furthermore, that the maps $R/tR \rightarrow tR/t^2R$ and $tR/t^2R \rightarrow t^2R/t^3R$ are isomorphisms.

Let $A$ be an associative $R$-algebra.  Left multiplication by $t \in R$ defines an endomorphism of $A$.  This endomorphism induces maps $A/tA \rightarrow tA/t^2A$ and $tA/t^2A \rightarrow t^2A/t^3A$.  Assume that the two maps $A/tA \rightarrow tA/t^2A$ and $tA/t^2A \rightarrow t^2A/t^3A$ are isomorphisms.

Given the data as above, the Hayashi construction equips the center $Z := Z(A/tA)$ of the $\mathbb C (\simeq R/tR)$-algebra $A/tA$ with a Poisson algebra structure in the following way.  For any $\overline a, \overline b \in Z$, choose representatives $a, b \in A$ of the classes $\overline a, \overline b \in A/tA$.  Since
\begin{align}
\overline a \overline b - \overline b \overline a = 0
\end{align}
in $A/tA$, we must have
\begin{align}
ab - ba \in tA.
\end{align}
Since multiplication by $t$ is an isomorphism from $A/tA$ to $tA/t^2A$, and $ab - ba$ represents a class in $tA/t^2A$, there exists a unique class, denoted by $\overline{\frac 1t (ab - ba)}$, in $A/tA$ which is mapped to the class of $ab - ba$ under the isomorphism $A/tA \rightarrow tA/t^2A$.  We define
\begin{align}
\{\overline a, \overline b\} := \overline{\frac 1t (ab - ba)}.
\end{align}
The Hayashi construction says that this is a well-defined Poisson bracket on $Z$.

Applying this construction to $R = \mathcal A$, $t = 1 - q$ and $A = C^+_{\mathcal A}$, we obtain a Poisson bracket $\{~~,~~\}$ on $Z(P)$.  Note that the conditions on $A = C^+_{\mathcal A}$ in the Hayashi construction are satisfied because of Proposition \ref{PBW-boson-int}.  Since $P$ is a commutative algebra, we have $Z(P) = P$.  Thus the Hayashi construction equips $P$ with a Poisson bracket $\{~~,~~\}$.  The goal of this section is to study the Poisson-geometric properties of the Poisson algebra $(P, \{~~,~~\})$ in detail.

We first work out explicitly the Poisson bracket in the case where $\mathfrak g$ is of type $A_n$.  In this case the Weyl group $W$ is isomorphic to the symmetric group $S_{n+1}$, which we identify with the group of permutations of the integers $\{1, 2, \cdots, n+1\}$.  Let us label the simple refections so that $s_i$, for $1 \le i \le n$, swaps $i$ and $i+1$ and fixes all the other integers.  The longest element $w_0$ in this case has
\begin{align}
(s_1 s_2 \cdots s_n) (s_1 s_2 \cdots s_{n-1}) \cdots (s_1 s_2) s_1
\end{align}
as one of its reduced expressions.  To simplify notation, for $1 \le i \le j \le n$, we write $x_{i,j}$ for $\overline r'_{\alpha_i + \cdots + \alpha_j}$.  Then we have the following formulas for the Poisson bracket in type $A_n$.

\begin{prop} \label{bra-An}
For $1 \le i \le j \le n$ and $1 \le k \le l \le n$, the Poisson bracket on $P$ in the case where $\mathfrak g$ is of type $A_n$ is given by
\begin{align}
\{x_{i,j}, x_{k,l}\} = 
\begin{cases}
0 & \mbox{if } j \le k-2 \\
x_{i,j}x_{k,l} + 2 x_{i,l} & \mbox{if } j = k - 1 \\
-2 x_{k,j}x_{i,l} & \mbox{if } i < k \le j < l \\
0 & \mbox{if } k < i \le j < l \\
-x_{i,j}x_{k,l} & \mbox{if } i = k, j < l \\
x_{i,j}x_{k,l} & \mbox{if } k < i, j = l.
\end{cases}
\end{align}
\end{prop}

\begin{proof}
This follows from Proposition \ref{Dri-Kil-val}, Lemma \ref{com-rel} and the Hayashi construction reviewed above.  Computation of the structure constants in Lemma \ref{com-rel} is not illuminating and is omitted.
\end{proof}

To make this proposition more accessible, we give a few examples.

\begin{exmp} \label{type-A2}
In the case where $n=2$, we write $x$ for $x_1$, $y$ for $x_2$ and $u$ for $x_{1,2}$.  The Poisson bracket is given explicitly as follows
\begin{align}
\{x,y\} = xy + 2u, ~ \{x,u\} = -xu, ~ \{y,u\} = yu.
\end{align}

Let $G$ be a connected algebraic group whose Lie algebra is $\mathfrak g$ and $B$ the Borel subgroup of $G$ whose Lie algebra is $\mathfrak b$.  The flag variety $G/B$ has a Poisson structure called the standard Poisson structure.  Basic definitions about the standard Poisson structure on $G/B$ will be reviewed in section \ref{rel-open-Bru-cell-sect} below.  Explicit formulas for the standard Poisson structure on the open Bruhat cell $Bw_0B/B$ in the flag variety $G/B$ has been computed by Elek and Lu in \cite{EL}.  In the case where the root system is of type $A_2$, their formulas read
\begin{align}
\{z_1, z_2\} = -z_1z_2, ~ \{z_1, z_3\} = z_1z_3 - 2z_2, ~ \{z_2, z_3\} = -z_2z_3.
\end{align}
It is easy to see that the $\mathbb C$-algebra map sending $x$ to $-z_1$, $y$ to $z_3$ and $u$ to $z_2$ is a Poisson isomorphism from $P$ to the coordinate ring of $Bw_0B/B$, thus establishing a Poisson isomorphism $Bw_0B/B \xrightarrow{\sim} \text {Spec} P$ of Poisson varieties.
\end{exmp}

\begin{exmp}
In the case where $n=3$, we write $x, y, z$ for $x_1, x_2, x_3$, respectively.  Also, we write $u, v$ for $x_{1,2}, x_{2,3}$, respectively, and $s$ for $x_{1,3}$.  Then the Poisson bracket on $P$ is given explicitly as follows
\begin{align}
& \{x,y\} = xy + 2u, ~ \{x,z\} = 0, ~ \{x,u\} = -xu, ~ \{x,v\} = xv + 2s, ~ \{x,s\} = -xs, \nonumber \\
& \{y,z\} = yz + 2v, ~ \{y,u\} = yu, ~ \{y,v\} = -yv, ~ \{y,s\} = 0, \nonumber \\
& \{z,u\} = -uz - 2s, ~ \{z,v\} = zv, ~ \{z,s\} = zs, \nonumber \\
& \{u,v\} = -2ys, ~ \{u,s\} = -us, \nonumber \\
& \{v,s\} = vs.
\end{align}
\end{exmp}

We proceed to write down explicit formulas for the Poisson bracket on $P$ in the case where $\mathfrak g$ is of type $G_2$.  In this case we write $\alpha_1, \alpha_2$ for the two simple roots, with $\alpha_1$ shorter than $\alpha_2$.  For the reduced expression (\ref{red-exp}) for $w_0$, we choose
\begin{align}
w_0 = s_{\alpha_1} s_{\alpha_2} s_{\alpha_1} s_{\alpha_2} s_{\alpha_1} s_{\alpha_2}.
\end{align}
Then the induced enumeration of the set of positive roots is
\begin{align}
\lambda_1 = \alpha_1 \preceq \lambda_2 = 3\alpha_1 + \alpha_2 \preceq \lambda_3 = 2\alpha_1 + \alpha_2 \preceq \lambda_4 = 3\alpha_1 + 2\alpha_2 \preceq \lambda_5 = \alpha_1 + \alpha_2 \preceq \lambda_6 = \alpha_2.
\end{align}
Write $x_i$ for $\overline r'_i$ for $1 \le i \le 6$.  Then we have the following explicit formulas for the Poisson bracket on $P$.

\begin{prop} \label{type-G2}
The Poisson bracket on $P$ in the case where $\mathfrak g$ is of type $G_2$ is given by
\begin{flalign}
& \{x_1, x_2\} = -3x_1x_2, ~ \{x_1, x_3\} = -x_1x_3 + 2x_2, ~ \{x_1, x_4\} = -6x_3^2, ~ \{x_1, x_5\} = x_1x_5 + 4x_3, & \nonumber \\
& \{x_1, x_6\} = 3x_1x_6 + 6x_5, ~ \{x_2, x_3\} = -3x_2x_3, ~ \{x_2, x_4\} = -3x_2x_4 + 6x_3^3, ~ \{x_2, x_5\} = -6x_3^2, & \nonumber \\
& \{x_2, x_6\} = 3x_2x_6 - 18x_3x_5 - 6x_4, ~ \{x_3, x_4\} = -3x_3x_4, ~ \{x_3, x_5\} = -x_3x_5 + 2x_4, ~ \{x_3, x_6\} = -6x_5^2, & \nonumber \\
& \{x_4, x_5\} = -3x_4x_5, ~ \{x_4, x_6\} = -3x_4x_6 + 6x_5^3, ~ \{x_5, x_6\} = -3x_5x_6.
\end{flalign}
\end{prop}

\begin{proof}
This again follows from Proposition \ref{Dri-Kil-val}, Lemma \ref{com-rel} and the Hayashi construction reviewed above.  Computation of the structure constants in Lemma \ref{com-rel} involves less work than in the case where $\mathfrak g$ is of type $A_n$, but is still unpleasant.  This computation is also omitted.
\end{proof}

\begin{rem}
Observe that the $\mathbb C$-algebra map sending $x_1$ to $-z_1$, $x_2$ to $z_2$, $x_3$ to $z_3$, $x_4$ to $-z_4$, $x_5$ to $z_5$, $x_6$ to $z_6$ is an isomorphism of Poisson algebras from $P$ to the coordinate ring of $Bw_0B/B$.  Here, $z_1, \cdots, z_6$ are the Bott-Samelson coordinates used by Elek and Lu in \cite{EL}.  Hence, we again have a Poisson isomorphism $Bw_0B/B \xrightarrow{\sim} \text {Spec} P$ of Poisson varieties.
\end{rem}

In general, using Proposition \ref{Dri-Kil-val}, Lemma \ref{mix} and Lemma \ref{com-rel}, for the Poisson bracket on $P$, we have the following

\begin{prop} \label{Pois-bra}
For any $1 \le i < j \le N$, the Poisson bracket of $\overline r'_i$ with $\overline r'_j$ in $P$ is given by
\begin{flalign}
& \{\overline r'_{i}, \overline r'_{j}\} = & \nonumber \\
& -(\lambda_i, \lambda_j) \overline r'_{i} \overline r'_{j} + \sum\limits_{\lambda_{\vec d} = \lambda_1 + \lambda_2} [(1-q)^{-1} (q_{\lambda_i} - q_{\lambda_i}^{-1}) (q_{\lambda_j} - q_{\lambda_j}^{-1}) (-1)^{|\vec d|} c_{\vec d} \prod\limits_{k=1}^N (q_{\lambda_k} - q_{\lambda_k}^{-1})^{-\lambda_k}]|_{q=1} \prod\limits_{k=1}^N (\overline r'_k)^{d_k}.
\end{flalign}
\end{prop}

Recall that, according to Proposition \ref{Mat-qc}, the $\mathbb C$-algebra $P$ is independent of the choice of the reduced expression (\ref{red-exp}).  It is natural to ask whether or not the Poisson algebra $(P, \{~~,~~\})$ depends on the reduced expression for $w_0$.  The following result is an answer to this question.

\begin{prop} \label{indep}
Let $\tilde r'_1, \cdots, \tilde r'_N$ be Kashiwara operators defined by a choice of reduced expression for $w_0$ which is different than (\ref{red-exp}).  Let $\{~~,~~\}^{\wedge}$ be the corresponding Poisson bracket on $P$ and $\overline r'_i = \overline r'_i(\overline {\tilde r'_1}, \cdots, \overline {\tilde r'_N}) \in \mathbb C[\overline {\tilde r'_1}, \cdots, \overline {\tilde r'_N}]$ be the induced coordinate changes.  Then, for all $1 \le i,j \le N$, we have
\begin{align} \label{compare-Poisson-brackets}
\{\overline r'_i(\overline {\tilde r'_1}, \cdots, \overline {\tilde r'_N}), \overline r'_j(\overline {\tilde r'_1}, \cdots, \overline {\tilde r'_N})\}^{\wedge} = \{\overline r'_i, \overline r'_j\} (\overline {\tilde r'_1}, \cdots, \overline {\tilde r'_N}).
\end{align}
\end{prop}

\subsubsection{Generic Rank of $P$: Comparison with the Kirillov-Kostant Poisson Bracket} \label{gen-rk-sect}

We first recall the Kirillov-Kostant Poisson bracket on $\mathfrak n^{\ast}$, the dual vector space of $\mathfrak n$, in the case where $\mathfrak g$ is of type $A_n$.  Abstractly, $\mathfrak n^{\ast}$ can be identified with the affine space $\mathbb A^{n(n+1)/2}$.  For any pair $(i,j)$ satisfying $1 \le i \le j \le n$, we have a coordinate function $y_{i,j}$ on the  algebraic variety $\mathfrak n^{\ast}$.  Specifically, choose $\mathfrak n$ to be the Lie subalgebra of $\mathfrak {sl}_{n+1}$ consisting of strictly upper triangular matrices.  Then, for $M \in \mathfrak n$, $y_{i,j}(M)$ equals the $(i,j)$th entry of $M$.  The Kirillov-Kostant Poisson bracket is given explicitly by the following formulas.
\begin{align} \label{Kir-bra}
\{y_{i,j}, y_{k,l}\} = 
\begin{cases}
y_{i,l} & \mbox{if } k = j + 1 \\
- y_{k,j} & \mbox{if } i = l + 1 \\
0 & \mbox{else}.
\end{cases}
\end{align}

Recall that, by Proposition \ref{a-presentation}, $P$ is a polynomial algebra in $N$ variables, where, when $\mathfrak g$ is of type $A_n$, $N$ equals $n(n+1)/2$.  It follows that $\text{Spec} P$ can be identified abstractly with the affine space $\mathbb A^{n(n+1)/2}$.  Thus we identify $\text{Spec} P$ with $\mathfrak n^{\ast}$ as algebraic varieties.  Comparing the formula above for the Kirillov-Kostant Poisson bracket and the formula in Proposition \ref{bra-An}, one sees readily that the Poisson bracket on $P$ degenerates to the Kirillov-Kostant Poisson bracket in the case where $\mathfrak g$ is of type $A_n$, in the sense that the Kirillov-Kostant Poisson bracket is the linear term of the Poisson bracket on $P$.

\begin{defn}
(a) The generic rank $\text {gr}_{\mathfrak g}$ of $(P, \{~~,~~\})$ is the maximum of the dimension of the symplectic leaves in $(\text {Spec} P, \{~~,~~\})$.

(b) The generic rank $\text {gr}'_{\mathfrak g}$ of the Kirillov-Kostant Poisson bracket is the maximum of the dimension of the symplectic leaves in $\mathfrak n^{\ast}$ equipped with the Kirillov-Kostant Poisson bracket.
\end{defn}

In the case where $\mathfrak g$ is of type $A_n$, we write $\text {gr}_n$ for $\text {gr}_{\mathfrak g}$ and $\text{gr}'_n$ for $\text{gr}'_{\mathfrak g}$.  Our observation that the Kirillov-Kostant Poisson bracket is a degeneration of the Poisson bracket on $P$ readily implies

\begin{lem} \label{gen-rk-ez}
$\text {gr}_n \ge \text {gr}'_n$.
\end{lem}

In fact, for a general $\mathfrak g$, one has the following much stronger statement.

\begin{thm} \label{gen-rk}
$\text {gr}_{\mathfrak g} = \text {gr}'_{\mathfrak g}$.
\end{thm}

We will be able to prove Theorem \ref{gen-rk} after we identify $(\text{Spec} P, \{~~,~~\})$ with a more familiar Poisson algebra in Theorem \ref{main-conj}.  For now we just remark that $\text {gr}'_n$ has been computed by Panov in \cite{Pan}.  According to him, $\text {gr}'_n = \left \lfloor \frac 12 n^2 \right \rfloor$, where $\left \lfloor ~~ \right \rfloor$ is the standard floor function.  Using computer programs, one can verify that $\text {gr}_n = \left \lfloor \frac 12 n^2 \right \rfloor$ for small $n$.


\subsubsection{Poisson Center} \label{Pois-cen-sect}

\begin{defn}
The Poisson center $Z_{\text{Pois}}(P)$ of $P$ is the set of all $f \in P$ such that $\{f,g\} = 0$ for all $g \in P$.  Elements of $Z_{\text{Pois}}(P)$ are called Casimir functions.
\end{defn}

We start with two examples.

\begin{exmp}
Consider the example where $\mathfrak g$ is of type $A_1$.  In this case, $P$ is a polynomial algebra with one single generator $x$.  So $P$ is Poisson commutative, i.e. the Poisson center $Z_{\text {Pois}}(P)$ is a polynomial algebra with $x$ as its generator.
\end{exmp}

\begin{exmp}
Consider the example where $\mathfrak g$ is of type $A_2$.  Let us use the same notations as in Example \ref{type-A2}.  It is easy to verify that $\psi := (xy + u)u \in Z_{\text {Pois}}(P)$.  In fact, one can prove that $Z_{\text {Pois}}(P)$ is freely generated by $\psi$, i.e. the Poisson center $Z_{\text {Pois}}(P)$ is a polynomial algebra with $\psi$ as its generator.
\end{exmp}

In general, to write down Casimir functions in $P$ for $\mathfrak g$ of type $A_n$, we need some notations.  Let
\begin{align}
\kappa = \{1 \le \kappa_1 < \kappa_2 < \cdots < \kappa_k = n\}
\end{align}
be a partition of $n$.  For such a partion $\kappa$, we say that $k$ is the size of $\kappa$ and denote the size of $\kappa$ by $|\kappa|$.  The intervals $[1, \kappa_1], [\kappa_1 + 1, \kappa_2], \cdots, [\kappa_{k-1} + 1, \kappa_k]$ are called the parts of $\kappa$.  We write $x_{\kappa}$ for the monomial 
\begin{align}
x_{1,\kappa_1} x_{\kappa_1+1, \kappa_2} \cdots x_{\kappa_{k-1}+1, \kappa_k},
\end{align}
and define $\psi \in P$ by 
\begin{align}
\psi := (\sum\limits_{\kappa \vdash n} x_{\kappa})x_{1,n},
\end{align}
where $\kappa \vdash n$ indicates that $\kappa$ is a partition of $n$.

\begin{prop} \label{Poi-cen}
If $\mathfrak g$ is of type $A_n$, then $\psi \in Z_{\text {Pois}}(P)$.
\end{prop}

\begin{proof}
Note that $P$ is Poisson-generated by $x_1, \cdots, x_n$, i.e. the smallest Poisson subalgebra of $P$ containing $x_1, \cdots, x_n$ is $P$.  Hence, to prove that $\psi$ is a Casimir function, it suffices to prove that $\psi$ Poisson-commutes with $x_1, \cdots, x_n$.

Define $\psi_1 := \sum\limits_{\lambda \vdash n} x_{\kappa}$.  So $\psi = x_{1,n} \psi_1$.

We first compute $\{x_1, \psi\}$.  For this purpose we define $T := \{\kappa \vdash n: \kappa_1 > 1\}$ and $T' := \{\kappa \vdash n: \kappa_1 = 1\}$.  By Proposition \ref{bra-An}, we have
\begin{align}
\{x_1, \psi\} = & \{x_1, x_{1,n} \psi_1\} \nonumber \\
= & \{x_1, x_{1,n}\} \psi_1 + x_{1,n} \{x_1, \psi_1\} \nonumber \\
= & - x_1 x_{1,n} \psi_1 + x_{1,n} \{x_1, \psi_1\}.
\end{align}
By our definition of $T$, $T'$ and Proposition \ref{bra-An},  we have
\begin{align}
\{x_1, \psi_1\} = & \sum\limits_{\kappa \in T} \{x_1, x_{\kappa}\} + \sum\limits_{\kappa \in T'} \{x_1, x_{\kappa}\} \nonumber \\
= & - x_1 \sum\limits_{\kappa \in T} x_{\kappa} + (x_1 \sum\limits_{\kappa \in T'} x_{\kappa} + 2 x_1 \sum\limits_{\kappa \in T'} x_{\tilde{\kappa}}),
\end{align}
where $\tilde{\kappa}$ is obtained from $\kappa$ by merging its first two parts into one.  More concretely, if $\kappa = \{1 \le \kappa_1 < \kappa_2 < \cdots < \kappa_k = n\}$, then $\tilde {\kappa}$ is the partition $\{1 \le \tilde{\kappa}_1 < \tilde{\kappa}_2 < \cdots < \tilde{\kappa}_{k-1} = n\}$, where $\tilde{\kappa}_i = \kappa_{i+1}$ for all $1 \le i \le k-1$.  Noticing that the function
\begin{align}
T' \rightarrow T: \kappa \mapsto \tilde{\kappa}
\end{align}
is a bijection, we see that $\sum\limits_{\kappa \in T'} x_{\tilde{\kappa}} = \sum\limits_{\kappa \in T} x_{\kappa}$.  It follows that
\begin{align}
\{x_1, \psi_1\} = & - x_1 \sum\limits_{\kappa \in T} x_{\kappa} + x_1 \sum\limits_{\kappa \in T'} x_{\kappa} + 2 x_1 \sum\limits_{\kappa \in T} x_{\kappa} \nonumber \\
= & x_1 \sum\limits_{\kappa \in T} x_{\kappa} + x_1 \sum\limits_{\kappa \in T'} x_{\kappa} \nonumber \\
= & x_1 \sum\limits_{\kappa \vdash n} x_{\kappa} \nonumber \\
= & x_1 \psi_1.
\end{align}
Therefore, $\{x_1, \psi\} = 0$.

Using a similar argument as in the previous paragraph, one can show that $\{x_n, \psi\} = 0$.

Let $i$ be an integer satisfying $1 < i < n$.  We prove that $\{x_i, \psi\} = 0$.  Note that, by Proposition \ref{bra-An}, we have
\begin{align}
\{x_i, \psi\} = \{x_i, x_{1,n} \psi_1\} = x_{1,n} \{x_i, \psi_1\}.
\end{align}
It thus suffices to prove that $\{x_i, \psi_1\} = 0$.  We do this by induction on $n$.

The base case $n = 3$ can be verified easily by hand, using Proposition \ref{bra-An}.

Now assume that $n > 3$.  A close look at the formula in Proposition \ref{bra-An} tells us that we have six different cases for the Poisson bracket of $x_{i,j}$ with $x_{k,l}$ because we have six different `relative positions' of the two intervals $[i,j]$ and $[k,l]$.

Suppose $i > 2$.  Notice that the relative position of the intervals $[i,i]$ and $[1, \kappa_1]$ is the same as the relative position of the intervals $[i,i]$ and $[2, \kappa_1]$, whenever $\kappa_1 > 1$.  So, in order to compute $\{x_i, x_{1, \kappa_1}\}$, for $\kappa_1 > 1$, we can first replace the index $1$ by $2$, then compute the Poisson bracket $\{x_i, x_{2, \kappa_1}\}$ as if we are in type $A_{n-1}$, and finally replace the index $2$ by $1$.  Using this observation, we compute
\begin{align}
\{x_i, \psi_1\} = & \sum\limits_{\kappa \in T} \{x_i, x_{\kappa}\} + \sum\limits_{\kappa \in T'} \{x_i, x_{\kappa}\} \nonumber \\
= & 0 + \sum\limits_{\kappa \in T'} \{x_i, x_{\kappa}\} \nonumber \\
= & \sum\limits_{\kappa \in T'} x_1 \{x_i, x_{\hat{\kappa}}\} \nonumber \\
= & 0,
\end{align}
where $\hat{\kappa}$ is obtained from $\kappa$ by deleting its first part.  More concretely, if $\kappa = \{1 \le \kappa_1 < \kappa_2 < \cdots < \kappa_k = n\}$, then $\tilde {\kappa} = \{1 < \tilde{\kappa}_1 < \tilde{\kappa}_2 < \cdots < \tilde{\kappa}_{k-1} = n\}$, where $\tilde{\kappa}_i = \kappa_{i+1}$ for all $1 \le i \le k-1$.  Here, in the second and last steps of the computation, we have used our observation above and the induction hypothesis.

Suppose $i < n-1$.  Notice that the relative position of the intervals $[i,i]$ and $[\kappa_{k-1} + 1, \kappa_k]$ is the same as the relative position of the intervals $[i,i]$ and $[\kappa_{k-1} + 1, \kappa_k - 1]$, whenever $\kappa_{k-1} + 1 < n$.  So, in order to compute $\{x_i, x_{\kappa_{k-1} + 1, \kappa_k}\}$, we can first replace the index $n$ by $n-1$, then compute the Poisson bracket $\{x_i, x_{\kappa_{k-1} + 1, n-1}\}$ as if we are in type $A_{n-1}$, and finally replace the index $n-1$ by $n$.  This observation allows us to use a similar argument as in the previous paragraph to show that $\{x_i, \psi_1\} = 0$ when $i < n-1$.

The last two paragraphs cover all possible cases. So the induction step is finished.
\end{proof}

\begin{rem}
For the Kirillov-Kostant Poisson bracket on $\mathfrak n^{\ast}$, the function $y_{1,n}$ is Casimir, as can be easily verified using formula (\ref{Kir-bra}).  Our $\psi$ is meant to be a replacement of $y_{1,n}$.
\end{rem}

The two examples above are clearly special cases of the last proposition.  We point out here that, for a general $n$, $\psi$ does not generate $Z_{\text {Pois}}(P)$, as the next example shows.

\begin{exmp}
Let $\mathfrak g$ be of type $A_3$.  An easy computation shows that $\psi' := x_{1,2} x_{2,3} - x_2 x_{1,3}$ is also a Casimir function.  It is clear that $\psi$ and $\psi'$ are algebraically independent.

The phenomenon that is worth noticing about this example is that $y_{1,3}$ and $y_{1,2} y_{2,3} - y_2 y_{1,3}$ are Casimir functions with respect to the Kirillov-Kostant Poisson bracket.  For a general $n$, we know that $y_{1,n}$ and $y_{1,n-1} y_{2,n} - y_{2,n-1} y_{1,n}$ are Casimir functions with respect to the Kirillov-Kostant Poisson bracket.  We also know that $\psi$ is a replacement of $y_{1,n}$.  Now the question is: what is a reasonable replacement of $y_{1,n-1} y_{2,n} - y_{2,n-1} y_{1,n}$?  When $n=3$ this example tells us that $\psi'$ is a good candidate.  It would be very interesting to find the answer to this question for a general $n$.  It would also be interesting to look for deeper reasons that explain the similarity of the Casimir functions for our Poisson bracket and those for the Kirillov-Kostant Poisson bracket.
\end{exmp}

Let $(Q, \{~~,~~\})$ be a Poisson algebra over $\mathbb C$.  Recall that a quantization of $(Q, \{~~,~~\})$ is an $\mathcal A$-algebra $D$ satisfying the conditions in the Hayashi construction, such that when we apply the Hayashi construction to the data $(R = \mathcal A, t = 1-q, A = D)$, the $\mathbb C$-Poisson algebra we get is Poisson isomorphic to $(Q, \{~~,~~\})$.  Let $D$ be a quantization of $(Q, \{~~,~~\})$ and $f \in Q$ be a Casimir function.  A quantization of the Casimir function $f$ is a lift of $f \in D/(1-q)D$ to an element in the center $Z(D)$ of $D$.

Now we quantize the Casimir function $\psi$ in Proposition \ref{Poi-cen}.  For $1 \le i \le j \le n$, to simplify notation, we write $r_{i,j}$ for the Kashiwara operator $r'_{\alpha_i + \cdots + \alpha_j}$.

\begin{prop}  \label{quan-of-Casimir}
If $\mathfrak g$ is of type $A_n$, then the element
\begin{align}
\Psi := \sum\limits_{\kappa \vdash n} q^{n - |\kappa|} r_{1,\kappa_1} r_{\kappa_1 + 1, \kappa_2} \cdots r_{\kappa_{|\kappa| - 1} + 1, \kappa_{|\kappa|}} r_{1,n}
\end{align}
is a quantization of the Casimir function $\psi$.
\end{prop}

\subsubsection{Compatible Poisson Brackets} \label{compatible-Pois-bra-sect}

Recall our formula (\ref{vf}) and discussions following Lemma \ref{a-coeff}.  For each $1 \le i \le N$, commuting with $\overline F_i$ defines a derivation on $P$.  These derivations make $P$ a representation of the Lie algebra $\mathfrak n^-$ (Lemma \ref{rep}).  In other words, we have a Lie algebra morphism $\mathfrak n^- \rightarrow \mathcal X (\text {Spec} P)$, where $\mathcal X (\text {Spec} P)$ stands for the Lie algebra of vector fields on $\text {Spec} P$ equipped with the standard commutator Lie bracket.  By abuse of notation, for all $1 \le i \le N$, we write $\overline F_i$ also for the image of the root vector $\overline F_i$ under this Lie algebra morphism.  The following general formula is an easy consequence of formula (\ref{vf}) or Proposition \ref{a-presentation}.

\begin{prop} \label{vf-formula}
For all $1 \le i,j \le N$, the action of the vector field $\overline F_i$ on the coordinate function $\overline r'_j$ is given by the formula
\begin{align} \label{vf-gen}
\overline F_i \cdot \overline r'_j = - \sum (n_{\lambda_j}^{\lambda_i, \vec e})|_{q=1} \overline r'_{(\vec e)}.
\end{align}
\end{prop}

It is natural to ask how the vector field $\overline F_i$ ($1 \le i \le N$) changes if one chooses a different reduced expression for $w_0$ than the one we fixed in (\ref{red-exp}).  To answer this question, we have

\begin{prop}
For all positive roots $\lambda$, let $F_{\lambda}, \tilde F_{\lambda} \in U^-$ be root vectors corresponding to the positive root $\lambda$ defined by two different choices of reduced expressions for $w_0$.  Then the vector fields $\overline F_{\lambda}$ and $\overline {\tilde F}_{\lambda}$ on $\text {Spec} P$ are the same, up to a sign.
\end{prop}

\begin{proof}
(Sketch.)  One proves that the root vectors $\overline F_{\lambda}$ and $\overline {\tilde F}_{\lambda}$ differ at most by a sign.  To this end one uses similar arguements as in the proof of Proposition \ref{Mat}.
\end{proof}

In the case where $\mathfrak g$ is of type $A_n$, formula (\ref{vf-gen}) can be made much more concrete.  Recall our notations and conventions for Proposition \ref{bra-An}.  In this situation, for $1 \le i \le j \le n$, we write $\partial_{i,j}$ for the partial derivative with respect to the variable $x_{i,j}$.  The following proposition is not hard to prove.

\begin{prop} \label{vf-gen-a}
Let $\mathfrak g$ be of type $A_n$.  For $1 \le i \le j \le n$, the vector field $\overline F_{\alpha_i + \cdots + \alpha_j}$ is given by
\begin{align}
\overline F_{\alpha_i + \cdots + \alpha_j} = -\partial_{i,j} + \sum \limits _{k=j+1}^n x_{j+1,k} \partial_{i,k}.
\end{align}
\end{prop}

For convenience, we write $\pi$ for the Poisson bivector on $\text {Spec} P$.  Namely, for $f,g \in P$, we have
\begin{align}
\{f,g\} = \pi (df, dg).
\end{align}
Let $[~~,~~]$ be the Schouten bracket of multi-vector fields on $\text {Spec} P$.  Motivated by the infinitesimal criterion for Poisson action of Poisson-Lie groups of Semenov-Tian-Shansky \cite{STS} (see also \cite{LM}), we study in type $A_n$ the `deformed bivector' $[\overline F_i, \pi]$, for all $1 \le i \le N$.  Concretely, for $f,g \in P$, $[\overline F_i, \pi](df,dg) = \overline F_{i \cdot} \{f,g\} - \{\overline F_{i \cdot} f, g\} - \{f, \overline F_{i \cdot} g\}$, which should be familiar to readers in the field of Poisson-Lie groups.

Note that, a priori, the deformed bivector $[\overline F_i, \pi]$, $1 \le i \le N$, need not be Poisson, i.e. Jacobi identity $[[\overline F_i, \pi], [\overline F_i, \pi]] = 0$ need not hold.

\begin{exmp} \label{def-biv}
Let $\mathfrak g$ be of type $A_2$.  Recall our notations in Example \ref{type-A2}.  We have
\begin{align}
& [\overline F_1, \pi](dx,dy) = y, ~ [\overline F_1, \pi](dx,du) = -2xy - u, ~ [\overline F_1, \pi](dy,du) = y^2; \nonumber \\
& [\overline F_2, \pi](dx,dy) = -x, ~ [\overline F_2, \pi](dx,du) = 0, ~ [\overline F_2, \pi](dy,du) = -u; \nonumber \\
& [\overline F_{\alpha_1 + \alpha_2}, \pi](dx,dy) = 2, ~ [\overline F_{\alpha_1 + \alpha_2}, \pi](dx,du) = x, ~ [\overline F_{\alpha_1 + \alpha_2}, \pi](dy,du) = -y.
\end{align}
It is then easy to verify that all of $[\overline F_1, \pi]$, $[\overline F_2, \pi]$ and $[\overline F_{\alpha_1 + \alpha_2}, \pi]$ are indeed Poisson bivectors.
\end{exmp}

This example is a manifestation of a much more general fact.  In fact, if $\mathfrak g$ is of type $A_n$ for an arbitrary $n$, we have

\begin{thm} \label{deformed-bivect}
Let $\mathfrak g$ be of type $A_n$.  For $1 \le i \le j \le n$, the deformed bivector $[\overline F_{\alpha_i + \cdots + \alpha_j}, \pi]$ is Poisson, i.e. Jacobi identity $[[\overline F_{\alpha_i + \cdots + \alpha_j}, \pi], [\overline F_{\alpha_i + \cdots + \alpha_j}, \pi]] = 0$ holds.
\end{thm}

The proof of Theorem \ref{deformed-bivect} is computational in nature.  It is omitted.

Recall that two Poisson bivectors $\pi$ and $\pi'$ on a variety $X$ are called compatible if $\pi + \pi'$ is (equivalently, all linear combinations of $\pi$ and $\pi'$ are) again Poisson.  By general properties of the Schouten bracket, it is easy to verify that $\pi$ and $\pi'$ are compatible if and only if the equation $[\pi, \pi'] = 0$ holds.

\begin{exmp} \label{comp}
By an easy computation, one verifies that all three Poisson brackets $[\overline F_1, \pi]$, $[\overline F_2, \pi]$ and $[\overline F_{\alpha_1 + \alpha_2}, \pi]$ on $\text{Spec} P$ in Example \ref{def-biv} are compatible with $\pi$.
\end{exmp}

Example \ref{comp} is a manifestation of the following general result.

\begin{thm} \label{comp-gen}
Let $\mathfrak g$ be of type $A_n$.  For any $1 \le i \le j \le n$, the Poisson bracket $[\overline F_{\alpha_i + \cdots + \alpha_j}, \pi]$ is compatible with $\pi$.
\end{thm}

Let $X$ be a variety and $\phi: X \rightarrow X$ an isomorphism of varieties.  Suppose that the source is equipped with a Poisson bracket $\{~~,~~\}$ whose Poisson bivector is $\pi$.  Then there is a unique Poisson bracket $\{~~,~~\}'$ on the target such that $\phi$ is a Poisson isomorphism.  The Poisson bivector corresponding to $\{~~,~~\}'$ is the push forward of $\pi$ along $\phi$.  Concretely, for any functions $f$ and $g$ on $X$,
\begin{align}
\{f,g\}' = (\phi^{-1})^{\ast} \{\phi^{\ast} f, \phi^{\ast} g\}.
\end{align}

Suppose that an algebraic group $K$ acts on $X$.  Let $\mathfrak k$ be the Lie algebra of $K$.  For $v \in \mathfrak k$ and $t \in \mathbb C$, consider the push forward of $\pi$ along the isomorphism $\exp{(tv)}: X \rightarrow X$, where $\exp : \mathfrak k \rightarrow K$ is the exponential map and the map $\exp{(tv)}$ sends $x \in X$ to $\exp{(tv)} \cdot x \in X$.  For functions $f$ and $g$ on $X$, the push forward of $\pi$ along $\exp{(tv)}$ takes the value
\begin{align}
\pi (df, dg) - t [\hat v, \pi] (df, dg) + t^2 [\hat v, [\hat v, \pi]] (df, dg) + \cdots
\end{align}
on the pair $(df,dg)$ of $1$-forms, where $\hat v$ stands for the vector field on $X$ generated by the infinitesimal action of $v \in \mathfrak k$, and $\cdots$ stands for terms containing the factor $t^m$ for some $m \ge 3$.

Suppose it happens to be the case that $[\hat v, [\hat v, \pi]] = 0$.  Then the formula above reduces to
\begin{align}
\pi (df, dg) - t [\hat v, \pi] (df, dg).
\end{align}
Jacobi identity for the push forward of $\pi$ along the isomorphism $\exp{(tv)}$ then implies that
\begin{align}
0 = & [\pi - t [\hat v, \pi], \pi - t [\hat v, \pi]] \nonumber \\
= & [\pi, \pi] - 2t [\pi, [\hat v, \pi]] + t^2 [[\hat v, \pi], [\hat v, \pi]].
\end{align}
Since $t \in \mathbb C$ is arbitrary, it follows in particular that $[\pi, [\hat v, \pi]] = 0$, i.e. the two Poisson bivectors $\pi$ and $[\hat v, \pi]$ are compatible.

By this argument, Theorem \ref{comp-gen} follows from the following theorem whose proof is an unilluminating case-by-case computation.  The proof of the theorem will be omitted for lack of space.

\begin{thm} \label{comp-reason}
Let $\mathfrak g$ be of type $A_n$.  For any $1 \le i \le j \le n$, we have $[\overline F_{\alpha_i + \cdots + \alpha_j}, [\overline F_{\alpha_i + \cdots + \alpha_j}, \pi]] = 0$.
\end{thm}

We write $\pi$ for the Poisson bivector of the Poisson bracket on $\text{Spec} P (\simeq \mathfrak n^{\ast})$ and $\pi_{KK}$ for the Poisson bivector for the Kirillov-Kostant Poisson bracket on $\mathfrak n^{\ast}$.

\begin{thm} \label{KK}
Let $\mathfrak g$ be of type $A_n$.  On $\mathfrak n^{\ast}$, the two Poisson bivectors $\pi$ and $\pi_{KK}$ are compatible.
\end{thm}

\begin{proof}
One proves that for functions $f,g$ on $\mathfrak n^{\ast}$, the bracket $\{~~,~~\}'$ defined by
\begin{align}
\{f,g\}' := \{f,g\} - 2 \{f,g\}_{KK}
\end{align}
is Poisson, where $\{~~,~~\}_{KK}$ stands for the Kirillov-Kostant Poisson bracket on $\mathfrak n^{\ast}$.  One proves this by induction on $n$.  The base case $n=1$ is trivial because $\pi = \pi_{KK} = 0$ in this situation.  For the induction step one only needs to compute, for $1 \le i \le j \le n$, $1 \le k \le l \le n$ and $1 \le r \le s \le n$, the Jacobiator
\begin{align}
J(x_{i,j},x_{k,l},x_{r,s}) := & \{\{x_{i,j}, x_{k,l}\}', x_{r,s}\}' + \{\{x_{k,l}, x_{r,s}\}', x_{i,j}\}' \nonumber \\
& + \{\{x_{r,s}, x_{i,j}\}', x_{k,l}\}'
\end{align}
in the case where at least one of $j, l, s$ is equal to $n$.  This can be done by another case-by-case analysis.  This analysis is tedius and unilluminating.  It is omitted.
\end{proof}

Many interesting results in this section work only in type $A$.  It is very interesting to explore whether or not analogous statements hold in other types.  Another interesting question that is worth thinking about is how to interpret the compatibility results (at least in type $A$) in terms of the geometry of the variety of Lagrangian subalgebras of Evens and Lu \cite{EvLu,EvLu2}.

Before closing this section, let us very quickly point out a way in which the vector fields $\overline F_i$, $1 \le i \le N$, may potentially help us determine the Poisson center $Z_{\text {Pois}}(P)$ of $P$.  Recall our notations prior to Proposition \ref{Poi-cen}.

\begin{lem}
Let $\mathfrak g$ be of type $A_n$.  Then the Poisson center of $P$ equipped with the deformed Poisson bivector $[\overline F_{\alpha_1 + \cdots + \alpha_n}, \pi]$ is freely generated as a $\mathbb C$-algebra by $x_{i,j}$ for $1 < i \le j <n$ and the function
\begin{align}
(\sum \limits_{\substack {\kappa \vdash n \\ \kappa \text{ non-trivial}}} x_{\kappa}) + 2x_{1,n}.
\end{align}
In particular, the Poisson center for the deformed Poisson bivector $[\overline F_{\alpha_1 + \cdots + \alpha_n}, \pi]$ is a polynomial algebra.
\end{lem}

Note that the Casimir function $\psi$ in Proposition \ref{Poi-cen} is an anti-derivative of $(\sum \limits_{\substack {\kappa \vdash n \\ \kappa \text{ non-trivial}}} x_{\kappa}) + 2x_{1,n}$ with respect to the variable $x_{1,n}$.  It is interesting to study what the anti-derivatives of other functions in the Poisson center for $[\overline F_{\alpha_1 + \cdots + \alpha_n}, \pi]$ tell us about the Poisson center for $\pi$.  If we can determine the Poisson center for $[\overline F_{\alpha_i + \cdots + \alpha_j}, \pi]$ for some arbitrary $i,j$ with $1 \le i \le j \le n$, it is interesting to ask what we can say about $Z_{\text{Pois}}(P)$ using this information.

\subsubsection{Relation to the Open Bruhat Cell in the Flag Variety} \label{rel-open-Bru-cell-sect}

Recall that $G$ is a connected semi-simple algebraic group whose Lie algebra is $\mathfrak g$ and $B \subseteq G$ be the Borel subgroup of $G$ whose Lie algebra is $\mathfrak b$.  Using the so-called Bott-Samelson coordinates, the standard Poisson bracket on the open Bruhat cell $Bw_0B/B$ in the flag variety $G/B$ has been computed explicitly by Elek and Lu in \cite{EL}.  They have also shown that the coordinate ring of the open Bruhat cell has a structure of a symmetric Poisson CGL extension in the sense of Goodearl and Yakimov \cite{GY}.  In this section, we first show that $(P, \{~~,~~\})$ also has a structure of a symmetric Poisson CGL extension.  Then we establish a Poisson isomorphism from $(P,\{~~,~~\})$ to the coordinate ring of the open Bruhat cell equipped with the standard Poisson structure.

To see the symmetric Poisson CGL extension structure on $P$, we need a torus action on $P$.  Let $\mathfrak t$ be the abelian Lie algebra with basis $\{h_{\lambda}: \lambda \in \Phi^+\}$, where the $h_{\lambda}$'s are formal symbols.  For $\lambda, \mu \in \Phi^+$, we make $h_{\lambda}$ act on $\overline r'_{\mu}$ by multiplication by $(\lambda,\mu)$.  Write $T$ for the torus whose Lie algebra is $\mathfrak t$.  The $\mathfrak t$-action on $P$ integrates to a $T$-action on $P$.  The following lemma is obvious.

\begin{lem} \label{Poisson-action}
The action of $T$ on $P$ is Poisson.
\end{lem}

\begin{prop} \label{CGL}
With the action of $T$ on $P$ described above, $P$ has a structure of a symmetric Poisson CGL extension.
\end{prop}

\begin{proof}
(Sketch.)  Most of the proof is routine.  The only nontrivial part is that $P$ is an iterated Poisson-Ore extension.  For this one uses the formula in Proposition \ref{Pois-bra}, the Levendorskii-Soibelman straightening law (Theorem \ref{LS}) and the paragraph in \cite{GY} after Definition 2.9.
\end{proof}

\begin{cor}
$P$ has a structure of a cluster algebra.
\end{cor}

\begin{proof}
This follows from Propostion \ref{CGL} and the main theorem of \cite{GY}.
\end{proof}

Write $\kappa$ for the Killing form on $\mathfrak g$.  For $\lambda \in \Phi$, write $\mathfrak g_{\lambda}$ for the $\lambda$-root space of $\mathfrak g$.  Choose, for all $\lambda \in \Phi^+$, nonzero vectors $e_{\lambda} \in \mathfrak g_{\lambda}$ and $f_{\lambda} \in \mathfrak g_{-\lambda}$ such that $\kappa (e_{\lambda}, f_{\lambda}) = 1$.  Then we get an element
\begin{align}
R' := \frac 12 \sum \limits_{\lambda \in \Phi^+} e_{\lambda} \wedge f_{\lambda} \in \wedge^2 \mathfrak g.
\end{align}
$R'$ gives rise to a bivector $\pi_{st}$ on $G$ defined by
\begin{align}
\pi_{st} (g) := (L_g)_{\ast} (R') - (R_g)_{\ast} (R'),
\end{align}
for all $g \in G$, where $L_g$ (resp. $R_g$) stands for left (resp. right) multiplication by $g$.  It turns out that, c.f. \cite{BGY, EL, EvLu, EvLu2, LM, Mi, STS}, $\pi_{st}$ is a Poisson bivector on $G$.  $\pi_{st}$ is called the standard Poisson bivector on $G$.

Let $p: G \rightarrow G/B$ be the natural projection map.  It is known that, c.f. \cite{BGY, EL, EvLu, EvLu2, LM, Mi, STS}, there is a unique Poisson bivector on $G/B$, also denoted by $\pi_{st}$, making the map $p$ Poisson.  This Poisson bivector is called the standard Poisson bivector on $G/B$.

Example \ref{type-A2}, Proposition \ref{type-G2} and Proposition \ref{CGL} are consequences of the following

\begin{thm} \label{main-conj}
There exists a Poisson isomorphism between the open Bruhat cell $Bw_0B/B$ in the flag variety $G/B$ equipped with the standard Poisson structure and $(\text {Spec} P, \{~~,~~\})$.
\end{thm}

The proof of Theorem \ref{main-conj} is postponed to the end of this section.

In \cite{DP}, De Concini and Procesi have defined an $\mathcal A$ subalgebra $A^+$ of the quantized universal enveloping algebra $U$.  By definition, $A^+$ is the smallest $\mathcal A$ subalgebra of $U$ containing $(q_{\alpha} - q_{\alpha}^{-1}) E_{\alpha}$ which is stable under the action of the braid group.  Recall our notation in Theorem \ref{LS}.  Notice that, for $1 \le i \le j \le N$, we have
\begin{align} \label{formula-1}
& [(q_{\lambda_i} - q_{\lambda_i}^{-1}) E_{\lambda_i}, (q_{\lambda_j} - q_{\lambda_j}^{-1}) E_{\lambda_j}] \nonumber \\
= & (q^{(\lambda_i,\lambda_j)} - 1) ((q_{\lambda_j} - q_{\lambda_j}^{-1}) E_{\lambda_j}) ((q_{\lambda_i} - q_{\lambda_i}^{-1}) E_{\lambda_i}) \nonumber \\
+ & \sum c_{\vec d} \frac {(q_{\lambda_i} - q_{\lambda_i}^{-1}) (q_{\lambda_j} - q_{\lambda_j}^{-1})} {(q_{\lambda_1} - q_{\lambda_1}^{-1})^{d_1} \cdots (q_{\lambda_N} - q_{\lambda_N}^{-1})^{d_N}} ((q_{\lambda_N} - q_{\lambda_N}^{-1}) E_{\lambda_N})^{d_N} \cdots ((q_{\lambda_1} - q_{\lambda_1}^{-1}) E_{\lambda_1})^{d_1}.
\end{align}
By Proposition \ref{Dri-Kil-val} and Lemma \ref{com-rel}, we have
\begin{align} \label{formula-2}
[- r'_{\lambda_i}, - r'_{\lambda_j}] = & (q^{(\lambda_i,\lambda_j)} - 1) (- r'_{\lambda_j}) (- r'_{\lambda_i}) \nonumber \\
+ & \sum c_{\vec d} \frac {(q_{\lambda_i} - q_{\lambda_i}^{-1}) (q_{\lambda_j} - q_{\lambda_j}^{-1})} {(q_{\lambda_1} - q_{\lambda_1}^{-1})^{d_1} \cdots (q_{\lambda_N} - q_{\lambda_N}^{-1})^{d_N}} (- r'_{\lambda_N})^{d_N} \cdots (- r'_{\lambda_1})^{d_1}.
\end{align}

Write $A_{cl}^+$ for the $\mathbb C$-algebra $A^+/(1-q) A^+$.  The Hayashi construction equips $A_{cl}^+$ with a Poisson bracket.  We remark here that in \cite{DP}, when constructing the Poisson bracket for $\overline a$ and $\overline b$, where $a, b \in A$, De Concini and Procesi have used the formula
\begin{align}
\{\overline a, \overline b\} := (\frac 1 {q-1} [a,b])|_{q=1},
\end{align}
as opposed to our convention
\begin{align}
\{\overline a, \overline b\} := (\frac 1 {1-q} [a,b])|_{q=1}.
\end{align}
Define a $\mathbb C$-algebra map $P \rightarrow A_{cl}^+$ by sending $\overline r'_{\lambda_i}$ to $\overline {- (q_{\lambda_i} - q_{\lambda_i}^{-1}) E_{\lambda_i}}$, where $\overline {- (q_{\lambda_i} - q_{\lambda_i}^{-1}) E_{\lambda_i}}$ stands for the image of $- (q_{\lambda_i} - q_{\lambda_i}^{-1}) E_{\lambda_i}$ in $A_{cl}^+$.

\begin{prop} \label{Poisson-isomorphism-for-main-theorem}
The map $P \rightarrow A_{cl}^+$ is an anti-Poisson isomorphism of Poisson algebras.
\end{prop}

\begin{proof}
That the map $P \rightarrow A_{cl}^+$ is an algebra isomorphism follows from our PBW theorem for $P$ (Proposition \ref{PBW-qc} or Proposition \ref{a-presentation}) and the PBW basis for $A_{cl}^+$ constructed in \cite{DP}.

That the map $P \rightarrow A_{cl}^+$ is anti-Poisson follows from formulas (\ref{formula-1}), (\ref{formula-2}) and the Hayashi construction.
\end{proof}

Let $B^- \subseteq G$ be the Borel subgroup of $G$ opposite to $B$ and define $H := B \cap B^-$.  Write $N^+$ (resp. $N^-$) for the unipotent radical of $B$ (resp. $B^-$).  Equip $G$ with the standard Poisson structure.  So the Poisson dual group $G^{\ast}$ of $G$ can be concretely realized as
\begin{align}
G^{\ast} = \{(t u_-, t^{-1} u): t \in H, u \in N^+, u_- \in N^-\}.
\end{align}
With this presentation of $G^{\ast}$, it is easy to see that $N^-$ can be identified with a subgroup of $G^{\ast}$ via the injective map $u_- \mapsto (u_-, e)$.  It is clear that $N^-$ is in fact a normal subgroup of $G^{\ast}$ and the quotient group can be identified with $B$.  So we get a short exact sequence of algebraic groups
\begin{align} \label{exact-sequence}
1 \rightarrow N^- \rightarrow G^{\ast} \rightarrow B \rightarrow 1,
\end{align}
where the map $G^{\ast} \rightarrow B$ is given by $(t u_-, t^{-1} u) \mapsto t^{-1} u$.

Recall that $N^-$ is an almost Poisson subgroup of $G^{\ast}$ in the sense of, for example, \cite{BGY}.  This in particular implies that $B \simeq G^{\ast}/N^-$ has a unique Poisson structure such that the natural quotient map $G^{\ast} \rightarrow G^{\ast}/N^- \simeq B$ is Poisson.  We write $\pi_{quot}$ for the unique Poisson structure and call it the quotient Poisson structure on $B$.  Recall also that $B$ is a Poisson subgroup of $G$.  So $B$ has a Poisson structure such that the natural inclusion map $B \hookrightarrow G$ is Poisson.  Let us call this Poisson structure on $B$ the standard Poisson structure, and write $\pi_{st}$ for it.  To compare the two Poisson structures on $B$, we have

\begin{lem} \label{equivalence-of-Poisson-structures}
The identity map of $B$ is an anti-Poisson isomorphism from $(B, \pi_{st})$ to $(B, \pi_{quot})$.
\end{lem}

\begin{proof}
Let $D := G \times G$ be the Drinfeld double of $G$.  So the Lie algebra of $D$ is the direct sum $\mathfrak d := \mathfrak g \oplus \mathfrak g$.  Recall the notations we used in the definition of the standard Poisson structure $\pi_{st}$ on $G$.  Choose a basis $h_1, \cdots, h_r$ for the Cartan subalgebra $\mathfrak h$ of $\mathfrak g$ such that
\begin{align}
2 \kappa (h_i, h_j) = \delta_{i,j},
\end{align}
for all $1 \le i,j \le r$, where $\delta_{i,j}$ is the Kronecker delta.

Notice that the Poisson dual group $G^{\ast}$ is a subgroup of $D$.  Its Lie algebra $\mathfrak g^{\ast}$ can thus be realized as the following Lie subalgebra of $\mathfrak d$:
\begin{align}
\mathfrak g^{\ast} = \{(h + n_-, - h + n): h \in \mathfrak h, n \in \mathfrak n, n_- \in \mathfrak n^-\}.
\end{align}
Write $\mathfrak g_{\Delta}$ for the diagonal $\{(x,x) \in \mathfrak d: x \in \mathfrak g\}$ of $\mathfrak d$.  Then we have a basis
\begin{align}
\{(e_{\lambda}, e_{\lambda}), (f_{\lambda}, f_{\lambda}), (h_i, h_i): \lambda \in \Phi^+, 1 \le i \le r\}
\end{align}
for $\mathfrak g_{\Delta}$ and a basis
\begin{align}
\{(f_{\lambda}, 0), (0, - e_{\lambda}), (h_i, - h_i): \lambda \in \Phi^+, 1 \le i \le r\}
\end{align}
for $\mathfrak g^{\ast}$.

Let
\begin{align}
R := \frac 12 \sum ((f_{\lambda}, 0) \wedge (e_{\lambda}, e_{\lambda}) + (0, - e_{\lambda}) \wedge (f_{\lambda}, f_{\lambda})) + \frac 12 \sum (h_i, - h_i) \wedge (h_i, h_i)
\end{align}
be an element in $\wedge^2 \mathfrak d$.  For $d \in D$, let $R_d: D \rightarrow D$ (resp. $L_d: D \rightarrow D$) be the map sending $d'$ to $d'd$ (resp. $dd'$).  It is well-known, c.f. \cite{EL, EvLu, EvLu2, LM, STS}, that the bivector
\begin{align}
\pi_D (d) := (R_d)_{\ast} R - (L_d)_{\ast} R, ~~ d \in D
\end{align}
is a Poisson bivector on $D$ and the inclusion $G^{\ast} \hookrightarrow D$ is anti-Poisson.

Let $\text{pr}_2: D \rightarrow G$ be the second projection map.  We prove that $(\text{pr}_2)_{\ast} \pi_D$ is a well-defined Poisson bivector.  For this we fix $g \in G$ and compute the value of $(\text{pr}_2)_{\ast} \pi_D$ at $g$.  Let $g' \in G$ and write $d$ for $(g',g)$.  Then we have
\begin{align}
(\text{pr}_2)_{\ast} (\pi_D (d)) = & (\text{pr}_2)_{\ast} (R_d)_{\ast} R - (\text{pr}_2)_{\ast} (L_d)_{\ast} R \nonumber \\
= & (R_g)_{\ast} (\text{pr}_2)_{\ast} R - (L_g)_{\ast} (\text{pr}_2)_{\ast} R,
\end{align}
since $\text{pr}_2 \circ R_d = R_g \circ \text{pr}_2$ and $\text{pr}_2 \circ L_d = L_g \circ \text{pr}_2$.  Here $R_g: G \rightarrow G$ (resp. $L_g: G \rightarrow G$) stands for the map sending $g'$ to $g'g$ (resp. $gg'$).  Recall that $R' = \frac 12 \sum e_{\lambda} \wedge f_{\lambda} \in \wedge^2 \mathfrak g$, so that $(\text{pr}_2)_{\ast} R = - R'$.  Hence we have
\begin{align}
(\text{pr}_2)_{\ast} (\pi_D (d)) = - (R_g)_{\ast} R' + (L_g)_{\ast} R'.
\end{align}
From this it is clear that $(\text{pr}_2)_{\ast} (\pi_D (d))$ is independent of the choice of $g'$, hence $(\text{pr}_2)_{\ast} \pi_D$ is well-defined.  Recall that the standard Poisson bivector $\pi_{st}$ on $G$ is given by
\begin{align}
\pi_{st} (g) = (L_g)_{\ast} R' - (R_g)_{\ast} R', ~~ g \in G.
\end{align}
So we conclude that $(\text{pr}_2)_{\ast} \pi_D = \pi_{st}$.  Therefore, $(\text{pr}_2)_{\ast} \pi_D$ is a well-defined Poisson bivector, and is equal to $\pi_{st}$.

Consider the composition
\begin{align}
G^{\ast} \hookrightarrow D \xrightarrow {\text{pr}_2} G.
\end{align}
It is anti-Poisson because it is the composition of an anti-Poisson map with a Poisson map.  For $(t u_-, t^{-1} u) \in G^{\ast}$, the composition sends $(t u_-, t^{-1} u) \in G^{\ast}$ to $t^{-1} u$.  Hence we see that the image of the composition is the Borel subgroup $B$ of $G$ and the composition, viewed as a map from $G^{\ast}$ to $B$, is nothing but the natural quotient map from $G^{\ast}$ to $B$ in our exact sequence (\ref{exact-sequence}).  Consequently, the composition $G^{\ast} \rightarrow B$ is Poisson for the quotient Poisson structure on $B$ and anti-Poisson for the standard Poisson structure on $B$.  Therefore, the quotient and standard Poisson structures on $B$ differ only by a minus sign.
\end{proof}

\begin{proof}[Proof of Theorem \ref{main-conj}]
For this proof we will use notations in \cite{DP} freely.

Recall the De Concini and Procesi have constructed in \cite{DP} a Poisson algebra $Z_0$ over $\mathbb C$ and proved that it is Poisson isomorphic to the coordinate ring $\mathbb C[G^{\ast}]$ of the Poisson dual group $G^{\ast}$.  They have also constructed a $\mathbb C$-vector space isomorphism $Z_0 \simeq Z_0^+ \otimes Z_0^0 \otimes Z_0^-$.  The subgroup $N^-$ of $G^{\ast}$ acts on $G^{\ast}$ by multiplication on the right.  This action induces an action of $N^-$ on $Z_0$ by algebra isomorphisms via the isomorphism $Z_0 \simeq \mathbb C[G^{\ast}]$.  Analyzing the construction of De Concini and Procesi, one can prove that the $N^-$-invariant part of $Z_0$ is $Z_0^0 \otimes Z_0^-$.  De Concini and Procesi have also shown that $Z_0^0 \otimes Z_0^-$ is a Poisson subalgebra of $Z_0$.  Hence, by the exact sequence (\ref{exact-sequence}), we have a Poisson isomorphism
\begin{align}
\mathbb C[B] \simeq \mathbb C[G^{\ast}]^{N^-} \simeq Z_0^0 \otimes Z_0^-,
\end{align}
where $B$ is equipped with the quotient Poisson structure $\pi_{quot}$.

Now let us use notations in De Concini and Procesi \cite{DP}.  In particular, $J$ stands for the reduced expression (\ref{red-exp}) for $w_0$.  Also, there is a bijection $\bar \cdot: \Pi \rightarrow \Pi$ defined by $\bar {\alpha} = - w_0 (\alpha)$.  Then
\begin{align}
w_0 = s_{\bar {\alpha}_{i_1}} \cdots s_{\bar {\alpha}_{i_N}}
\end{align}
is also a reduced expression for $w_0$.  This reduced expression will be denoted by $\bar J$.  Let the $x$, $y$ and $z$'s be defined as in \cite{DP}.  For $1 \le k \le N$, one verifies that
\begin{align}
x_k^J = - \overline {K_{\lambda_k}^{-1} (q_{\lambda_k} - q_{\lambda_k}^{-1}) E_{\lambda_k}},
\end{align}
where $\overline {K_{\lambda_k}^{-1} (q_{\lambda_k} - q_{\lambda_k}^{-1}) E_{\lambda_k}}$ stands for the image of $K_{\lambda_k}^{-1} (q_{\lambda_k} - q_{\lambda_k}^{-1}) E_{\lambda_k}$ after quotienting out by the ideal generated by $(1-q)$.  Comparing with Proposition \ref{Poisson-isomorphism-for-main-theorem}, we see that $\overline r'_k$ maps to $z_{\lambda_k} x_k^J$ under the anti-Poisson isomorphism in Proposition \ref{Poisson-isomorphism-for-main-theorem}, so that $P$ is anti-Poisson isomorphic to the Poisson subalgebra of $Z_0^0 \otimes Z_0^-$ generated by $z_{\lambda_k} x_k^J$ for all $1 \le k \le N$.  

Recall that $W = N_G(H)/H$, where $N_G(H)$ is the normalizer of $H$ in $G$.  Let $\dot{w_0}$ be a lift of $w_0 \in W = N_G(H)/H$ to $N_G(H) \subseteq G$.  It is not hard to show that the map $(B, \pi_{st}) \rightarrow (Bw_0B/B, \pi_{st})$ defined by $b \mapsto b \dot{w_0} B/B$ is Poisson.  Thus the composition $\text{Spec} Z_0^0 \otimes Z_0^- \rightarrow B \rightarrow Bw_0B/B$ is anti-Poisson, where $B$ ane $Bw_0B/B$ are equipped with the standard Poisson structures.  The isomorphism $\text{Spec} Z_0^0 \otimes Z_0^- \simeq \text{Spec} Z_0^0 \times \text{Spec} Z_0^-$ of varieties allows us to decompose a point $p \in \text{Spec} Z_0^0 \otimes Z_0^-$ as $p = (p_0, p_-)$, where $p_0 \in \text{Spec} Z_0^0$ and $p_- \in \text{Spec} Z_0^-$.  Analyzing the argument of De Concini and Procesi in \cite{DP}, we see that the composition above sends $p \in \text{Spec} Z_0^0 \otimes Z_0^-$ to $(Z(p_0) X(p_-) Z^{-1}(p_0)) w_0 B/B$, where the maps $X: \text{Spec} Z_0^- \rightarrow N$ and $Z: \text{Spec} Z_0^0 \rightarrow H$ are defined in \cite{DP}.

Again, for the meaning of all the undefined notations in this paragraph, the readers are referred to \cite{DP}.  Define $X^J_k := T_{w_0} (Y^{\bar J}_k)$ for $1 \le k \le N$.  Using facts in \cite{DP}, we have
\begin{align}
X^J_k = & T_{w_0} (Y^{\bar J}_k) \nonumber \\
= & T_{w_0} (\exp{(y_k^{\bar J} f_k^{\bar J})}) \nonumber \\
= & t_0 \exp {(T_{w_0}(y_k^{\bar J}) f_k^{\bar J})} t_0^{-1} \nonumber \\
= & t_0 \exp {(x_k^J f_k^{\bar J})} t_0^{-1} \nonumber \\
= & \exp{(-x_k^J e_k^J)},
\end{align}
where $\exp : \mathfrak g \rightarrow G$ is the exponential map.  Notice that $X^J = T_{w_0} (Y^{\bar J}) = T_{w_0} (Y^{\bar J}_N \cdots Y^{\bar J}_1) = T_{w_0} (Y^{\bar J}_N) \cdots T_{w_0} (Y^{\bar J}_1) = X^J_N \cdots X^J_1$.  Hence, for $p = (p_0,p_-) \in \text{Spec} Z_0^0 \otimes Z_0^-$, we have
\begin{align}
Z(p_0) X(p_-) Z^{-1}(p_0) = & (Z(p_0) X^J_N(p_-) Z^{-1}(p_0)) \cdots (Z(p_0) X^J_1(p_-)  Z^{-1}(p_0)) \nonumber \\
= & \exp{(- z_{\lambda_N}(p_0) x^J_N(p_-) e^J_N)} \cdots \exp{(- z_{\lambda_1}(p_0) x^J_1(p_-) e^J_1)}.
\end{align}

Recall that the map $N^+ \rightarrow Bw_0B/B$ sending $u$ to $u \dot{w_0} B/B$ is an isomorphism of varieties.  Recall also that, for $1 \le k \le N$, we have a coordinate function $a_k: N^+ \rightarrow \mathbb C$ on $N^+$ sending $\exp{(t_{\lambda_N} e^J_N)} \cdots \exp{(t_{\lambda_1} e^J_1)}$ in $N^+$ to $t_{\lambda_k}$.  Here, we have used the fact that the map $\mathbb A^N \rightarrow N^+$ sending $(t_{\lambda_1}, \cdots, t_{\lambda_N})$ to $\exp{(t_{\lambda_N} e^J_N)} \cdots \exp{(t_{\lambda_1} e^J_1)}$ is an isomorphism of varieties.  Our arguments above imply that, for all $1 \le k \le N$, under the composition
\begin{align}
\text{Spec} Z_0^0 \otimes Z_0^- \rightarrow (B, \pi_{st}) \rightarrow (Bw_0B/B, \pi_{st}),
\end{align}
the coordinate function $a_k$ pulls back to $- z_{\lambda_k} x^J_k$, which is the image of $- \overline r'_k$ under the anti-Poisson isomorphism in Proposition \ref{Poisson-isomorphism-for-main-theorem}.  Hence, the injection $\mathbb C[Bw_0B/B] \rightarrow Z_0^0 \otimes Z_0^-$ sends $\mathbb C[Bw_0B/B]$, where $Bw_0B/B$ is equipped with the standard Poisson structure, Poisson isomorphically to $P$.
\end{proof}

\begin{rems}
1. A consequence of Theorem \ref{main-conj} is that $C^+_{\mathcal A}$ is a quantization of the open Bruhat cell $Bw_0B/B$ in the flag variety $G/B$ equipped with the standard Poisson structure.  Hence, Lemma \ref{com-rel} can be viewed as a quantization of the explicit formulas for the standard Poisson bracket on the open Bruhat cell $Bw_0B/B$ in \cite{EL}.  This answers the question of Elek an Lu in \cite{EL} about how to quantize their formulas.

2. An interesting question is to compare $C^+_{\mathcal A}$ with the `preferred' quantization of the coordinate ring of $Bw_0B/B$ constructed by Mi in \cite{Mi}.  Our quantum algebra $C^+_{\mathcal A}$ has the advantage that it is very explicit, at least representation-theoretically, while Mi's approach is very general but abstract.  The hope is that this comparison will provide a concrete understanding of Mi's abstract approach towards quantizing Poisson CGL extensions.  Computation in type $G_2$ suggests that there is no obvious $\mathcal A$-algebra isomorphism between $C_{\mathcal A}^+$ and Mi's `preferred' quantization of the coordinate ring of $Bw_0B/B$ equipped with the standard Poisson structure.

3. Another interesting consequence of Theorem \ref{main-conj} is that the Poisson brackets constructed in Theorem \ref{comp-gen} are now compatible with the standard Poisson bracket on the open Bruhat cell.  It would be very interesting to interpret these Poisson brackets, or the vector fields $\overline F_{\alpha_i + \cdots + \alpha_j}$, for all $1 \le i \le j \le n$, in terms of the geometry of the flag variety and explain compactibility of $\pi$ with $[\overline F_{\alpha_i + \cdots + \alpha_j}, \pi]$ using this interpretation.  We point out that the vector fields $\overline F_{\alpha_i + \cdots + \alpha_j}$, for $1 \le i \le j \le n$, are in general not Hamiltonian.

4. Recall that $B$ is a Poisson subgroup of $G$ and the induced Poisson structure on $B$ is referred to as the standard one.  The action of $B$ on $G$ by left multiplication induces an action of $B$ on the open Brahut cell $Bw_0B/B$, making it a $B$-homogeneous Poisson space, where $B$ is equipped with the standard Poisson structure.  Hence, by Theorem \ref{main-conj}, $\text{Spec} P$ can then be identified with a $B$-orbit in the variety $\mathcal L$ of Lagrangian subalgebras of the Drinfeld double of the Lie bialgebra $\mathfrak b$ corresponding to the standard Poisson structure on $B$ (see \cite{EvLu, EvLu2}).  Note that if we equip $B$ with the zero Poisson bracket, and make $B$ act on $\mathfrak n^{\ast}$ by the coadjoint action, then each $B$-orbit in $\mathfrak n^{\ast}$ is a $B$-homogeneous Poisson space for the Kirillov-Kostant Poisson bracket on $\mathfrak n^{\ast}$.  Hence, each $B$-orbit in $\mathfrak n^{\ast}$ embeds into the variety $\mathcal L'$ of Lagrangian subalgebras of the Drinfeld double of the Lie bialgebra $\mathfrak b$ corresponding to the zero Poisson structure on $B$.  It is interesting to study what our compatibility result Theorem \ref{KK} tells us about the geometry of $\mathcal L$ and $\mathcal L'$.
\end{rems}

\begin{proof}[Proof of Theorem \ref{gen-rk}]
By Theorem \ref{main-conj}, $\text{gr}_{\mathfrak g}$ equals the maximum of the dimension of the symplectic leaves in $Bw_0B/B$ equipped with the standard Poisson structure.  By Example 4.9 of \cite{EvLu2}, the latter quantity equals the dimension of $\mathfrak h^{-w_0}$, where $\mathfrak h^{-w_0}$ stands for the vector subspace of $\mathfrak h$ fixed by the action of $-w_0$.  By Proposition 1.10 of \cite{Kos}, $\text{gr}_{\mathfrak g}'$ equals the same number.
\end{proof}

\begin{rem}
This argument is largely due to S. Evens.  The author would like to thank him for generously sharing this proof.
\end{rem}

\section{Appendix: Proof of Theorem \ref{LS-str}}

In this section we prove our stronger version of the Levendorskii-Soibelman straightening law (Theorem \ref{LS-str}).  We prove a version slightly different than Theorem \ref{LS-str}.  It is not hard to check that Theorem \ref{LS-str} follows from this slightly different version.  Also, we will only prove the statement in Theorem \ref{LS-str} involving the $F$'s.  The proof of the statement involving the $E$'s is similar.  Our proof is adapted from the proof of Theorem \ref{LS} in \cite{DP}.

For this section, when we write $F^{\vec d}$, for $\vec d \in (\mathbb Z_{\ge 0})^N$, we mean $F_1^{d_1} \cdots F_N^{d_N}$.  It is known that $\{F^{\vec d}: \vec d \in (\mathbb Z_{\ge 0})^N\}$ is also a $\mathbb C(q)$-basis for $U^-$, c.f. \cite{DP, Jan}.  The advantage of this basis is that we do not have to reverse the natural order of the root vectors $F_1, \cdots F_N$.  This makes the computation slightly easier.

We shall prove

\begin{thm} \label{LS-var}
For $1 \le i < j \le N$, we have
\begin{align}
F_{\lambda_i} F_{\lambda_j} - q^{(\lambda_i, \lambda_j)} F_{\lambda_j} F_{\lambda_i} = \sum_{\vec d \in (\mathbb Z_{\ge 0})^N} c_{\vec d} F^{\vec d},
\end{align}
where $c_{\vec d} \in \mathcal A$ and $c_{\vec d} = 0$ whenever $d_k \neq 0$ for some $k \in [1,i] \cup [j,N]$.  Moreover, if $c_{\vec d} \neq 0$, then it is divisible by $(1-q)^{(\sum \limits _{k=1}^N d_k) - 1}$ in $\mathcal A$.
\end{thm}

\begin{proof}
We first note that when the root system $\Phi$ is of rank $2$, the statement can be verified by a straightforward computation.  In particular, we may assume that in the root system $\Phi$, there is no irreducible component of type $G_2$.

We prove Theorem \ref{LS-var} by induction on $j-i$.  Without loss of generality, we may assume that $i = 1$ and $j = r$ for some $1 < r \le N$.  The base case $r = 2$ can be verified by the same computation in a root system of rank $2$ as in the previous paragraph.  Assume that $r > 2$ and Theorem \ref{LS-var} has been verified for $j = 2, 3, \cdots, r-1$.  We have a number of cases.

\underline {Case 1.}  The roots $\alpha_{i_{r-1}}$ and $\alpha_{i_r}$ generate a root system of type $A_1 \times A_1$.

It follows from the assumption that
\begin{align}
\lambda_r = s_{\alpha_{i_1}} \cdots s_{\alpha_{i_{r-2}}} s_{\alpha_{i_{r-1}}} (\alpha_{i_r}) = s_{\alpha_{i_1}} \cdots s_{\alpha_{i_{r-2}}} (\alpha_{i_r})
\end{align}
and
\begin{align}
F_r = T_{\alpha_{i_1}} \cdots T_{\alpha_{i_{r-2}}} T_{\alpha_{i_{r-1}}} (F_{\alpha_{i_r}}) =  T_{\alpha_{i_1}} \cdots T_{\alpha_{i_{r-2}}} (F_{\alpha_{i_r}}).
\end{align}
Also, the assumption implies that
\begin{align}
s_{\alpha_{i_1}} \cdots s_{\alpha_{i_{r-2}}} s_{\alpha_{i_r}}
\end{align}
is a reduced expression.  So this reduced expression can be completed to a reduced expression for $w_0$.  Write $F'_1, \cdots, F'_N$ for the root vectors defined by this new reduced expression for $w_0$.  They correspond to positive roots $\lambda'_1, \cdots, \lambda'_N$.  Then our computation above tells us that $F_k = F'_k$ for $k = 1, \cdots, r-2$, and $F_r = F'_{r-1}$.  By induction hypothesis, we get
\begin{align} \label{aaaaa}
F_1 F_r - q^{(\lambda_1,\lambda_r)} F_r F_1 = F'_1 F'_{r-1} - q^{(\lambda'_1,\lambda'_{r-1})} F'_{r-1} F'_1 = \sum c_{\vec d} (F')^{\vec d},
\end{align}
where $c_{\vec d}$ has the properties stated in Theorem \ref{LS-var}.  Since $F_2 = F'_2, \cdots, F_{r-2} = F'_{r-2}$, $(F')^{\vec d} = F^{\vec d}$ for all $\vec d \in (\mathbb Z_{\ge 0})^N$ such that the coefficient $c_{\vec d}$ in formula (\ref{aaaaa}) is nonzero.  It follows that $F_1 F_r - q^{(\lambda_1,\lambda_r)} F_r F_1 = \sum c_{\vec d} F^{\vec d}$, where $c_{\vec d}$ has the properties claimed in Theorem \ref{LS-var}.

\underline {Case 2.}  The roots $\alpha_{i_{r-1}}$ and $\alpha_{i_r}$ generate a root system of type $A_2$ and $s_{\alpha_{i_1}} \cdots s_{\alpha_{i_{r-2}}} s_{\alpha_{i_r}}$ is a reduced expression.

Write $u$ for the element $s_{\alpha_{i_1}} \cdots s_{\alpha_{i_{r-2}}}$ in the Weyl group $W$.  By abuse of notation, when we write $u$, we also mean the reduced expression $s_{\alpha_{i_1}} \cdots s_{\alpha_{i_{r-2}}}$.  The assumptions of this case imply that
\begin{align}
u s_{\alpha_{i_r}}
\end{align}
is a reduced expression.  So it can be completed to a reduced expression for$w_0$.  Write $F'_1, \cdots, F'_N$ for the root vectors defined by this new reduced expression for $w_0$.  They correspond to positive roots $\lambda'_1, \cdots, \lambda'_N$.  It is easy to see that $F_k = F'_k$ for $k = 1, \cdots, r-2$ and $T_u (F_{\alpha_{i_r}}) = F'_{r-1}$.  By induction hypothesis, we have
\begin{align} \label{bbbbb}
F_1 (T_u (F_{\alpha_{i_r}})) - q^{(\alpha_{i_1}, u(\alpha_{i_r}))} (T_u (F_{\alpha_{i_r}})) F_1 = F'_1 F'_{r-1} - q^{(\alpha_{i_1}, u(\alpha_{i_r}))} F'_{r-1} F'_1 = \sum c_{\vec d} (F')^{\vec d},
\end{align}
where $c_{\vec d}$ has the properties claimed in Theorem \ref{LS-var}.  Note that $(F')^{\vec d} = F^{\vec d}$ for all $\vec d \in (\mathbb Z_{\ge 0})^N$ such that the coefficient $c_{\vec d}$ in formula (\ref{bbbbb}) is nonzero.  We have $F_1 (T_u (F_{\alpha_{i_r}})) - q^{(\alpha_{i_1}, u(\alpha_{i_r}))} (T_u (F_{\alpha_{i_r}})) F_1 = \sum c_{\vec d} F^{\vec d}$, where $c_{\vec d}$ has the properties claimed in Theorem \ref{LS-var}.  For notational simplicity we write $S'_{(1,r-1)}$ for $F_1 (T_u (F_{\alpha_{i_r}})) - q^{(\alpha_{i_1}, u(\alpha_{i_r}))} (T_u (F_{\alpha_{i_r}})) F_1$.  Similarly, our induction hypothesis also implies that
\begin{align}
F_1 F_{r-1} - q^{(\lambda_1,\lambda_{r-1})} F_{r-1} F_1 = \sum \tilde c_{\vec d} F^{\vec d}, 
\end{align}
where $\tilde c_{\vec d}$ has the properties claimed in Theorem \ref{LS-var}.  Again for notational simplicity, we write $S_{(1,r-1)}$ for $F_1 F_{r-1} - q^{(\lambda_1,\lambda_{r-1})} F_{r-1} F_1$.

Since
\begin{align}
F_r = & T_u T_{\alpha_{i_{r-1}}} (F_{\alpha_{i_r}}) \nonumber \\
= & T_u (F_{\alpha_{i_r}} F_{\alpha_{i_{r-1}}} - q F_{\alpha_{i_{r-1}}} F_{\alpha_{i_r}} \nonumber \\
= & T_u (F_{\alpha_{i_r}}) T_u (F_{\alpha_{i_{r-1}}}) - q T_u (F_{\alpha_{i_{r-1}}}) T_u (F_{\alpha_{i_r}}) \nonumber \\
= & T_u (F_{\alpha_{i_r}}) F_{r-1} - q F_{r-1} T_u (F_{\alpha_{i_r}}),
\end{align}
it follows that
\begin{align}
F_r F_1 = & T_u (F_{\alpha_{i_r}}) F_{r-1} F_1 - q F_{r-1} T_u (F_{\alpha_{i_r}}) F_1 \nonumber \\
= & T_u (F_{\alpha_{i_r}}) (q^{-(\lambda_1,\lambda_{r-1})} F_1 F_{r-1} - q^{-(\lambda_1,\lambda_{r-1})} S_{(1,r-1)}) \nonumber \\
& - q F_{r-1} (q^{-(\lambda_1,u(\alpha_{i_r}))} F_1 T_u (F_{\alpha_{i_r}}) - q^{-(\lambda_1,u(\alpha_{i_r}))} S'_{(1,r-1)}) \nonumber \\
= & q^{- (\lambda_1,\lambda_{r-1}) - (\lambda_1, u(\alpha_{i_r}))} (F_1 T_u (F_{\alpha_{i_r}}) - S'_{(1,r-1)}) F_{r-1} - q^{-(\lambda_1,\lambda_{r-1})} T_u (F_{\alpha_{i_r}}) S_{(1,r-1)} \nonumber \\
& - q^{1 - (\lambda_1, u(\alpha_{i_r})) - (\lambda_1, \lambda_{r-1})} (F_1 F_{r-1} - S_{(1,r-1)}) T_u (F_{\alpha_{i_r}}) + q^{1 - (\lambda_1, u(\alpha_{i_r}))} F_{r-1} S'_{(1,r-1)} \nonumber \\
= & q^{- (\lambda_1,\lambda_{r-1} + u(\alpha_{i_r}))} F_1 F_r - q^{- (\lambda_1,\lambda_{r-1} + u(\alpha_{i_r}))} (S'_{(1,r-1)} F_{r-1} - q^{1 + (\lambda_1,u(\alpha_{i_r}))} F_{r-1} S'_{(1,r-1)}) \nonumber \\
& + q^{1 - (\lambda_1,\lambda_{r-1} + u(\alpha_{i_r}))} (S_{(1,r-1)} T_u (F_{\alpha_{i_r}}) - q^{- 1 + (\lambda_1,u(\alpha_{i_r}))} T_u (F_{\alpha_{i_r}}) S_{(1,r-1)}).
\end{align}

We analyze $S'_{(1,r-1)} F_{r-1} - q^{1 + (\lambda_1,u(\alpha_{i_r}))} F_{r-1} S'_{(1,r-1)}$.  It follows from induction hypothesis and the definition of $S'_{(1,r-1)}$ that
\begin{align}
S'_{(1,r-1)} F_{r-1} - q^{1 + (\lambda_1 + \lambda_{r-1},u(\alpha_{i_r}))} F_{r-1} S'_{(1,r-1)} = \sum a_{\vec d} F^{\vec d},
\end{align}
where $a_{\vec d} \in \mathcal A$, $a_{\vec d} \neq 0$ only when $d_k = 0$ for all $k \le 2$ and $k \ge r-1$.  Moreover, if $a_{\vec d} \neq 0$, it is divisible by $(1-q)^{(\sum\limits_{k=1}^N d_k) - 1}$.  Hence, we have
\begin{align}
& S'_{(1,r-1)} F_{r-1} - q^{1 + (\lambda_1,u(\alpha_{i_r}))} F_{r-1} S'_{(1,r-1)} \nonumber \\
= & q^{- (\lambda_{r-1},u(\alpha_{i_r}))} (q^{(\lambda_{r-1},u(\alpha_{i_r}))} S'_{(1,r-1)} F_{r-1} - q^{1 + (\lambda_1 + \lambda_{r-1},u(\alpha_{i_r}))} F_{r-1} S'_{(1,r-1)}) \nonumber \\
= & q^{- (\lambda_{r-1},u(\alpha_{i_r}))} ((q^{(\lambda_{r-1},u(\alpha_{i_r}))} - 1) S'_{(1,r-1)} F_{r-1} + \sum a_{\vec d} F^{\vec d}).
\end{align}
Recall that each coefficient in $S'_{(1,r-1)}$ is divisible a power of $(1-q)$ as predicted by Theorem \ref{LS-var}.  Since $q^{(\lambda_{r-1},u(\alpha_{i_r}))} - 1$ is divisible by $(1-q)$, each coefficient in $S'_{(1,r-1)} F_{r-1}$ is also divisible by a power of $(1-q)$ as predicted by Theorem \ref{LS-var}.

The analysis of and conclusion about $S_{(1,r-1)} T_u (F_{\alpha_{i_r}}) - q^{- 1 + (\lambda_1,u(\alpha_{i_r}))} T_u (F_{\alpha_{i_r}}) S_{(1,r-1)}$ is exactly the same.  Note that
\begin{align}
(\lambda_1 + \lambda_{r-1}, u(\alpha_{i_r})) = & (\lambda_1 + u(\alpha_{i_{r-1}}), u(\alpha_{i_r})) \nonumber \\
= & (\lambda_1, u(\alpha_{i_r})) + (u(\alpha_{i_{r-1}}), u(\alpha_{i_r})) \nonumber \\
= & (\lambda_1, u(\alpha_{i_r})) + (\alpha_{i_{r-1}}, \alpha_{i_r}) \nonumber \\
= & (\lambda_1, u(\alpha_{i_r})) - 1.
\end{align}
So the summand $S_{(1,r-1)} T_u (F_{\alpha_{i_r}})$ will not appear in $S_{(1,r-1)} T_u (F_{\alpha_{i_r}}) - q^{- 1 + (\lambda_1,u(\alpha_{i_r}))} T_u (F_{\alpha_{i_r}}) S_{(1,r-1)}$.  In other words, $F'_{r-1}$ will not appear in $S_{(1,r-1)} T_u (F_{\alpha_{i_r}}) - q^{- 1 + (\lambda_1,u(\alpha_{i_r}))} T_u (F_{\alpha_{i_r}}) S_{(1,r-1)}$.

Notice that $(\lambda_1,\lambda_r) = (\lambda_1, u s_{\alpha_{i_{r-1}}} (\alpha_{i_r})) = (\lambda_1, \lambda_{r-1} + u (\alpha_{i_r}))$.  So, combining the formulas in the last three paragraphs, we are done in this case.

\underline {Case 3.}  The roots $\alpha_{i_{r-1}}$ and $\alpha_{i_r}$ generate a root system of type $A_2$ and $s_{\alpha_{i_1}} \cdots s_{\alpha_{i_{r-2}}} s_{\alpha_{i_r}}$ is not a reduced expression.

Let $u$ be the same as in Case 2.  The assumptions imply that there exists a reduced expression $v$ such that $v s_{\alpha_{i_r}}$ is a reduced expression for $u$.  By abuse of notation, we write $v$ also for the product of the simple reflections in the reduced expression $v$.  Observe that the length $l(v s_{\alpha_{i_{r-1}}})$ of $v s_{\alpha_{i_{r-1}}}$ is great than the length $l(v)$ of $v$.  In fact, if this is not the case, then
\begin{align}
l(u) + 2 = l(u s_{\alpha_{i_{r-1}}} s_{\alpha_{i_r}}) = l(v s_{\alpha_{i_r}} s_{\alpha_{i_{r-1}}} s_{\alpha_{i_r}}) = l(v s_{\alpha_{i_{r-1}}} s_{\alpha_{i_r}} s_{\alpha_{i_{r-1}}}) < l(v) + 2,
\end{align}
which contradicts the assumption that $s_{\alpha_{i_1}} \cdots s_{\alpha_{i_{r-2}}} s_{\alpha_{i_r}}$ is not a reduced expression.

By the exchange lemma (c.f. \cite[Section~1.7]{Hum}), there exists $k \in \{1, \cdots, r-2\}$ such that
\begin{align}
s_{\alpha_{i_1}} \cdots s_{\alpha_{i_{k-1}}} s_{\alpha_{i_{k+1}}} \cdots s_{\alpha_{i_{r-2}}}
\end{align}
is a reduced expression for $v$ and
\begin{align}
s_{\alpha_{i_1}} \cdots s_{\alpha_{i_{k-1}}} s_{\alpha_{i_{k+1}}} \cdots s_{\alpha_{i_{r-2}}} s_{\alpha_{i_r}}
\end{align}
is a reduced expression for $u$.  Consequently, we must have
\begin{align}
F_r = T_u T_{\alpha_{i_{r-1}}} (F_{\alpha_{i_r}}) = T_v T_{\alpha_{i_r}} T_{\alpha_{i_{r-1}}} (F_{\alpha_{i_r}}) = T_v (F_{\alpha_{i_{r-1}}}).
\end{align}
Since $l(v s_{\alpha_{i_{r-1}}}) > l(v)$,
\begin{align}
s_{\alpha_{i_1}} \cdots s_{\alpha_{i_{k-1}}} s_{\alpha_{i_{k+1}}} \cdots s_{\alpha_{i_{r-2}}} s_{\alpha_{i_{r-1}}}
\end{align}
is a reduced expression of length $r-2$.  Hence, if $k$ is not equal to $1$, we can compute $F_1 F_r - q^{(\lambda_1,\lambda_r)} F_r F_1$ using this reduced expression and the induction hypothesis.  The rest of the argument is the same as in Case 2.  If $k$ is equal to $1$, then we have $\alpha_{i_1} = \alpha_{i_r}$.  It follows that $v s_{\alpha_{i_r}} = s_{\alpha_{i_r}} v$, i.e. $s_{\alpha_{i_r}} = v s_{\alpha_{i_r}} v^{-1} = s_{v(\alpha_{i_r})}$.  So the roots $\alpha_{i_r}$ and $v(\alpha_{i_r})$ are proportional.  Since $s_{\alpha_{i_2}} \cdots s_{\alpha_{i_{r-2}}} s_{\alpha_{i_r}}$ is a reduced expression, $v(\alpha_{i_r})$ is a positive root.  Hence, $v(\alpha_{i_r}) = \alpha_{i_r}$.  As a consequence, we have
\begin{align}
F_1 = F_{\alpha_{i_1}} = F_{\alpha_{i_r}} = T_v (F_{\alpha_{i_r}}).
\end{align}
It follows that
\begin{align}
F_1 F_r = T_v (F_{\alpha_{i_r}}) T_v (F_{\alpha_{i_{r-1}}}) = - q^{-1} T_v T_{\alpha_{i_r}} (F_{\alpha_{i_{r-1}}}) + q^{-1} F_r F_1 = - q^{-1} F_{r-1} + q^{-1} F_r F_1.
\end{align}
Notice that $(\lambda_1,\lambda_r) = (v(\alpha_{i_r}), v(\alpha_{i_{r-1}})) = (\alpha_{i_r}, \alpha_{i_{r-1}}) = -1$.  So, we conclude that
\begin{align}
F_1 F_r - q^{(\lambda_1,\lambda_r)} F_r F_1 = - q^{-1} F_{r-1}.
\end{align}
We are done for this case.

\underline {Case 4.}  The roots $\alpha_{i_{r-1}}$ and $\alpha_{i_r}$ generate a root system of type $B_2$, $\alpha_{i_{r-1}}$ is longer and
\begin{align}
s_{\alpha_{i_1}} \cdots s_{\alpha_{i_{r-2}}} s_{\alpha_{i_r}}
\end{align}
is a reduced expression.

Again let $u$ be the same as before.  Since $F_r = T_u T_{\alpha_{i_{r-1}}} (F_{\alpha_{i_r}}) = T_u (F_{\alpha_{i_r}}) F_{r-1} - q^2 F_{r-1} T_u (F_{\alpha_{i_r}})$ and
\begin{align}
u s_{\alpha_{i_r}} = s_{\alpha_{i_1}} \cdots s_{\alpha_{i_{r-2}}} s_{\alpha_{i_r}}
\end{align}
is a reduced expression, exactly the same argument as in Case 2 works in this case.  We omit the details.

\underline {Case 5.}  The roots $\alpha_{i_{r-1}}$ and $\alpha_{i_r}$ generate a root system of type $B_2$, $\alpha_{i_r}$ is longer and
\begin{align}
s_{\alpha_{i_1}} \cdots s_{\alpha_{i_{r-2}}} s_{\alpha_{i_r}}
\end{align}
is a reduced expression.

Let $u$ be the same as before.  The assumptions imply that $u s_{\alpha_{i_r}} s_{\alpha_{i_{r-1}}}$ is a reduced expression.  So it can be completed to a reduced expression for $w_0$.  Write $F'_1, \cdots, F'_N$ for the root vectors defined by this new reduced expression for $w_0$.  They correspond to positive roots $\lambda'_1, \cdots, \lambda'_N$.  Write $S_{(1,r-1)}$ for $F_1 F_{r-1} - q^{(\lambda_1,\lambda_{r-1})} F_{r-1} F_1$.  By induction hypothesis, the coefficients in $S_{(1,r-1)}$ have the properties claimed in Theorem \ref{LS-var}.  Similarly, define $S'_{(1,r-1)} := F'_1 F'_{r-1} - q^{(\lambda'_1,\lambda'_{r-1})} F'_{r-1} F'_1$.  Coefficients in $S'_{(1,r-1)}$ have the properties claimed in Theorem \ref{LS-var}.  By a lengthy calculation, we have
\begin{align}
F_r F_1 = & [- q^{- 2 - (\lambda_1,\lambda_{r-1}) - (\lambda_1,\lambda'_r)} F_1 F'_r F_{r-1} + q^{- 2 - (\lambda_1,\lambda_{r-1}) - (\lambda_1,\lambda'_r)} F_1 F_{r-1} F'_r \nonumber \\
& + (q^{-2} - q^2) q^{- (\lambda_1, 2 \lambda_{r-1} + \lambda'_{r-1})} F_1 F_{r-1} F'_{r-1} F_{r-1} - (q^{-2} - q^2) q^{- (\lambda_1, 2 \lambda_{r-1} + \lambda'_{r-1})} F_1 F^2_{r-1} F'_{r-1}] \nonumber \\
& + [-q^{-2} F'_r S_{(1,r-1)} + q^{-2-(\lambda_1, \lambda'_r)} S_{(1,r-1)} F'_r] \nonumber \\
& + [-q^{-2-(\lambda_1, \lambda_{r-1})} S'_{(1,r-1)} F_{r-1} + q^{-2} F_{r-1} S'_{(1,r-1)}] \nonumber \\
& + (q^{-2} - q^2) [q^{-(\lambda_1, \lambda_{r-1})} F_{r-1} S'_{(1,r-1)} F_{r-1} - F^2_{r-1} S'_{(1,r-1)}] \nonumber \\
& + (q^{-2} - q^2) [q^{-(\lambda_1,\lambda_{r-1} + \lambda'_{r-1})} S_{(1,r-1)} F'_{r-1} F_{r-1} + F_{r-1} F'_{r-1} S_{(1,r-1)} \nonumber \\
& - q^{-(\lambda_1, \lambda_{r-1} + \lambda'_{r-1})} S_{(1,r-1)} F_{r-1} F'_{r-1} - q^{-(\lambda_1,\lambda'_{r-1})} F_{r-1} S_{(1,r-1)} F'_{r-1}].
\end{align}
Each of the five summands above (note that the first two lines is one single summand, the last two lines is also one signle summand) can be analyzed in the same way as in Case 2.  The details will take too much space to be written down.  It is omitted.

\underline {Case 6.}  The roots $\alpha_{i_{r-1}}$ and $\alpha_{i_r}$ generate a root system of type $B_2$ and $s_{\alpha_{i_1}} \cdots s_{\alpha_{i_{r-2}}} s_{\alpha_{i_r}}$ is not a reduced expression.

Let $u$ be the same as before.  This case has two subcases.  The first subcase is where $l(u s_{\alpha_{i_r}} s_{\alpha_{i_{r-1}}}) < l(u s_{\alpha_{i_r}})$.  The second subcase is where $l(u s_{\alpha_{i_r}} s_{\alpha_{i_{r-1}}}) > l(u s_{\alpha_{i_r}})$.  In both subcases, the analysis is a combination of the analysis in Case 2 and Case 3.  More specifically, non-reducedness of the expression
\begin{align}
s_{\alpha_{i_1}} \cdots s_{\alpha_{i_{r-2}}} s_{\alpha_{i_r}}
\end{align}
is dealt with in the same way as in Case 3; counting the power of $(1-q)$ in each coefficient, as always, is done in the same way as in Case 2.  The details are omitted.
\end{proof}





\noindent Li, Yu: Department of Mathematics, University of Chicago, 5734 S. University Ave., Chicago, IL 60637.  Email: liyu@math.uchicago.edu

\end{document}